\documentclass[12pt,a4paper]{article}

\usepackage{latexsym}
\usepackage{amsmath}
\usepackage{amssymb}
\usepackage{amsthm}
\usepackage{amscd}
\usepackage{mathrsfs}
\usepackage[all]{xy}
\usepackage{graphicx}
\usepackage{comment}
\usepackage{arydshln}

\newtheorem{thm}{Theorem}[section]
\newtheorem{prop}[thm]{Proposition}
\newtheorem{defn}[thm]{Definition}
\newtheorem{lem}[thm]{Lemma}
\newtheorem{cor}[thm]{Corollary}

\newtheorem{rem}[thm]{Remark}

\theoremstyle{remark}

\makeatletter

\@addtoreset{equation}{section}
\makeatother

\makeatletter

\@addtoreset{equation}{section}
\makeatother

\makeatletter
\newcommand{\subsubsubsection}{\@startsection{paragraph}{4}{\z@}%
 {1.0\Cvs \@plus.5\Cdp \@minus.2\Cdp}%
 {.1\Cvs \@plus.3\Cdp}%
 {\reset@font\sffamily\normalsize}
 }
\makeatother
\DeclareMathOperator{\Spec}{Spec}

\DeclareMathOperator{\id}{id}

\DeclareMathOperator{\End}{End}
\DeclareMathOperator{\Hom}{Hom}
\DeclareMathOperator{\Ext}{Ext}

\DeclareMathOperator{\Aut}{Aut} 
 
\DeclareMathOperator{\Ind}{Ind}

\DeclareMathOperator{\Nrd}{Nrd}

\DeclareMathOperator{\Spa}{Spa}
\DeclareMathOperator{\Lie}{Lie}

\DeclareMathOperator{\IC}{IC}
\DeclareMathOperator{\Sht}{Sht}
\DeclareMathOperator{\Spd}{Spd}
\DeclareMathOperator{\Perf}{Perf}

\DeclareMathOperator{\GL}{GL}
\DeclareMathOperator{\Bun}{Bun}
\DeclareMathOperator{\lis}{lis}

\DeclareMathOperator{\ad}{ad}
\DeclareMathOperator{\Hck}{Hck}
\DeclareMathOperator{\Div}{Div}

\DeclareMathOperator{\AffPerf}{AffPerf}

\newcommand{\ab}{\mathrm{ab}}

\newcommand{\alg}{\mathrm{alg}}

\newcommand{\Gm}{\mathbb{G}_{\mathrm{m}}}

\newcommand{\Hc}{\operatorname{H}_{\mathrm{c}}}

\newcommand{\Eis}{\mathrm{Eis}}
\newcommand{\nmEis}{\mathrm{nEis}}
\newcommand{\ol}{\overline}
\newcommand{\ul}{\underline}
\newcommand{\wh}{\widehat}
\newcommand{\wt}{\widetilde}
\newcommand{\lra}{\longrightarrow}
\newcommand{\ra}{\rightarrow}

\newcommand{\bb}{\mathbb}

\newcommand{\ca}{\mathcal}

\newcommand{\scr}{\mathscr}

\def\basic{\mathrm{basic}}
\def\mf{\mathfrak}
\def\det{\mathrm{det}}
\def\nmEisP{\mathrm{nEis}_{P}}
\def\EisP{\mathrm{Eis}_{P}}
\def\Loc{\mathrm{Loc}}
\def\mcLie{\mathcal{L}\mathrm{ie}}
\def\Ad{\mathrm{Ad}}
\newcommand{\Lcoinv}{\mathbb{X}_{*}(L^{\mathrm{ab}}_{\ol{F}})_{\Gamma_{F}}}
\newcommand{\Lcoinvdom}{\mathbb{X}_{*}(L^{\mathrm{ab}}_{\ol{F}})_{\Gamma_{F}}^{+}}
\newcommand{\cf}{\textit{cf.\ }}
\usepackage{tikz-cd}

\begin{document}

\title
{Dualizing complexes on the moduli\\ of parabolic bundles} 

\author{Linus Hamann and Naoki Imai}
\date{}
\maketitle


\begin{abstract}
For a non-archimedean local field $F$ and a connected reductive group $G$ over $F$ equipped with a parabolic subgroup $P$, we show that the dualizing complex on $\Bun_{P}$, the moduli stack of $P$-bundles on the Fargues--Fontaine curve, can be described explicitly in terms of the modulus character of $P$. As applications, we identify various characters appearing in the theory of local and global Shimura varieties, show the Harris--Viehmann conjecture in the Hodge--Newton reducible case, and carry out some computations of the geometric Eisenstein functors for general parabolics.
\end{abstract}
\tableofcontents 
\section{Introduction}
Let $G/F$ be a connected reductive group over a non-archimedean field $F$. Let $p$ be the characteristic of the residue field,  $q$ denote its cardinality, and $k$ denote its algebraic closure. Let $\Lambda$ be a torsion ring of order prime to $p$ unless otherwise stated. We let $\Bun_{G}$ denote the $v$-stack parameterizing $G$-bundles on the relative Fargues--Fontaine curve $X_{S}$, for $S \in \Perf_k$, the category of perfectoid spaces over $k$. It is known \cite[Theorem~IV.1.19]{FaScGeomLLC} that this defines an $\ell$-cohomologically smooth Artin $v$-stack of pure $\ell$-dimension $0$. Concretely, this translates into the fact that, \'etale locally on $\Bun_{G}$, the dualizing complex $K_{\Bun_{G}} \in D(\Bun_{G})$ is isomorphic to the constant sheaf $\Lambda$ on $\Bun_{G}$. On the other hand, $\Bun_{G}$ admits a locally closed stratification into Harder--Narasimhan (abbv. HN) strata $i^{b}\colon \Bun_{G}^{b} \hookrightarrow \Bun_{G}$, indexed by elements $b \in B(G)$ in the Kottwitz set of $G$. These HN-strata are isomorphic to $[\ast/\wt{G}_{b}]$, where $\wt{G}_{b}$ is a group diamond admitting a surjection $\wt{G}_{b} \ra \ul{G_{b}(F)}$ with kernel given by an $\ell$-adically contractible unipotent part denoted $\wt{G}_{b}^{> 0}$. Pulling back along the induced map $[\ast/\wt{G}_{b}] \ra [\ast/\ul{G_{b}(F)}]$ induces an equivalence $D(\Bun_{G}^{b}) \cong D([\ast/\ul{G_{b}(F)}])$, and the RHS is in turn isomorphic to the unbounded (left-completed) derived category of smooth $\Lambda$-representations of $G_{b}(F)$ \cite[Theorem~V.0.1 (ii)-(iii)]{FaScGeomLLC}. It follows that one has a semi-orthogonal decomposition of any object $\mathcal{F} \in D(\Bun_{G})$ into complexes of $G_{b}(F)$-representations, for varying $b \in B(G)$.

It is a natural question to wonder what the $G_{b}(F)$-representations corresponding to the dualizing complex $K_{\Bun_{G}}$ under the semi-orthogonal decomposition are. In particular, any sheaf on the classifying stack $D([\ast/\ul{G_{b}(F)}])$ corresponding to a smooth character of $G_{b}(F)$ will be \'etale locally constant, but not constant on the nose. Our first result is as follows.
\begin{prop}[Proposition \ref{prop: dualcomplexforreductive}]\label{prop:intdualBunG}
For $G/F$ any connected reductive group, we have an isomorphism 
\[ K_{\Bun_{G}} \cong \Lambda \]
of objects in $D(\Bun_{G})$.
\end{prop}
\begin{rem}
This is claimed in a footnote in \cite[V.5]{FaScGeomLLC} indicating an argument, which is close to the one in Remark \ref{rem: WildProof}. 
\end{rem}
The idea for proving this is to note that, on the semistable locus $\Bun_{G}^{\mathrm{ss}}$, we have a decomposition 
\[ \Bun_{G}^{\mathrm{ss}} \cong \bigsqcup_{b \in B(G)_{\basic}} [\ast/\ul{G_{b}(F)}], \]
and the dualizing complex on such classifying stacks $[\ast/\ul{G_{b}(F)}]$ is given by a module of Haar measures on $G_{b}(F)$, which is trivial as a $G_{b}(F)$-representation since $G_{b}$ is unimodular. We can then extend to all of $\Bun_{G}$ using the following result, which is interesting in its own right. 
\begin{prop}[Proposition \ref{prop: Localsystemsextend}]\label{prop:intssequiv}
For $G/F$ a connected reductive group, we have a well-defined equivalence of categories 
\[ \widetilde{\Loc}(\Bun_{G}^{\mathrm{ss}}) \xrightarrow{\cong} \Loc(\Bun_{G}) \]
\[ A \mapsto R^{0}j_{G*}(A), \]
where $\Loc(X) \subset D(X)$ denotes the category of local systems on $X$ with respect to the \'etale topology, $\widetilde{\Loc}(\Bun_{G}^{\mathrm{ss}}) \subset \Loc(\Bun_{G}^{\mathrm{ss}})$ denotes the subcategory of \'etale local systems which admit a trivialization induced from an \'etale covering of $\Bun_{G}$, and $j_{G}: \Bun_{G}^{\mathrm{ss}} \hookrightarrow \Bun_{G}$ is the natural open immersion. 
\end{prop}
\begin{rem}
A proof of this result and the previous result through a different but less explicit method has been carried out for $\GL_{2}$ in the Masters Thesis of Ruth Wild \cite{WildBunGcohomology}. The idea of Wild generalizes and we sketch this in Remark \ref{rem: WildProof}. Our argument is very similar, but fits in more naturally with the other calculations we carry out.
\end{rem}
\begin{rem}
The restriction to the subcategory $\widetilde{\Loc}(\Bun_{G}^{\mathrm{ss}})$ implies (non-obviously) that $R^{0}j_{G*}(A)$ is also an \'etale local system so that the map is well-defined. In general, the inclusion $\widetilde{\Loc}(\Bun_{G}^{\mathrm{ss}}) \subset \Loc(\Bun_{G}^{\mathrm{ss}})$ is strict. For example, consider a finite-dimensional smooth irreducible representation $\rho$ of the division algebra $D^{*}_{\frac{1}{n}}$. This defines a sheaf on a connected component of $\Bun_{\GL_{n}}^{\mathrm{ss}}$. 
Suppose $\rho$ is the mod $\ell$ reduction of a $\ol{\bb{Q}}_{\ell}$-representation $\tilde{\rho}$ such that the mod $\ell$ reduction of the semi-simplified Harris--Taylor parameter (= Fargues--Scholze parameter) of $\tilde{\rho}$ is supercuspidal (i.e does not factor through a proper Levi) then it follows, since the Fargues--Scholze parameter correspondence is compatible with mod $\ell$ reduction (by the $\Lambda$-linearity of the maps appearing in \cite[Theorem~IX.5.2]{FaScGeomLLC}), that the Fargues--Scholze parameter of $\rho$ is supercuspidal, and it in turn follows that its $*$-pushforward to all of $\Bun_{G}$ is the same as its $!$-pushforward (cf. \cite[Lemma~2.26]{BMHNCompuni}). Therefore, we see that $R^{0}j_{G*}(A)$ does not lie in $\widetilde{\Loc}(\Bun_{G}^{\mathrm{ss}})$ if $A \in \Loc(\Bun_{G}^{\mathrm{ss}})$ has such $\rho$ as a constituent.
\end{rem}
The proof we give of Proposition \ref{prop:intssequiv} is in fact related to the computation of the dualizing object on certain $\ell$-cohomologically smooth charts $\pi_{b}\colon \mathcal{M}_{b} \ra \Bun_{G}$, as studied in \cite[Section~V.3]{FaScGeomLLC}. These can be viewed as subspaces of the moduli stack of $P$-bundles, where $P$ is a parabolic. In particular, consider such a parabolic $P$ with Levi factor $L$ and the natural diagram 
\[ \begin{tikzcd}
\Bun_{P} \arrow[d,"\mf{q}_{P}"] \arrow[r,"\mf{p}_{P}"] & \Bun_{G} \\
\Bun_{L}, &  
\end{tikzcd} \]
where $\Bun_{P}$ parameterizes $P$-bundles on $X_{S}$. The $v$-stack $\Bun_{P}$ is an $\ell$-cohomologically smooth Artin $v$-stack, and the map $\mf{q}_{P}$ induces an isomorphism on the set of connected components. Moreover, each connected component is pure of some fixed $\ell$-dimension (Proposition \ref{prop: qissmooth}, Corollary \ref{cor: conncompsBunP}). We write $\dim(-)\colon |\Bun_{P}| \ra \mathbb{Z}$ for this locally constant function. It is natural to ask for an explicit description of the dualizing complex of $\Bun_{P}$. As in the case of $\Bun_{G}$, it is helpful to first look at the classifying stack $[\ast/\ul{P(F)}]$, which will define an open subspace of $\Bun_{P}$ corresponding to the trivial $P$-structure. As above, the dualizing complex on this classifying stack will be identified with the module of Haar measures, which is isomorphic to the modulus character $\delta_{P}$ as a $P(F)$-representation, by definition. 

One can wonder whether this relationship between the dualizing complex on $\Bun_{P}$ and the modulus character holds over the entire space. To formulate this, we consider the abelianization $L^{\mathrm{ab}}$ of $L$, and write 
\[ \ol{\mf{q}}_{L} \colon \Bun_{L} \ra \Bun_{L^{\mathrm{ab}}} \]
for the natural map. Since $L^{\mathrm{ab}}$ is commutative, we have an isomorphism
\[ \Bun_{L^{\mathrm{ab}}} \cong \bigsqcup_{\ol{\theta} \in B(L^{\mathrm{ab}})} [\ast/\ul{L^{\mathrm{ab}}(F)}]. \]
We write $\Delta_{P}$ for the sheaf on $\Bun_{L}$ obtained by pulling back along $\ol{\mf{q}}_{L}$ the sheaf on $\Bun_{L^{\mathrm{ab}}}$ whose value is given by $\ol{\delta}_{P}$ on each connected component, where $\ol{\delta}_{P}\colon L^{\mathrm{ab}}(F) \ra \Lambda^{*}$ is 
a factorization of $\delta_P$ (\cf \S \ref{sec:Strmod}). Our main theorem is as follows.
\begin{thm}[Theorem \ref{thm: main}]\label{thm:intmain}
For $G/F$ any connected reductive group, with parabolic $P$ and Levi $L$, we have an isomorphism
\[ K_{\Bun_{P}} \cong \mf{q}_{P}^{*}(\Delta_{P})[2\dim(\Bun_{P})] \]
of objects in $D(\Bun_{P})$.
\end{thm}
In other words, the dualizing complex is the geometrization of the modulus character in the representation theory of $p$-adic reductive groups.

Using Proposition \ref{prop:intssequiv}, Theorem \ref{thm:intmain} is reduced to determining the dualizing object on $\Bun_{P}^{\mathrm{ss}}$. 
This is the locus given by the preimage of the semistable locus in $\Bun_{L}^{\mathrm{ss}} = \bigsqcup_{b_{L} \in B(L)_{\basic}} [\ast/\ul{L_{b_{L}}(F)}]$ under $\mf{q}_{P}$. Here, for each $b_{L} \in B(L)_{\basic}$, we have an explicit geometric description of the connected components $\Bun_{P}^{b_{L}}$ of $\Bun_{P}^{\mathrm{ss}}$ in terms of certain stacks attached to Banach--Colmez spaces called Picard $v$-groupoids. However it is difficult to determine the action of $L_{b_{L}}(F)$ on the Picard $v$-groupoids directly. We use the fact that the action of $L_{b_{L}}(F)$ factors through an inner form of the automorphism group of a subspace of the Lie algebra of the unipotent radical of $P$. By this trick, the question is reduced to calculating dualizing complexes on what we call Picard $v$-groupoids attached to vector bundles (See Definition \ref{def: picardvgroup}), which we can do explicitly (Proposition \ref{prop: Picardvgroupfullaction}). 

When $G$ is quasi-split and $P$ is a Borel, this recovers \cite[Theorem~1.2]{HamGeomES}, where this result was used as a fundamental computational input in the theory of geometric Eisenstein series over the Fargues--Fontaine curve in the principal case. The result we prove here should serve a similar role in building an analogous theory for general parabolics. We partially carry this out in \S \ref{sec: Eisensteinseries} by defining the Eisenstein functor with the appropriate modulus character twists in light of Theorem \ref{thm:intmain} (Definition \ref{defn: normalizedeisfunct}), proving some of its basic properties (Lemmas \ref{lem: compositionofeisensteinseries} and \ref{lem: Weylgrouptranslate}), and computing its values after restricting to the semistable locus of $\Bun_{P}$ (Theorem \ref{thm: eisensteinvalues}). 

The importance of this result is also emphasized by its relation to various other computations occurring in the Fargues--Scholze geometric Langlands program and the theory of local Shimura varieties. Namely, we already mentioned above that, in the quasi-split case, the $\ell$-cohomologically smooth charts for $\Bun_{G}$ can be realized as $\Bun_{P}^{\theta}$, where $\theta$ has dominant isocrystal slopes (which we recall are the negatives of the Harder--Narasimhan slopes). More specifically, assume that $G$ is quasi-split with a choice of Borel $B$. Given $b \in B(G)$, there is a maximal Levi subgroup $L^{b} \subset G$ such that $b \in B(G)$ admits a basic reduction to $b_{L^{b}} \in B(L^{b})_{\basic}$ which has $G$-dominant isocrystal slopes. Here $L^{b}$ is the centralizer of the slope homomorphism of $b$. We write $P^{b}$ for the standard parabolic with Levi factor $L^{b}$, and $P^{b,-}$ for its opposite. In this case, we have an identification $G_{b} \cong (L^b)_{b_{L^{b}}}$, where $(L^{b})_{b_{L^{b}}}$ is the inner twist of $L^{b}$ induced by $b_{L^{b}}$. In particular, the group $G_{b}$ is an inner twist of $L^{b}$. The projections $\mf{q}_{P^{b}}$ and $\mf{q}_{P^{b,-}}$ then give rise to maps
\[ \mf{q}_{b}^{+}\colon \mathcal{M}_{b} \cong \Bun_{P^{b}}^{b_{L^{b}}} \ra [\ast/\ul{G_{b}(F)}], \]
\[ \mf{q}_{b}^{-}\colon [\ast/\wt{G}_{b}] \cong \Bun_{P^{b,-}}^{b_{L^{b}}} \ra [\ast/\ul{G_{b}(F)}] \]
respectively, where $\mf{q}_{b}^{+}$ and $\mf{q}_{b}^{-}$ (resp. $\mathcal{M}_{b}$ and $[\ast/\tilde{G}_{b}]$) identify with the corresponding maps considered in \cite{FaScGeomLLC}.
\begin{cor}[Proposition \ref{prop: moduluscharacterinnaturalsituations}]\label{cor:intdualMb}
For $G/F$ a connected reductive group, there exists a character $\delta_{b}$ such that we have isomorphisms 
\[ \mf{q}_{b}^{+*}(\delta_{b})[2\langle 2\rho_{G}, \nu_{b} \rangle] \cong K_{\mathcal{M}_{b}} \]
\[ \mf{q}_{b}^{-*}(\delta_{b}^{-1})[-2\langle 2\rho_{G}, \nu_{b} \rangle] \cong K_{[\ast/\wt{G}_{b}]}, 
\]
where $\rho_{G}$ denotes the half sum of the positive roots. If $G$ is quasi-split then $\delta_{b}$ is the modulus character of the standard parabolic $P^{b} \subset G$ transferred to $G_{b} \cong (L^{b})_{b_{L^{b}}}$ via the inner twisting from $L^{b}$ (In general, see Definition \ref{defn: generaldefnofmoduluschar}).
\end{cor}
For the rest of our results, we work in the larger level of generality where $\Lambda$ is a $\bb{Z}_{\ell}$-algebra and we will be implicitly working in the $D_{\lis}$ formalism of \cite[Section~VII.6]{FaScGeomLLC}, which we recall in \S 4.1. The second claim in Corollary \ref{cor:intdualMb} is compatible with the key result used in the proof of \cite[Proposition~1.1.4]{HanBeijLec}, where an alternative proof is mentioned using the compatibility of Fargues--Scholze parameters with Verdier duality on $\Bun_{G}$. We write out the details of this alternative argument in \S \ref{ssec:connection}. It is also essentially equivalent to the following corollary.
\begin{cor}[Proposition \ref{prop: compactsuppcohom}]\label{cor:intcptsupp}
The compactly supported cohomology (resp. dualizing complex) of $\wt{G}_{b}^{> 0}$ as a $G_{b}(F)$-representation is isomorphic to $\delta_{b}^{-1}[-2\langle 2\rho_{G}, \nu_{b} \rangle]$ (resp. $\delta_{b}[2\langle 2\rho_{G}, \nu_{b} \rangle]$) via the right action of $G_{b}(F)$ induced by right conjugation.
\end{cor}
The character appearing in this compactly supported cohomology appears in several places in the literature; namely in \cite{KosOngenlg} and \cite{HamLeeTorsVan} in the guise of Manotovan's product formula, and similarly in \cite{GINsemi} and \cite{HamGeomES} in the guise of comparing Hecke correspondences on $\Bun_{G}$ with the cohomology of local Shimura varieties. We summarize this as follows.
\begin{cor}[Corollary \ref{cor: moduluscharappearinginShimVars} and Corollary \ref{cor:cptcohch}]\label{cor:introconnection}
The following is true. 
\begin{enumerate}
\item The sheaf $\ol{\mathbb{F}}_{\ell}(d_{b})$ with $G_{b}(F)$-equivariant structure appearing in \cite[Theorem~7.1,Lemma~7.4]{KosOngenlg} and \cite[Theorem~1.15]{HamLeeTorsVan} is isomorphic to $\delta_{b}$ as a sheaf with $G_{b}(F)$-action. 
\item The character $\kappa$ appearing in \cite[Lemma~11.1]{HamGeomES} is equal to $\delta_{b}^{-1}$. 
\item The character $\kappa$ appearing in \cite[Lemma~4.18 (ii),Theorem~4.26]{GINsemi} is equal to $\delta_{P,\theta}$ (\cf \S \ref{ssec:connection} for the notations). 
\end{enumerate}
\end{cor}

As an application, we show the Harris--Viehmann conjecture \cite[Conjecture~8.4]{RVlocSh} in the Hodge--Newton reducible case. 
Assume that $\Lambda$ admits a fixed choice of square root of $q$.  
Let $b,b' \in B(G)$ and $\mu \in X_*(G)$. Let $E_{\mu}$ denote the reflex field of $\mu$. 
We write $R\Gamma_{\mathrm{c}}(\Sht_{G,b,b',K'}^{\mu})$ for 
the cohomology of moduli spaces of shtukas for $(G,b,b',\mu)$ at level $K' \subset G_{b'}(F)$. 
First, we show the following compatibility with character twists, which is useful in its own right. 

\begin{prop}[Proposition \ref{prop:cohtwist}]\label{prop:introcohtwist}
Let $\chi$ be a character of $G^{\ab}(F)$. 
We write $\chi_b$ and $\chi_{b'}$ for the characters of $G_b(F)$ and $G_{b'}(F)$ determined by $\chi$. 
Let $\chi_{\mu}$ be the character of $W_{E_{\mu}}$ determined by the L-parameter of $\chi$ and $\mu$. 
For a smooth representation $\rho$ of $G_b(F)$, 
we have 
\begin{align*}
\varinjlim_{K' \subset G_{b'}(F)} 
R\Hom_{G_b(F)} (R\Gamma_{\mathrm{c}}&(\Sht_{G,b,b',K'}^{\mu}),\rho \otimes \chi_b) \\ &\cong 
\varinjlim_{K' \subset G_{b'}(F)} R\Hom_{G_b(F)} (R\Gamma_{\mathrm{c}}(\Sht_{G,b,b',K'}^{\mu}),\rho ) \otimes \chi_{b'} \otimes \chi_{\mu}, 
\end{align*}
in the derived category of $G_{b'}(F) \times W_{E_{\mu}}$-representations, where $G_{b'}(F)$ acts smoothly and $W_{E_{\mu}}$ acts continuously.
\end{prop}

Using Corollary \ref{cor:introconnection} and Proposition \ref{prop:introcohtwist} as well as results in \cite{GINsemi}, 
we show the following form of the Harris--Viehmann conjecture (\cf \cite[Conjecture~8.4]{RVlocSh}, where we note that there are missing Tate twists here as first observed by Bertoloni Meli (See \cite[Conjecture~3.2.1]{BertoloniELTypeHarrisConjecture} for the fixed conjecture and the discussion around \cite[Example~3.2.5]{BertoloniELTypeHarrisConjecture} for an explanation))
\footnote{This also fixes a gap in \cite{HanHarr}. In particular, the Tate twists in \cite[Theorem~4.13]{HanHarr} should be replaced by modulus character twists coming from Corollary \ref{cor:introconnection}, and then one has to argue as in \S 4.2 to deduce Proposition \ref{prop:RGammaparab}.}: 

\begin{thm}[Theorem \ref{thm: HarrisViehmannConj}]{\label{thm:introHarrisViehmannConj}}
Assume that $F$ is $p$-adic, $G$ is quasi-split, 
$([b],[b'],\mu)$ is Hodge--Newton reducible for $L$ with reductions $[\theta],[\theta'] \in B(L)$ of $b$ and $b'$, $L \supset L^b$, $b'$ is basic and $\mu$ is minuscule. Let $P'$ be a parabolic of $G_{b'}$ corresponding to $P=LP^{b,-}$.
For a smooth representation $\rho$ of $G_b(F)$, 
we have 
\begin{align*}
\varinjlim_{K' \subset G_{b'}(F)} &
R\Hom_{G_b(F)}(R\Gamma_{\mathrm{c}}(\Sht_{G,b,b',K'}^{\mu}),\rho)\\ & \cong 
\Ind_{P'(F)}^{G_{b'}(F)} \varinjlim_{K' \subset L_{\theta'}(F)} R\Hom_{L_{\theta}(F)} (R\Gamma_{\mathrm{c}}(\Sht_{L,\theta,\theta',K'}^{\mu}),\rho) \otimes || \cdot ||^{-\frac{d_{P,\theta}}{2}}[-d_{P,\theta}], 
\end{align*}
in the derived category of $G_{b'}(F) \times W_{E_{\mu}}$-representations, where $G_{b'}(F)$ acts smoothly and $W_{E_{\mu}}$ acts continuously, $d_{P,\theta}$ is an integer defined in \S \ref{ssec:modch}, and $|| \cdot ||$ denotes the norm character of $W_{E_{\mu}}$ under the geometric normalization of local class field theory. 
\end{thm}

In \S \ref{sec:DualPic}, we recall the notion of the Picard $v$-groupoid attached to a vector bundle $\mathcal{E}$ on the Fargues--Fontaine $F$ and prove several results on its dualizing complex. In \S \ref{sec:Strmod}, we describe the geometric structure of the moduli stack of parabolic structures and prove Proposition \ref{prop:intdualBunG}, Proposition \ref{prop:intssequiv}, Theorem \ref{thm:intmain}, and Corollary \ref{cor:intdualMb}. In \S \ref{sec:App}, we deduce the applications to the Harris--Viehmann Conjecture and geometric Eisenstein series. 
\subsection*{Acknowledgements}
We would like to thank Laurent Fargues, David Hansen, Teruhisa Koshikawa, Peter Scholze, and Ruth Wild for helpful discussions. Many of the key ideas surrounding this project were found during the Hausdorff Trimester program ``Arithmetic of the Langlands Program'' at the Hausdorff Research Institute for Mathematics funded by the Deutsche Forschungsgemeinschaft (DFG, German Research Foundation) under Germany's Excellence Strategy - EXC-2047/1 - 390685813, and we would like to thank the organizers for the fantastic learning environment. We thank Konrad Zou for helpful comments. 
We also thank the referee for helpful suggestions to improve the manuscript. 
This work was carried out while L. Hamann was a NSF postdoctoral scholar, and he thanks them for their support. This work was also supported by JSPS KAKENHI Grant Number 22H00093. 

\subsection*{Notation}
\begin{itemize}
\item For a field $F$, let $\Gamma_F$ 
denote the absolute Galois group of $F$. 
\item For a non-archimedean local field $F$, let 
$\breve{F}$ denote the 
completion of the maximal unramified 
extension of $F$. We write $\sigma$ for the arithmetic Frobenius on $\breve{F}$. 
\item Let $F$ be a non-archimedean local field with residue characteristic $p$.  Let $k_F$ be the residue field of $F$. Let $q$ be the cardinality of $k_F$ and $\omega$ a uniformizing element. 
\item Let $C := \widehat{\ol{F}}$ be the completion of the algebraic closure of $F$. Let $k$ be the algebraic closure of $k_F$. 
\item We let $\Perf_k$ denote the category of perfectoid spaces over $k$, and 
$\AffPerf_k$ the category of affinoid perfectoid spaces over $k$. 
For $S \in \Perf_k$, let $\Perf_S$ denote the category of perfectoid spaces over $S$. 
\item Let $\Lambda$ be a torsion ring of order prime to $p$ unless otherwise stated. 
\item For a $v$-stack or diamond $Z$, we write $D(Z) := D_{\text{\'et}}(Z,\Lambda)$ for the derived category of \'etale $\Lambda$ sheaves on $Z$, as defined in \cite{SchEtdia}\footnote{At several points, we will take limits of these derived categories. By this, we will mean that we take the limit (in the $\infty$-category of small $\infty$-categories) of the $\infty$-categorical enhancements of these derived categories constructed in \cite[Proposition~26.2]{SchEtdia}, and then pass to the homotopy category to obtain a derived category. Similarly, when taking a limit of objects in one such fixed derived category, we will mean the homotopy limit, which similarly is formed by taking the limit in the $\infty$-categorical enhancement and then considering the resulting object in the homotopy category.}
\item For a fine map \cite[Definition~1.3]{GHWEnSixfunc} $f\colon X \ra Y$ of decent $v$-stacks \cite[Definition~1.2]{GHWEnSixfunc},  we will work with the four functors $Rf_{!},Rf_{*}\colon D(X) \ra D(Y)$ and $Rf^{!},Lf^{*}\colon D(Y) \ra D(X)$ constructed in \cite{GHWEnSixfunc}, which, together with $R\mathcal{H}om(-,-)$ and $\otimes^{\mathbb{L}}$, extend the six functor formalism in \cite{SchEtdia}.
\item We let $\ast = \Spd(k)$. 
\item For $X \ra \ast$ a decent $v$-stack fine over the base, we write $K_{X} := f^{!}(\Lambda) \in D(X)$ for the dualizing complex, and $R\Gamma_{c}(X,\Lambda) := f_{!}(\Lambda) \in D(\Lambda)$ for the compactly supported cohomology.
\item Let $G$ be a connected reductive algebraic group over $F$. Let $P$ be a parabolic subgroup of $G$, and $L$ the Levi factor of $P$. 
\item For a connected reductive group $H/F$, we write $D(H(F),\Lambda)$ for the unbounded (left-completed) derived category of smooth representations.
\item We let $\Ind_{P(F)}^{G(F)}\colon D(L(F),\Lambda) \ra D(G(F),\Lambda)$ 
denote the functor given by parabolic induction. We note that this has no higher derived functors since parabolic induction is exact. 
\item We let $\delta_{P}\colon P(F) \ra \Lambda^{*}$ denote the modulus character (\cf \cite[I. 2.6]{VigReplmod}). After fixing a choice of square root of $q$ in $\Lambda$, we define the square root $\delta_{P}^{1/2}$ of the modulus character with respect to this choice, and write $i_{P(F)}^{G(F)}(-) := \Ind_{P(F)}^{G(F)}(- \otimes \delta_{P}^{1/2})$ for the normalized parabolic induction functor. 
\item For $S \in \Perf_k$, we write $X_{S}$ for the associated relative adic Fargues--Fontaine curve. 
For $S \in \AffPerf_k$, we write $X_{S}^{\alg}$ for the associated schematic Fargues--Fontaine curve. 
We simply write $X$ for $X_{C^{\flat}}$, where $C^{\flat}$ is the tilt of $C$. 
\item Let $B(G) = G(\breve{F})/(b \sim gb\sigma(g)^{-1})$ be the Kottwitz set of $G$. For $b \in B(G)$, we write $G_{b}$ for the $\sigma$-centralizer of $b$
\item For $b \in G(\breve{F})$, let $b^{\ab}$ and $b^{\ad}$ denote the image of $b$ in $G^{\ab}(\breve{F})$ and $G^{\ad}(\breve{F})$ respectively. 
\item For $b \in B(G)$ and $S \in \Perf_k$ (resp. $S \in \AffPerf_k$), we write $\mathcal{E}_{b,S}$ (resp. $\mathcal{E}_{b,S}^{\mathrm{alg}}$) for the associated $G$-bundle on $X_S$ (resp. $X_S^{\mathrm{alg}}$).
\item We write $\ol{(-)}$ for the natural composition:
\[ B(L) \ra B(L^{\mathrm{ab}}) \cong^{\kappa} \mathbb{X}_{*}(L^{\mathrm{ab}}_{\ol{F}})_{\Gamma_{F}}. \]
\item We write
\[ \langle -,- \rangle\colon \Lcoinv \times \mathbb{X}^{*}(L_{\ol{F}}^{\mathrm{ab}})^{\Gamma_{F}} \ra \mathbb{Z}, \]
for the natural pairing induced by the pairing between cocharacters and characters.
\item We recall that the Kottwitz set $B(G)$ is equipped with two maps:
\begin{itemize}
    \item The slope homomorphism
    \[ \nu\colon B(G) \rightarrow \mathbb{X}_{*}(T_{\overline{F}})^{+,\Gamma_{F}}_{\mathbb{Q}} \]
    \[ b \mapsto \nu_{b}, \]
    where $\mathbb{X}_{*}(T_{\overline{F}})_{\mathbb{Q}}^{+}$ is the set of rational geometric dominant cocharacters of $G$. 
    \item The Kottwitz invariant
    \[ \kappa_{G}\colon B(G) \rightarrow \pi_{1}(G)_{\Gamma_{F}}, \]
    where $\pi_1(G)$ denotes the algebraic fundamental group of Borovoi.
\end{itemize}
\item If $G$ is quasi-split with a fixed Borel pair $(B,T)$ then, for $b \in B(G)$, we let $L^b$ be the centralizer of $\nu_b$, $P^{b}$ the parabolic subgroup of $G$ associated to $\nu_b$ such that the weights of $\nu_b$ in $\Lie (P^{b})$ is non-negative, and $P^{b,-}$ the parabolic subgroup of $G$ opposite to $P^{b}$. 
\item Throughout, we will use the terminology that $X$ is an iterated fibration of $v$-stacks in the $v$-stacks $Y_{i} \ra \ast$ for $i = 1,\ldots,n$. What we mean precisely is that there exists
a sequence of morphisms 
\[ X_{n} \xrightarrow{f_{n}} \cdots \ra X_{1} \xrightarrow{f_{1}} X_0=\ast \]
such that $X_{n} = X$ and the maps $f_{i} \colon X_{i} \ra X_{i - 1}$ are fibrations in $Y_{i}$ for the $v$-topology on $X_{i - 1}$. In other words, there exists a $v$-surjection $S \ra X_{i - 1}$ from a perfectoid space $S$ such that the pullback of $f_{i}$ along the surjection identifies with the natural projection  $S \times Y_{i} \ra S$.
\end{itemize}

\section{Dualizing complexes on Picard $v$-groupoids}\label{sec:DualPic}

For $i \in \mathbb{Z}$, $S \in \Perf_k$, and a two term complex $\mathcal{E}^{*} = \{\mathcal{E}^{-1} \rightarrow \mathcal{E}^{0}\}$ of vector bundles on $X_{S}$, we define the $v$-sheaves $\mathcal{H}^{i}(\mathcal{E}^{*}) \rightarrow S$ as the sheaves sending $T \in \Perf_{S}$ to the hyper-cohomology $\mathbb{H}^{i}(X_{T},\mathcal{E}^{*}_{T})$, where $\mathcal{E}^{*}_{T}$ is the base change of the complex to $X_{T}$. 
If $\mathcal{E}^{*}$ is a two term complex of vector bundles on $X$, we also have an absolute version $\mathcal{H}^{i}(\mathcal{E}^{*}) \rightarrow \Spd k$. 
We will make repeated use of the following fact.
\begin{lem}\label{lem:BCadjisom}
Let $f \colon S' \to S$ be a morphism of small v-stacks. We assume one of the following: 
\begin{enumerate}
\item $S \in \Perf_k$ and $\ca{E}$ is a vector bundle on $X_S$ that is everywhere of positive (resp. negative) slopes.  
\item $\ca{E}$ is a vector bundle on $X$ that is of positive (resp. negative) slopes. 
\end{enumerate}
If $f$ is a torsor under $\ca{H}^0(\ca{E})$ (resp. $\ca{H}^1(\ca{E})$), then $f$ is $\ell$-cohomologically smooth, 
$f^* \colon D(S) \to D(S')$ is fully faithful and 
the adjunction morphism 
$Rf_! Rf^! A \to A$ is an isomorphism for $A \in D(S)$. 
\end{lem}
\begin{proof}
The first claim follows from \cite[Proposition II.3.5]{FaScGeomLLC}.
The second claim is \cite[Proposition V.2.1]{FaScGeomLLC}. 
The third claim is reduced to the case where the torsor is split and $\ca{E}=\ca{O}_X(1/n)$ for some $n$, as in the proof of \cite[Proposition V.2.1]{FaScGeomLLC}. In this case, the claimed isomorphism is proved in the proof of \cite[Proposition V.2.1]{FaScGeomLLC}. 
\end{proof}

\begin{lem}\label{lem:BCadjisomAn}
Let $f \colon S' \to S$ be a morphism of small v-stacks over $\Spd C$. 
If $f$ is a torsor under $(\mathbb{A}_C^n)^{\diamond}$, then $f$ is $\ell$-cohomologically smooth, 
$f^* \colon D(S) \to D(S')$ is fully faithful and 
the adjunction morphism 
$Rf_! Rf^! A \to A$ is an isomorphism for $A \in D(S)$. 
\end{lem}
\begin{proof}
This is proved in the same way as Lemma \ref{lem:BCadjisom}. 
\end{proof}

The calculation of the dualizing complex on $\Bun_{P}$ will ultimately reduce to a calculation of the dualizing complex on the following kinds of spaces discussed in \cite[Definition~3.6]{HamJacCrit}.
\begin{defn}{\label{def: picardvgroup}}
Let $S \in \Perf_k$ and $\mathcal{E}^{*} = \{\mathcal{E}^{-1} \rightarrow \mathcal{E}^{0}\}$ a two term complex of vector bundles on $X_{S}$
We consider the action of $x \in \mathbb{H}^{-1}(X_{T},\mathcal{E}^{*})$ on $y \in \mathbb{H}^{0}(X_{T},\mathcal{E}^{*})$ via the map $y \mapsto y + d(x)$ and check that this gives a well-defined action on cohomology. This allows us to form the Banach--Colmez space like $v$-stack
\[ \mathcal{P}(\mathcal{E}^{*}) := [\mathcal{H}^{0}(\mathcal{E}^{*})/\mathcal{H}^{-1}(\mathcal{E}^{*})] \rightarrow S, \]
which we refer to as the Picard $v$-groupoid of $\mathcal{E}^{*}$.
\end{defn}
We will only consider this in the case that $\mathcal{E}^{*} := \mathcal{E}[1]$ for a vector bundle $\mathcal{E}$ on $X$. It follows from \cite[Proposition~3.14]{HamJacCrit} that $\mathcal{P}(\mathcal{E}[1])$ is $\ell$-cohomologically smooth over $\Spd(C)$ 
of pure $\ell$-dimension equal to $-\mathrm{deg}(\mathcal{E})$. In this case, the stack $\mathcal{P}(\mathcal{E}[1]) \ra \Spd(C)$ also has an absolute version, which we abusively also write as $\mathcal{P}(\mathcal{E}[1]) \ra \Spd(k)$ and will work with for the rest of the section.

We consider the dualizing complex $K_{\mathcal{P}(\mathcal{E}[1])}$ on this space. We let $G_{\mathcal{E}}$ denote the automorphism group of the isocrystal attatched to $\mathcal{E}$, and write 
\[ \mathrm{det}\colon G_{\mathcal{E}}(F) \ra (G_{\mathcal{E}}/G_{\mathcal{E}}^{\mathrm{der}})(F) \cong (F^{*})^{r} \xrightarrow{\mathrm{mult}} F^{*}, 
\]
where the last map is given by taking the product of the elements and $r \in \bb{N}_{\geq 1}$ is the number of simple factors of $G_{\mathcal{E}}^{\mathrm{ad}}$. The group $G_{\mathcal{E}}(F)$ acts on 
$K_{\mathcal{P}(\mathcal{E}[1])}$ via the right action on $\mathcal{E}$ by the inverse. The main result of this section is the following.
\begin{prop}{\label{prop: Picardvgroupfullaction}}
We have an isomorphism
\[ K_{\mathcal{P}(\mathcal{E}[1])} \cong |\mathrm{det}(\cdot)|[-2\mathrm{deg}(\mathcal{E})] \]
of sheaves with $G_{\mathcal{E}}(F)$-action. Here $|\cdot|\colon F^{*} \ra \Lambda^{*}$ denotes the norm character. 
\end{prop}
We first need some preliminaries. For $\lambda=d/r$ where $r$ is a positive integer and $d$ is an integer prime to $r$, 
we define the isocrystal $V_{-\lambda}$ as the vector space 
$\breve{F}^r$ with Frobenius $\varphi_{-\lambda}$ given by $\varpi_{-\lambda}^{-d}$, where 
\begin{equation}\label{eq:Frlam}
\varpi_{-\lambda} = 
\begin{pmatrix}
0 & 1 & 0 & \cdots & 0\\
0 & 0 &  1  &        & \vdots \\
\vdots &        & \ddots & \ddots    & 0 \\
0 &        &        & 0 & 1 \\
\varpi & 0 & \cdots & 0 & 0
\end{pmatrix} \in M_r (F),
\end{equation}
and 
let $\ca{O}(\lambda)$ be the vector bundle on the Fargues--Fontaine curve given by the isocrystal $V_{-\lambda}$ as \cite[Section~II.2]{FaScGeomLLC}. 
We put $D_{-\lambda} =\End (V_{-\lambda})$. 
Then $\ca{O}(\lambda)$ has a right action of $D_{-\lambda}$ by the inverse. 
We view $\varpi_{-\lambda}$ in \eqref{eq:Frlam} as an element of 
$D_{-\lambda}$, and denote it by the same symbol $\varpi_{-\lambda}$. 

\begin{lem}\label{lem:Olamstable}
Let $\lambda=d/r$ where $r$ is a positive integer and $d$ is an integer prime to $r$. 
Then there is an isomorphism 
\[
 \ca{H}^{0}(\ca{O}(\lambda)) \cong \Spd k[[x_1^{1/q^{\infty}},\ldots ,x_{d}^{1/q^{\infty}}]] 
\]
such that the action of $\ca{O}_{D_{-\lambda}}$ on 
$\mathcal{H}^{0}(\mathcal{O}(\lambda))$ via automorphisms of $\mathcal{O}(\lambda)$ is induced from an action of 
$\ca{O}_{D_{-\lambda}}$ on 
$\Spd k[[x_1,\ldots ,x_{d}]]$ over $\Spd k$ by taking inverse limits over the $q$-power Frobenius map. 
\end{lem}
\begin{proof}
We follow the proof of \cite[Proposition II.2.5 (iv)]{FaScGeomLLC}. 
In the $p$-adic case, the question is reduced to the case where $F=\mathbb{Q}_p$ as in the proof of \cite[Proposition II.2.5 (iv)]{FaScGeomLLC}. Then, in the proof of \cite[Proposition 3.1.3 (iii)]{ScWeMpd}, 
the coordinates of $\mathcal{H}^{0}(\mathcal{O}(\lambda))$ are obtained from the coordinates of $G$, where $G$ is the connected $p$-divisible group over $k$ corresponding to $\ca{O}_{\breve{F}}^r \subset V_{-\lambda}$ with Frobenius $\varphi_{-\lambda}$ via Dieudonn\'{e} theory. Since $\ca{O}_{D_{-\lambda}}$ acts on $G$ in a way that is compatible with the action of $\mathcal{O}_{D_{-\lambda}}$ on $\mathcal{H}^{0}(\mathcal{O}(\lambda))$ by automorphisms, the claim in the $p$-adic case follows. 

Assume that $F$ is of equal characteristic. 
By the proof of \cite[Proposition 3.1.3 (iii)]{ScWeMpd}, 
$\mathcal{H}^{0}(\mathcal{O}(\lambda))(\Spa (R,R^+))$ is identified with $B_{R,[1,\infty]}^{\varphi^r=\varpi^d}$, whose element 
$\sum_{i \in \mathbb{Z}} r_i \varpi^i$ is determined by topologically nilpotent $r_1, \ldots , r_d$. We note that $B_{R,[1,\infty]}^{\varphi^r=\varpi^d}$ is identified with $\varphi_{-\lambda}$-fixed part $(B_{R,[1,\infty]}^d)^{\varphi_{-\lambda}}$ of $B_{R,[1,\infty]}^d$ under 
\begin{equation}\label{eq:idenBdB}
    (B_{R,[1,\infty]}^d)^{\varphi_{-\lambda}} \cong 
    B_{R,[1,\infty]}^{\varphi^r=\varpi^d} ;\ 
    (b_i)_{1 \leq i \leq d} \mapsto b_1. 
\end{equation}
Under the identification \eqref{eq:idenBdB}, the natural action of $a \in \mathbb{F}_{q^r}$ and $\varpi_{-\lambda}$ on $(B_{R,[1,\infty]}^d)^{\varphi_{-\lambda}}$ is identified with the action on $B_{R,[1,\infty]}^{\varphi^r=\varpi^d}$ 
given by $r_i \mapsto a r_i$ and 
$r_i \mapsto \varphi^{m_i} (r_{\sigma (i)})$ for $1 \leq i \leq d$, where $m_i \in \mathbb{Z}$ and $\sigma \in S_d$. Since $\varphi_{-\lambda}=\varpi_{-\lambda}^{-d} \varphi$, we see that the action of $\varpi_{-\lambda}^d$ on $r_1, \ldots , r_d$ is given by 
$r_i \mapsto \varphi (r_i)$. Hence we have $\sum_{1 \leq i \leq d} m_i =1$. Therefore, replacing the coordinates $r_i$ with $1 \leq i \leq d$ by their $q^{n_i}$-th power for some $n_i \in \mathbb{Z}$, we may assume that one of $m_i$ for $1 \leq i \leq d$ is $1$ and the other are $0$. Let $x_1, \ldots ,x_d$ be the coordinate of $\mathcal{H}^{0}(\mathcal{O}(\lambda))$ given by $r_1, \ldots , r_d$ after this replacement. Then $k[[x_1, \ldots , x_d]]$ is stable under the action of $\mathbb{F}_{q^r}$ 
and $\varpi_{-\lambda}$. 
Since $\ca{O}_{D_{-\lambda}}$ is generated by $\mathbb{F}_{q^r}$ and $\varpi_{-\lambda}$, the claim follows. 
\end{proof}

We recall that the map $\det \colon \GL_n (D_{-\lambda}) \ra F^{*}$ can be identified with the reduced norm, which we denote by $\mathrm{Nrd}_{M_{n}(D_{-\lambda})/F}$. 
\begin{lem}{\label{lemm:Oaction}}
Let $\lambda=d/r$ where $r$ is a positive integer and $d$ is an integer prime to $r$. 
Let $n$ be a positive integer. 
Let $f_{\lambda,n} \colon \ca{H}^{0}(\ca{O}(\lambda)^n) \to *$, 
$g_{\lambda,n} \colon [*/\ca{H}^{0}(\ca{O}(\lambda)^n)] \to *$ 
and $h_{\lambda,n} \colon \ca{H}^{1}(\ca{O}(\lambda)^n) \to *$ be the structure maps, where $\mathcal{H}^{0}(\ca{O}(\lambda))$ (resp. $\mathcal{H}^{1}(\ca{O}(\lambda))$) is the absolute Banach--Colmez space parameterizing elements of $H^{0}(X_{S},\ca{O}(\lambda)_{S})$ (resp. $H^{1}(X_{S},\ca{O}(\lambda)_{S})$) for $S \in \Perf_k$.  
\begin{enumerate}
\item \label{enu:Opos}
If $\lambda$ is positive, then $\ca{H}^{1}(\ca{O}(\lambda)^n)=*$ and 
we have isomorphisms 
\begin{align*}
 f_{\lambda,n,!}\Lambda &\cong |\Nrd_{M_n(D_{-\lambda})/F}(\cdot)|[-2dn] \\ 
 f_{\lambda,n}^! \Lambda &\cong |\Nrd_{M_n(D_{-\lambda})/F}(\cdot)|^{-1}[2dn] \\ 
 g_{\lambda,n,!}\Lambda &\cong |\Nrd_{M_n(D_{-\lambda})/F}(\cdot)|^{-1}[2dn] \\ 
 g_{\lambda,n}^! \Lambda &\cong |\Nrd_{M_n(D_{-\lambda})/F}(\cdot)|[-2dn]  
\end{align*}
as sheaves with $\GL_n (D_{-\lambda})$-action. 
\item \label{enu:Ozero}
If $\lambda=0$, then $\ca{H}^{0}(\ca{O}(\lambda)^n)\cong \underline{F}^n$, 
$\ca{H}^{1}(\ca{O}(\lambda)^n)=*$ and we have isomorphisms 
\begin{align*}
 g_{\lambda,n,!}\Lambda &\cong |\Nrd_{M_n(F)/F}(\cdot)|^{-1} \\ 
 g_{\lambda,n}^! \Lambda &\cong |\Nrd_{M_n(F)/F}(\cdot)| 
\end{align*}
as sheaves with $\GL_n (F)$-action. 
\item \label{enu:Oneg}
If $\lambda$ is negative, then $\ca{H}^{0}(\ca{O}(\lambda))=*$ and 
we have isomorphisms 
\begin{align*}
 h_{\lambda,n,!}\Lambda &\cong |\Nrd_{M_n(D_{-\lambda})/F}(\cdot)|^{-1}[2dn]\\ 
 h_{\lambda,n}^! \Lambda &\cong |\Nrd_{M_n(D_{-\lambda})/F}(\cdot)|[-2dn] 
\end{align*}
as sheaves with $\GL_n (D_{-\lambda})$-action. 
\end{enumerate}
\end{lem}
\begin{proof}
The first claims in \ref{enu:Opos} and \ref{enu:Oneg}, and the first two claims in  \ref{enu:Ozero} are proven in \cite[Proposition II.2.5]{FaScGeomLLC}. 
The later two claims in  \ref{enu:Ozero} follow from 
\cite[Example~4.2.4]{HKWKotloc}. 
For the remaining claims, it suffices to construct isomorphisms as sheaves that are compatible with the actions of $\GL_n (\ca{O}_{D_{-\lambda}})$ and $\mathrm{diag}(\varpi_{-\lambda}^{m_1}, \ldots , \varpi_{-\lambda}^{m_n})$ for $(m_1,\ldots,m_n) \in \bb{Z}^n$, since these generate $\GL_n (D_{-\lambda})$ by the Cartan decomposition. 
We note that it suffices to show the claim after the base change to $\Spd C$, by invoking \cite[Theorem~V.1.1]{FaScGeomLLC}. 

We prove the second claim in \ref{enu:Opos} 
in a way similar to \cite[Lemma~4.17]{GINsemi}. The claims on $g_{\lambda,n}$ are reduced to the claims on $f_{\lambda,n}$ by Lemma \ref{lem:BCadjisom}. 
First, we assume that $0<\lambda \leq 1$. 
We take an isomorphism 
\[
 \ca{H}^{0}(\ca{O}(\lambda)^n) \cong \Spd k[[x_1^{1/q^{\infty}},\ldots ,x_{dn}^{1/q^{\infty}}]] 
\]
using Lemma \ref{lem:Olamstable}. 
By \cite[Theorem 24.1]{SchEtdia}, we know that 
$f_{\lambda,n}^! \Lambda \cong \Lambda [2dn]$ if we forget the group actions. 
Hence, the claim on $f_{\lambda,n}^!\Lambda$ is reduced to the claim on $f_{\lambda,n,!}\Lambda$ by Lemma \ref{lem:BCadjisom}. 
By \cite[Lemma 1.3]{ImaConv}, we have an isomorphism 
\[
f_{\lambda,n,!}\Lambda \cong |\Nrd_{M_n(D_{-\lambda})/F}(\cdot)|[-2dn] 
\]
that is compatible with the action of 
$\mathrm{diag}(\varpi_{-\lambda}^{m_1}, \ldots , \varpi_{-\lambda}^{m_n})$. 
The action of $\GL_n (\ca{O}_{D_{-\lambda}})$ on $\mathcal{H}^{0}(\mathcal{O}(\lambda)^{n})$ is induced from an action of $\GL_n (\ca{O}_{D_{-\lambda}})$ on 
$\Spd k[[x_1,\ldots ,x_{dn}]]$ over $\Spd k$ by taking inverse limits, as in Lemma \ref{lem:Olamstable}. We show that $\GL_n (\ca{O}_{D_{-\lambda}})$ acts trivially on 
$f_{\lambda,n,!}\Lambda$. 
As mentioned above, it suffices to show this after base changing to $\Spd C$. 

We can see that the action of $\GL_n (\ca{O}_{D_{-\lambda}})$ on 
$\Spd k[[x_1,\ldots ,x_{dn}]] \times_{\Spd k} \Spd C$ comes from an action of $\GL_n (\ca{O}_{D_{-\lambda}})$ on $\mathring{\mathbb{D}}^{dn}_C=\Spa (\mathcal{O}_C[[x_1,\ldots ,x_{dn}]]) \times_{\Spa(\mathcal{O}_{C},\mathcal{O}_{C})} \Spa(C,\mathcal{O}_{C})$. If $F$ is $p$-adic, this follows from \cite[Proposition 10.2.3]{ScWeBLp}. In the equal characteristic case, we can use the base change to $\Spa(C,\mathcal{O}_{C})$ of the action on $k[[x_1,\ldots ,x_{dn}]]$. 

By \cite[Lemma 15.6]{SchEtdia}, it suffices to show the same claim for $\mathring{\mathbb{D}}^{dn}_C$. 
By Poincare duality (\cf \cite[Corollary 7.5.6]{HubEtadic}), we have 
\[
 \Hc^{2dn} (\mathring{\mathbb{D}}^{dn}_C,\Lambda) \cong 
 H^0 (\mathring{\mathbb{D}}^{dn}_C,\Lambda(dn)) =\Lambda(dn). 
\]
Since any automorphism of $\mathring{\mathbb{D}}^{dn}_C$ over $C$ acts trivially on $H^0 (\mathring{\mathbb{D}}^{dn}_C,\Lambda(dn))$, we see the claim. 

Assume that $\lambda >1$. 
Then we have an exact sequence 
\begin{equation}\label{eq:BCBCA}
 0 \to \ca{H}^{0}(\ca{O}(\lambda -1)^n)|_{\Spd C} \to 
 \ca{H}^{0}(\ca{O}(\lambda)^n)|_{\Spd C} \xrightarrow{p_{\lambda,n,C}} (\bb{A}_C^{nr})^{\diamond} \to 0 
\end{equation}
induced by 
\begin{equation}\label{eq:Olaminf}
 0 \to \ca{O}_X(\lambda -1)^n \to 
  \ca{O}_X(\lambda)^n \to i_{\infty *}C^{nr} \to 0 
\end{equation}
as in \cite[Example 15.2.9.4]{ScWeBLp}, where 
$i_{\infty} \colon \infty \hookrightarrow X$ is the point given by the untilt $C$ of $C^{\flat}$. 
In \eqref{eq:Olaminf}, the first non-zero map is given by the multiplication by a non-zero section of $\mathcal{O}_X(1)$ vanishing at $\infty$ (\cf \cite[Proposition II.2.3]{FaScGeomLLC}), and the second non-zero map is given by the restriction to $\infty$. 
Under the equality $D_{-(\lambda-1)}=D_{-\lambda} \subset \End (\breve{F}^r)$, the group $\GL_n (D_{-\lambda})$ acts on $\ca{O}_X(\lambda -1)^n$ and in turn on $\ca{H}^{0}(\ca{O}(\lambda -1)^n)|_{\Spd C}$. 
By taking the fiber of $\ca{O}_X(\lambda)^n$ at $\infty$, we have a homomorphism $\GL_n (D_{-\lambda}) \to \GL_{nr}(C)$. Via this homomorphism, $\GL_n (D_{-\lambda})$ acts on $\bb{A}_C^{nr}$ and $(\bb{A}_C^{nr})^{\diamond}$. 
Then the sequence \eqref{eq:BCBCA} is equivariant for the action of $\GL_n (D_{-\lambda})$ by the construction. 

Let $\psi_{\lambda,n,C} \colon (\bb{A}_C^{nr})^{\diamond} \to \Spd C$ be the structure morphism. 
Then we have 
\begin{equation}\label{eq:Anr!}
\psi_{\lambda,n,C,!} \Lambda \cong \Lambda [-2nr]
\end{equation} 
by \cite[Lemma 1.3]{ImaConv}. 
The action of $\GL_n (D_{-\lambda})$ on \eqref{eq:Anr!} is  trivial by the argument as above using Poincare duality since the action of $\GL_n (D_{-\lambda})$ on 
$(\bb{A}_C^{nr})^{\diamond}$ is induced from an action on $\bb{A}_C^{nr}$. 
Hence, we also have 
\begin{equation}\label{eq:Anrupp}
\psi_{\lambda,n,C}^! \Lambda \cong \Lambda [2nr]. 
\end{equation} 
We show that 
\begin{equation}\label{eq:plam!}
 p_{\lambda,n,C,!} \Lambda \cong |\Nrd_{M_n(D_{-\lambda})/F}(\cdot)|[-2(d-r)n]. 
\end{equation}
By Lemma \ref{lem:BCadjisom}, we may check this after the pullback under $p_{\lambda,n,C}$, where \eqref{eq:BCBCA} splits. Hence the isomorphism follows from the case of $\lambda -1$.  
We have 
\begin{align*}
 f_{\lambda,n,C,!} \Lambda &\cong 
 \psi_{\lambda,n,C,!} (p_{\lambda,n,C,!} \Lambda ) \\
 &\cong 
 \psi_{\lambda,n,C,!} ( |\Nrd_{M_n(D_{-\lambda})/F}(\cdot)|[-2(d-r)n]) 
 \cong |\Nrd_{M_n(D_{-\lambda})/F}(\cdot)|[-2dn], 
\end{align*}
where we use \eqref{eq:plam!} at the second isomorphism and \eqref{eq:Anr!} at the third isomorphism. 
The claim on $f_{\lambda,n,C}^! \Lambda$ is proved in the same way using \eqref{eq:BCBCA} and \eqref{eq:Anrupp}. 

We prove the second claim in \ref{enu:Oneg}. 
Assume that $-1 \leq \lambda <0$. Then we have 
\[
 0 \to \ca{H}^{0}(\ca{O}(\lambda +1)^n)|_{\Spd C} \to (\bb{A}_C^{nr})^{\diamond} \to 
 \ca{H}^{1}(\ca{O}(\lambda )^n)|_{\Spd C} \to 0, 
\]
as in \cite[p.~138]{ScWeBLp}. 
Hence we have 
\begin{align*}
 h_{\lambda,n,C,!} \Lambda &\cong g_{\lambda+1,n,C,!} \Lambda \otimes \psi_{\lambda,n,C,!} \Lambda \cong g_{\lambda+1,n,C,!} \Lambda [-2nr]  ,\\ 
 h_{\lambda,n,C}^! \Lambda &\cong g_{\lambda+1,n,C}^! \Lambda \otimes \psi_{\lambda,n,C}^! \Lambda \cong g_{\lambda+1,n,C}^! \Lambda [2nr]
\end{align*}
by \eqref{eq:Anr!} and \eqref{eq:Anrupp}. 
Therefore the claim follows from the claim \ref{enu:Opos} and \ref{enu:Ozero}. 

Assume that $\lambda <-1$. Then we have 
\[
 0 \to (\bb{A}_C^{nr})^{\diamond} \to \ca{H}^{1}(\ca{O}(\lambda )^n)|_{\Spd C} \to  
 \ca{H}^{1}(\ca{O}(\lambda +1 )^n)|_{\Spd C} \to 0 
\]
as \cite[p.~139]{ScWeBLp}. 
Hence the claim is reduced to the case of $\lambda +1$ by \eqref{eq:Anr!} and \eqref{eq:Anrupp} using Lemma \ref{lem:BCadjisomAn}. 
\end{proof}
With this in hand we can prove Proposition \ref{prop: Picardvgroupfullaction}. 
\begin{proof}[Proof of Proposition \ref{prop: Picardvgroupfullaction}]
By \cite[Proposition~II.2.5]{FaScGeomLLC}, we have identifications $\mathcal{H}^{1}(\mathcal{E}) \cong \mathcal{H}^{1}(\mathcal{E}^{< 0})$ and $\mathcal{H}^{0}(\mathcal{E}) \cong \mathcal{H}^{0}(\mathcal{E}^{\geq 0}) = \mathcal{H}^{0}(\mathcal{E}^{> 0}) \times \mathcal{H}^{0}(\mathcal{E}^{= 0})$. Therefore, we can factorize the structure map $f_{\mathcal{E}}\colon \mathcal{P}(\mathcal{E}) \ra \Spd(k)$ as
\[ [\mathcal{H}^{1}(\mathcal{E}^{< 0})/\mathcal{H}^{0}(\mathcal{E}^{\geq 0})] \xrightarrow{f^{< 0}} [\Spd(F)/\mathcal{H}^{0}(\mathcal{E}^{\geq 0})] \xrightarrow{f^{> 0}} [\Spd(F)/\mathcal{H}^{0}(\mathcal{E}^{= 0})] \xrightarrow{f^{= 0}} \Spd(k). \]
Since all these maps are $\ell$-cohomologically smooth by the slope decomposition of $\ca{E}$ and \cite[Proposition~II.2.5]{FaScGeomLLC}, we can use the formula
\[ R((g \circ f)^{!})(\mathbb{F}_{\ell}) = Rg^{!}(Rf^{!}(\mathbb{F}_{\ell})) = Rf^{!}(\mathbb{F}_{\ell}) \otimes_{\mathbb{F}_{\ell}} f^{*}Rg^{!}(\mathbb{F}_{\ell}) \]
for a pair of $\ell$-cohomologically smooth maps $f$ and $g$, as follows from the Definition (\cite[Definition~IV.1.15]{FaScGeomLLC}). This reduces us to showing the claim in three cases. 
\begin{enumerate}
\item The slopes of $\mathcal{E}$ are strictly positive.
\item The slopes of $\mathcal{E}$ are all $0$.
\item The slopes of $\mathcal{E}$ are strictly negative. 
\end{enumerate}
The claim now follows by writing $\mathcal{E}$ as a direct sum of semistable bundles of fixed slope, the K\"unneth formula, and Lemma \ref{lemm:Oaction}. 
\end{proof}

\section{Structure of the moduli space of $P$-structures}\label{sec:Strmod}
\subsection{Modulus character}\label{ssec:modch}
We make use of the following notation.
\begin{enumerate}
\item Let $P$ be a parabolic with Levi factor $L$. 
We set $U=R_{\mathrm{u}}(P)$ to be the unipotent radical. 
The adjoint action of $P$ on $\Lie (U)$ induces 
\[
P \to \Aut \left( \bigwedge^{\dim U} \Lie (U) \right) \cong \Gm, 
\] 
which factors through $L^{\ab}$ and defines an element $\xi_P \in \mathbb{X}^{*}(L^{\mathrm{ab}}_{\ol{F}})^{\Gamma_{F}}$. 
\item Recall that $\delta_{P}\colon P(F) \ra \Lambda^{\times}$ is the modulus character of $P$. We define $\ol{\delta}_{P}\colon L^{\mathrm{ab}}(F) \ra \Lambda^{\times}$ by $\ol{\delta}_{P}(x)=\lvert \xi_P (x) \rvert$. 
\item Given $\theta \in B(L)$, we consider the composition 
\[ \delta_{P,\theta}\colon L_{\theta}(F) \ra (L^{\ab})_{\theta^{\ab}}(F) \cong L^{\ab}(F) \xrightarrow{\ol{\delta}_{P}} \Lambda^{\times}, \]
where $\theta^{\ab}$ denotes the image of $\theta$ in $B(L^{\ab})$. 
We put $d_{P,\theta} := \langle \xi_P, \ol{\theta} \rangle$. 
\end{enumerate}

For a locally free sheaf $\scr{E}$ over a scheme $S$, we define the group scheme $\mathbf{W}(\scr{E})$ over $S$ by $\mathbf{W}(\scr{E})(T)=(\scr{E} \otimes_{\ca{O}_S} \ca{O}_T)(T)$. 

\begin{lem}\label{lem:filtU}
Let $P$ be a parabolic subgroup of a reductive group scheme $G$ over a scheme $S$ with a Levi subgroup $L$. Then 
there is a filtration 
\[ U=U_0 \supset U_1 \supset \cdots  \supset U_n =\{1\} 
\]
of the unipotent radical of $P$ 
stable under the conjugation action of $P$ such that for $0 \leq i \leq n-1$ 
\begin{enumerate}
\item\label{en:PUfact}
the action of $P$ on $U_{i}/U_{i+1}$ factors through $L$, 
\item\label{en:UUWgrp}
there is an isomorphism 
\[
 U_{i}/U_{i+1} \cong \mathbf{W} (\Lie (U_{i}/U_{i+1})) 
\]
of group schemes with an action of $L$, 
\item\label{en:UWsch}
there is an isomorphism 
\[
 U \cong \mathbf{W} \left( \bigoplus_{0 \leq i \leq n-1} \Lie (U_{i}/U_{i+1})\right) 
\]
of schemes with an action of $L$, and
\item\label{en:LieUtwodec}
if $S$ is connected semi-local,  
there is unique $\alpha_i \in X^*(A_L)$ such that 
$\Lie (U_{i})=\Lie (U)_{\alpha_i} \oplus \Lie (U_{i+1})$ for $0 \leq i \leq n-1$, where $A_{L}$ is the maximal split torus inside the center of $L$. 
\end{enumerate}
\end{lem}
\begin{proof}
We do the argument as in \cite[XXVI, 2.1.1]{SGA3-3} for a pinning of $G$ adapted to $P$ such that the maximal torus is contained in $L$. 
Then, there is an isomorphism 
$U_{i}/U_{i+1} \cong \mathbf{W} (\scr{E}_i)$ compatible with the actions of $L$ for a locally free sheaf $\scr{E}_i$ on $S$ with a linear action of $L$. By taking the Lie algebras, we obtain an isomorphism $\Lie (U_{i}/U_{i+1}) \cong \scr{E}_i$ compatible with the actions of $L$. 
Furthermore,  $\prod_{a(\gamma)=i+1}U_{\gamma}$ in \cite[XXVI, 2.1.1]{SGA3-3} gives a splitting of $U_i \to U_i/U_{i+1}$ as schemes with the action of $L$. 

In the argument of \cite[XXVI, 2.1.2]{SGA3-3}, different pinnings are conjugate by $L$ since we use only the pinnings such that the maximal tori are contained in $L$. Hence $\prod_{a(\gamma)=i+1}U_{\gamma}$ above  descends. Therefore we obtain a filtration $\{ U_i \}_{0 \leq i \leq n}$ satisfying the conditons \ref{en:PUfact}, \ref{en:UUWgrp} and \ref{en:UWsch}. 

Assume that $S$ is connected semi-local. 
We note that $L$ is the centralizer of $A_L$ in $G$ by \cite[XXVI, Proposition 6.8]{SGA3-3}. 
By using the weight decomposition of $\Lie (U_{i}/U_{i+1})$ with respect to the action of $A_L$, we can refine the filtration $\{ U_i \}_{0 \leq i \leq n}$ further so that the condition \ref{en:LieUtwodec} is satisfied. 
\end{proof}

\begin{lem}
The modulus character $\delta_{P}$ is equal to 
$P(F) \to L(F) \to L^{\ab}(F) \xrightarrow{\ol{\delta}_{P}} \Lambda^{\times}$. 
\end{lem}
\begin{proof}
Since the both characters are trivial on $U(F)$, it suffices to show the equality on $L(F)$. 
Let $l$ be an element of $L(F)$. 
We take a compact open subgroup $K$ of $P(F)$, and 
$\{ U_i \}_{0 \leq i \leq n}$ as Lemma \ref{lem:filtU}. 
We put $K_i=K \cap U_i(F)$ for $0 \leq i \leq n$. 
Let $\ol{K}$ and 
$\overline{K}_i$ be the images of $K$ in $L(F)$ and 
$K_i$ in $(U_i/U_{i+1})(F)$ for $0 \leq i \leq n-1$. 
By Lemma \ref{lem:filtU} \ref{en:UWsch}, $\ol{K}$ and 
$\overline{K}_i$ are compact open in $L(F)$ and 
$(U_i/U_{i+1})(F)$ respectively, and 
we have 
\begin{equation}\label{eq:delPdec}
\delta_P(l)=[\Ad(l)K:K] =[\Ad(l)\ol{K}:\ol{K}] 
\prod_{0 \leq i \leq n-1} [\Ad(l)\ol{K}_i :\ol{K}_i] . 
\end{equation}
Since the modulus character of $L$ is trivial, we have 
\begin{equation}\label{eq:adolK}
 [\Ad(l)\ol{K}:\ol{K}]=1. 
\end{equation}
We take an $\ca{O}_F$-lattice $N_i$ of $\Lie (U_i/U_{i+1})(F)$
which corresponds to a compact open subgroup of $\ol{K}_i$ under the isomorphism in Lemma \ref{lem:filtU} \ref{en:UUWgrp}. Then we have 
\begin{equation}\label{eq:adolKN}
[\Ad(l)\ol{K}_i :\ol{K}_i]=[\Ad(l)N_i :N_i]. 
\end{equation}
The claim follows from \eqref{eq:delPdec}, 
\eqref{eq:adolK} and \eqref{eq:adolKN}. 
\end{proof}

\subsection{Connected components of $\Bun_{P}$}\label{ssec:conncompBun}
We have a natural diagram
\[ \begin{tikzcd}
\Bun_{P} \arrow[r,"\mf{p}_{P}"] \arrow[d,"\mf{q}_{P}"] \arrow[dd,bend left=60,"\mf{q}^{\dagger}_{P}"] & \Bun_{G} \\
\Bun_{L} \arrow[d,"\ol{\mf{q}}_{L}"] &  \\
\Bun_{L^{\mathrm{ab}}} & 
\end{tikzcd} \] 
given by the corresponding maps of reductive groups.

By \cite[Corollary~IV.1.23]{FaScGeomLLC}, we have decompositions
\[ \Bun_{L} = \bigsqcup_{\theta \in B(L)_{\basic}} \Bun_{L}^{(\theta)} \quad \text{and} \quad \Bun_{L^{\mathrm{ab}}} = \bigsqcup_{\ol{\theta} \in \Lcoinv} \Bun_{L^{\mathrm{ab}}}^{\ol{\theta}} \]
into connected components via the Kottwitz invariant. We let $\Bun_{P}^{(\theta)}$ denote the fiber of $\mf{q}_{P}$ over the connected component $\Bun_{L}^{(\theta)}$, and write $\mf{q}_{P}^{(\theta)}$ (resp. $\mf{p}_{P}^{(\theta)}$) for the base change (resp. restriction)
of $\mf{q}_{P}$ (resp. $\mf{p}_{P}$) to this connected component. To describe these components in maximum generality, we first establish some lemmas.

\begin{lem}\label{lem:BunPPad}
Consider a parabolic $P \subset G$ with Levi factor $L$, its image defines a parabolic $P^{\mathrm{ad}} \subset G^{\mathrm{ad}}$ with Levi factor $L^{\mathrm{ad}}$. The natural diagram 
\begin{equation*}{\label{eqn: adjointgroupCartdiag}} \begin{tikzcd}
\Bun_{P} \arrow[r] \arrow[d,"\mf{q}_{P}"]  & \Bun_{P^{\mathrm{ad}}} \arrow[d,"\mf{q}_{P^{\mathrm{ad}}}"] \\
 \Bun_{L} \arrow[r] & \Bun_{L^{\mathrm{ad}}} 
\end{tikzcd} 
\end{equation*}
is Cartesian. 
\end{lem}
\begin{proof}
This follows from the fact that the fibers of $\mf{q}_{P}$ over an $L$-bundle $\mathcal{F}_{L}$ can be described in terms of torsors under $\mathcal{F}_{L} \times^{L} U$, where $L$ acts through the group $L^{\mathrm{ad}}$.
\end{proof}

For $b \in P(\breve{F})$ and $S \in \Perf_k$, we define $P$-bundle $\ca{E}_{P,b}$ on $X_S$ by 
$(\varphi \times b \sigma )\backslash (Y_S \times P)$.

\begin{lem}\label{lem:BunPPb}
Let $b \in G(\breve{F})$ be a basic element. Let $P \subset G$ and $Q_b \subset G_b$ 
be parabolic subgroups which correspond under the natural isomorphism 
$G_{\breve{F}} \cong G_{b,\breve{F}}$. 
\begin{enumerate}
\item 
We have $b \in P(\breve{F})$. 
\item\label{enu:isomPPb}
We have an isomorphism
\[ \Bun_{P} \xrightarrow{\sim} \Bun_{Q_{b}} ;\ \mathcal{E} \mapsto \mathcal{I}\mathrm{som} (\mathcal{E}_{P,b},\mathcal{E}) \]
of $v$-stacks. This gives the commutative diagram 
\begin{equation*} \begin{tikzcd}
\Bun_{P} \arrow[r,"\sim"] \arrow[d,"\mf{q}_{P}"]  & \Bun_{Q_b} \arrow[d,"\mf{q}_{Q_b}"]  \\
 \Bun_{L} \arrow[r,"\sim"] & \Bun_{L_b}.  
\end{tikzcd} 
\end{equation*}
\end{enumerate}

\end{lem}
\begin{proof}
Since 
$P_{\breve{F}} \subset G_{\breve{F}}$ is stable under $x \mapsto b \sigma (x) b^{-1}$, we have $b \in N_{G_{\breve{F}}}(P_{\breve{F}})=P_{\breve{F}}$. Then we can check the claim \ref{enu:isomPPb} easily (\cf the proof of \cite[Corollary~23.3.2]{ScWeBLp} when $P=G$). 
\end{proof}

We now have our key structural result on the components $\Bun_{P}^{(\theta)}$.
\begin{prop}{\label{prop: qissmooth}}
The map $\mf{q}_{P}$ is an $\ell$-cohomologically smooth (non-representable) map of Artin $v$-stacks with connected fibers. In particular, if $\theta \in B(L)_{\basic}$, then the map $\mf{q}_{P}^{(\theta)}$ is $\ell$-cohomologically smooth of pure $\ell$-dimension 
equal to $d_{P,\theta}$. 
\end{prop}
\begin{proof}
The first part follows from \cite[Proposition~3.16]{HamJacCrit} and its proof. The second part follows from combining \cite[Lemma~3.8, Proposition~4.7]{HamJacCrit} 
and \cite[Theorem~IV.1.19]{FaScGeomLLC} when $G$ is quasi-split. The general case is reduced to the quasi-split case using Lemma \ref{lem:BunPPad} and Lemma \ref{lem:BunPPb}. 
\end{proof}
In particular, we deduce the following from this. 
\begin{cor}{\label{cor: conncompsBunP}}
We have a decomposition into connected components
\[ \Bun_{P} = \bigsqcup_{\theta \in B(L)_{\basic}} \Bun_{P}^{(\theta)} . \]
\end{cor}
\begin{proof}
This follows by combining with the previous Proposition together with the fact that $\ell$-cohomologically smooth maps are universally open \cite[Proposition~23.1]{SchEtdia}\footnote{Strictly speaking, this only applies to representable morphisms; however, it is easy to see the proof extends to this context.}.  
\end{proof}
We consider the semistable locus $\bigsqcup_{\theta \in B(L)_{\basic}} [\ast/\ul{L_{\theta}(F)}] \cong \Bun_{L}^{\mathrm{ss}} \subset \Bun_{L}$, and write $\Bun_{P}^{\mathrm{ss}}$ for the preimage along $\mf{q}_{P}$. We have decompositions into connected components
\[ \bigsqcup_{\theta \in B(L)_{\basic}} \mf{q}_{P}^{\theta}\colon \Bun_{P}^{\mathrm{ss}} \cong \bigsqcup_{\theta} \Bun_{P}^{\theta} \ra \bigsqcup_{\theta} \Bun_{L}^{\theta}, \]
where $\mf{q}_{P}^{\theta}$ is the base change of $\mf{q}_{P}$ to the HN-strata $\Bun_{L}^{\theta} \hookrightarrow \Bun_{L}$. Ultimately, the proof of our main Theorem will reduce to an explicit description of the connected components $\Bun_{P}^{\theta}$ of $\Bun_{P}^{\mathrm{ss}}$ in terms of Banach--Colmez spaces. Those connected components are the preimage of the basic HN-strata. 

\subsection{Description of the preimage HN-strata}

Let $\theta \in B(L)$. 
We take $\{ U_i \}_{0 \leq i \leq n}$ and $\{ \alpha_i \}_{0 \leq i \leq n-1}$ as in Lemma \ref{lem:filtU} where $S = \Spec{(F)}$. 
We note that $\Lie (U_{i}/U_{i+1}) \cong \Lie (U)_{\alpha_i}$. 
Let $\mathcal{L}ie (U)_{\alpha_i,\theta}$ be the vector bundle on the Fargues--Fontaine curve with Kottwitz element given by the image of $\theta$ under $L \to \GL (\Lie (U)_{\alpha_i})$. Recalling the minus sign when passing between isocrystals and $G$-bundles, we note that this has degree equal to $-\langle \alpha_{i}, \theta \rangle$.

We write $\Bun_{P}^{\theta}$ for the preimage of the HN-strata $\Bun_{L}^{\theta}$ along $\mf{q}_P$. 
In order to describe $\Bun_{P}^{\theta}$, for $S \in \AffPerf_k$, we consider the torsors 
\[ \mathcal{Q}_{\theta,S} := \mathcal{E}_{\theta,S}^{\mathrm{alg}} \times^{L} P  \] 
and 
\[ R_{\mathrm{u}}\mathcal{Q}_{\theta,S} := \mathcal{E}_{\theta,S}^{\mathrm{alg}} \times^{L} U, \]
which we view as group schemes over $X_S^{\mathrm{alg}}$.

We now come to the key description of the space $\Bun_{P}^{\theta}$.
\begin{lem}{\label{lem:geomdescpofconncomp}}
Let $\Bun_{P,*}^{\theta}$ 
denote the fiber of the map 
\[\mf{q}_{P}^{\theta}\colon \Bun^{\theta}_{P} \ra \Bun_{L}^{\theta} \cong [\ast/\wt{L}_{\theta}] \]
 over $\ast \ra [\ast/\wt{L}_{\theta}]$. Then the map $\Bun_{P,*}^{\theta} \ra \ast$ is an iterated fibration of the absolute Picard $v$-groupoids
\begin{equation*}{\label{eqn: LieBC space}}
 \ca{P}(\mcLie (U)_{\alpha_i,\theta}[1]) \ra \ast 
\end{equation*}
for $0 \leq i \leq n-1$.
\end{lem}
\begin{proof}
The fiber $\Bun_{P,*}^{\theta}$ will parametrize, for $S \in \AffPerf_k$, $P$-bundles on $X_{S}^{\alg}$ whose underlying $L$-bundle is equal to $\mathcal{E}_{\theta,S}^{\alg}$. This can be expressed as the set of torsors under the group scheme $R_{\mathrm{u}}(\mathcal{Q}_{\theta,S}) =\mathcal{E}_{\theta,S}^{\alg} \times^{L} U$ 
on $X_{S}^{\alg}$. Then the claim follows from 
Lemma \ref{lem:filtU} and the fact that, for a vector bundle $\mathcal{E}^{\mathrm{alg}}$ on $X_{S}^{\alg}$, the set of $\mathcal{E}^{\mathrm{alg}}$-torsors is parametrized by $H^{1}(X_{S}^{\alg},\mathcal{E}^{\mathrm{alg}})$, and the automorphisms of such $\mathcal{E}^{\mathrm{alg}}$-torsors is parametrized by $H^{0}(X_{S}^{\alg},\mathcal{E}^{\mathrm{alg}})$. Here we have implicitly used GAGA for the Fargues--Fontaine curve \cite[Proposition~II.2.7]{FaScGeomLLC}.  
\end{proof} 

\begin{cor}\label{cor:qPidentify}
Let $b \in G(\breve{F})$ be the image of $\theta$. 
We assume the conjugate action of $\nu_{\theta}$ on $\Lie (U)$ has non-positive weights. 
Then $\mf{q}_P^{\theta} \colon \Bun^{\theta}_{P} \ra \Bun_{L}^{\theta}$ identifies with the natural map $[*/\wt{G}_b] \to [\ast/\wt{L}_{\theta}]$. 
\end{cor}
\begin{proof}
This follows from the proof of Lemma \ref{lem:geomdescpofconncomp}, since 
$H^1 (X_S, \mcLie (U)_{\alpha_i,\theta})=0$ by the assumption. 
\end{proof}

We now prove the following proposition using Lemma \ref{lem:geomdescpofconncomp}. 
\begin{prop}\label{prop:KBunPtheta}
We have an isomorphism 
\[ K_{\Bun_{P}^{\theta}} \cong \mf{q}_{P}^{\theta *}(\delta_{P,\theta})[2d_{P,\theta}] \]
in $D(\Bun_{P}^{\theta})$. 
\end{prop}
\begin{proof}
The adjoint action of $P$ on $\Lie (U_{i}/U_{i+1})$ induces 
\[
P \to \Aut \left( \bigwedge^{\dim (U_{i}/U_{i+1})} \Lie (U_{i}/U_{i+1}) \right) \cong \Gm, 
\] 
which factors through $L^{\ab}$ and define $\xi_{P,i} \in \mathbb{X}^{*}(L^{\mathrm{ab}}_{\ol{F}})^{\Gamma_{F}}$ to be the corresponding character. 
We put $d_{P,\theta,i}= \langle \xi_{P,i}, \ol{\theta} \rangle$ and 
\[
 \delta_{P,\theta,i} \colon L_{\theta} (F) \to (L^{\ab})_{\theta^{\ab}} (F) = L^{\ab} (F) \xrightarrow{\ol{\delta}_{P,i}} \Lambda^{\times} ,
\] 
where $\ol{\delta}_{P,i}(x)=\lvert \xi_{P,i} (x)\rvert$. 
By Lemma \ref{lem:geomdescpofconncomp}, it suffices to show that the dualizing complex of $\ca{P}(\mcLie (U)_{\alpha_i,\theta}[1])$ is isomorphic to 
\[ \Lambda [2d_{P,\theta,i}] \]
with the action of $L_{\theta}$ given by $\delta_{P,\theta,i}$. 

We consider the natural map
\[ L \ra \GL (\Lie (U)_{\alpha_i}) ;\ l  \mapsto ( u \mapsto lul^{-1} ) \]
given by the left conjugation action, and the element $\theta_{\mathrm{Ad},i}$ defined as the image of $\theta$ under the induced map $B(L) \ra B(\GL (\Lie (U)_{\alpha_i}))$. We have 
\[ \mathcal{E}_{\theta_{\mathrm{Ad},i}} \cong \mcLie (U)_{\alpha_i,\theta} .\]
The right action of $L_{\theta}(F)$ on $\Bun_{P,*}^{\theta}$ factors through the natural map
\[ L_{\theta}(F) \ra \GL (\Lie (U)_{\alpha_i})_{\theta_{\mathrm{Ad},i}}(F) \]
composed with the right action of $\GL (\Lie (U)_{\alpha_i})_{\theta_{\mathrm{Ad},i}}(F)$ on $\mathcal{P}(\mcLie (U)_{\alpha_i,\theta}[1])$ given by acting via automorphisms on $\mcLie (U)_{\alpha_i,\theta}$ and taking inverses. Hence, the claim follows from Proposition  \ref{prop: Picardvgroupfullaction}. 
\end{proof}

It is convenient to slightly reformulate this. We define the following.
\begin{defn}
We let $\ol{\Delta}_{P}$ be the sheaf on
\[ \Bun_{L^{\mathrm{ab}}} \cong \bigsqcup_{\ol{\theta} \in \Lcoinv} [\ast/\ul{L^{\mathrm{ab}}(F)}], \] 
whose value on each connected component is given by $\ol{\delta}_{P}$.
\end{defn} 
We let $\mf{q}_{P}^{\dagger,\mathrm{ss}}\colon \Bun_{P}^{\mathrm{ss}} \ra \Bun_{L^{\mathrm{ab}}}$ be the restriction of $\mf{q}_{P}^{\dagger}$ to the semistable locus. Recalling the definition of $\delta_{P,\theta}$, we can reformulate the previous Proposition as follows. 
\begin{cor}{\label{cor: mainthmsemistable}}
We have an isomorphism 
\[ K_{\Bun_{P}^{\mathrm{ss}}} \cong \mf{q}_{P}^{\dagger,\mathrm{ss}*}(\ol{\Delta}_{P})[2\dim(\Bun_{P})] \]
of objects in $D(\Bun_{P}^{\mathrm{ss}})$.
\end{cor}
We have the following lemma, where the maps $\mf{q}_{b}^{\pm}$ are as defined in the introduction.
\begin{lem}\label{lem:Mbad}
For $b \in G(\breve{F})$, 
the natural diagrams 
\begin{equation*}
\begin{tikzcd}
			\ca{M}_b \arrow[r] \arrow[d,"\mf{q}_{b}^{+}"]  & \ca{M}_{b^{\ad}} \arrow[d,"\mf{q}^{+}_{b^{\ad}}"] \\
			{[*/\ul{G_b (F)}]} \arrow[r] & {[*/\ul{G^{\ad}_{b^{\ad}} (F)}]}, 
		\end{tikzcd}  
  \quad
	\begin{tikzcd}
	{[*/\wt{G}_b]} \arrow[r] \arrow[d,"\mf{q}_{b}^{-}"]  & {[*/\wt{G}^{\ad}_{b^{\ad}} ]} \arrow[d,"\mf{q}^{-}_{b^{\ad}}"] \\
	{[*/\ul{G_b (F)}]} \arrow[r] & {[*/\ul{G^{\ad}_{b^{\ad}} (F)}]}  
\end{tikzcd} 
\end{equation*}
	are Cartesian. 
\end{lem}
\begin{proof}
Let $S \in \AffPerf_k$. 
Let $H$ be the inner twisting of $G \times X_S^{\alg}$ by $\ca{E}_{b,S}^{\alg}$ over $X_S^{\alg}$. 
Then we have $\wt{G}_b^{>0}(S)=H^{>0}(X_S^{\alg})$ as in the proof of \cite[Proposition III.5.1]{FaScGeomLLC}. Hence $G \to G^{\ad}$ induces an isomorphism $\wt{G}_b^{>0} \cong \wt{G}_{b^{\ad}}^{\ad,>0}$, since 
$H^{>0}$ is a unipotent radical of a parabolic subgroup $H^{\geq 0}$ of $H$ by \cite[Proposition III.5.2]{FaScGeomLLC}. Hence the first diagram is Cartesian. 

By the proof of \cite[Proposition V.3.5]{FaScGeomLLC}, 
the fiber of $\mf{q}^{+}_b$ is parametrized by $H^{<0}$-torsors, where 
$H^{<0}$ is the unipotent radical of a parabolic subgroup $H^{\leq 0}$ 
of $H$. Hence the second diagram is also Cartesian for a similar reason. 
\end{proof}

\begin{lem}\label{lem:Mbpure}
Let $b' \in G(\breve{F})$ be an basic element. We put $G'=G_{b'}$. Then 
$b \mapsto bb'^{-1}$ induces a bijection $B(G) \cong B(G')$. 
Furthermore, for $b \in G(\breve{F})$ we have isomorphisms of
the commutative diagrams 
	\begin{equation*} 
        \begin{tikzcd}
		\ca{M}_b \arrow[r,"\sim"] \arrow[d,"\mf{q}_{b}^{+}"]  & \ca{M}_{bb'^{-1}} \arrow[d,"\mf{q}_{bb'^{-1}}^{+}"] \\
		{[*/\ul{G_b (F)}]} \arrow[r,"\sim"] & {[*/\ul{G'_{bb'^{-1}} (F)}]} , 
	\end{tikzcd}  
        \quad 
	\begin{tikzcd}
		{[*/\wt{G}_b]} \arrow[r,"\sim"] \arrow[d,"\mf{q}_{b}^{-}"]  & {[*/\wt{G'}_{bb'^{-1}} ]} \arrow[d,"\mf{q}_{bb'^{-1}}^{-}"] \\
		{[*/\ul{G_b (F)}]} \arrow[r,"\sim"] & {[*/\ul{G'_{bb'^{-1}} (F)}]}.  
	\end{tikzcd} 
\end{equation*}
\end{lem}
\begin{proof}
The second bijection can be checked easily (\cf \cite[Lemma 2.2]{ImaConv}). 
We define $H$ as in the proof of Lemma \ref{lem:Mbad} using $G$ and $b$. We define $H'$ similarly using $G'$ and $bb'^{-1}$. Then we have a natural identification $H \cong H'$ by the construction. Hence, the isomorphisms in the diagram follow in the same way as the proof of Lemma \ref{lem:Mbad}. 
\end{proof}

We put $d_b=\langle 2\rho_{G}, \nu_{b} \rangle$. We have the following definition.
\begin{defn}{\label{defn: generaldefnofmoduluschar}}
We take a quasi-split inner form $G^{\ad,*}$ of $G^{\ad}$ and $b' \in G^{\ad,*}(\breve{F})$ such that $G^{\ad,*}_{b'} = G^{\ad}$. We put $b^*=b^{\ad}b'$. Then we have 
$G^{\ad,*}_{b^*}(F) \cong G^{\ad}_{b^{\ad}}(F)$ as in Lemma \ref{lem:Mbpure}. 
We fix a Borel pair of $G^{\ad,*}$. 
We define 
\[
 \delta_b \colon G_b(F) \to G^{\ad}_{b^{\ad}}(F) \cong G^{\ad,*}_{b^*}(F) \xrightarrow{\delta_{P^{b^*},b^*_L}} \Lambda^{\times}, 
\]
where $b^*_L$ is the image of $b^*$ in the Levi quotient of $P^{b^*}$. 
\end{defn}

\begin{prop}{\label{prop: moduluscharacterinnaturalsituations}}
We have isomorphisms 
\begin{align*}
 K_{\mathcal{M}_{b}} \cong \mf{q}_{b}^{+*}(\delta_{b})[2d_b] , \quad K_{[\ast/\wt{G}_{b}]} \cong \mf{q}_{b}^{-*}(\delta_{b}^{-1})[-2d_b]. 
\end{align*}
\end{prop}
\begin{proof}
By Lemma \ref{lem:Mbad} and Lemma \ref{lem:Mbpure}, we may assume that 
$G$ is quasi-split. We fix a Borel pair of $G$. 
Then the first claim follows from 
Proposition \ref{prop:KBunPtheta} for $P^{b}$ by \cite[Example V.3.4]{FaScGeomLLC}. 
The second claim follows from Proposition \ref{prop:KBunPtheta} for $P^{b,-}$.
\end{proof}

\subsection{Local systems on $\Bun_{G}$ and reduction to the semistable locus}
For an Artin $v$-stack $Z$, we write $\Loc(Z) \subset D(Z)$ for the category of $\Lambda$-local systems on $Z$ with respect to the \'etale topology on $Z$. We write $j_{G}\colon \Bun_{G}^{\mathrm{ss}} \hookrightarrow \Bun_{G}$ for the open subspace defined by the semistable bundles. We recall that we have a decomposition 
\[ \bigsqcup_{b \in B(G)_{\basic}} \Bun_{G}^{b} \cong \Bun_{G}^{\mathrm{ss}}, \]
and that $\Bun_{G}^{b} \cong [\ast/\ul{G_{b}(F)}]$ is just the classifying stack attached to the $\sigma$-centralizer of $b$ for $b$ basic. 
The set $B(G)$ has a partial ordering, where $b \geq b'$ if $\nu_{b} - \nu_{b'}$ is a positive linear combination of coroots and $\kappa_{G}(b) = \kappa_{G}(b')$. 

We start by deducing some interesting consequences of Corollary \ref{cor: mainthmsemistable}.
\begin{prop}{\label{prop: Localsystemsextend}}
There is well-defined equivalence of categories:
\[ \widetilde{\Loc}(\Bun_{G}^{\mathrm{ss}}) \xrightarrow{\cong} \Loc(\Bun_{G}) \]
\[ A \mapsto R^{0}j_{G *}(A), \]
where $\widetilde{\Loc}(\Bun_{G}^{\mathrm{ss}}) \subset \Loc(\Bun_{G}^{\mathrm{ss}})$ denotes the subcategory of \'etale local systems which admit a trivialization induced from an \'etale covering of $\Bun_{G}$.
\end{prop}
\begin{proof}
The claim that this is well-defined will follow from the proof. In particular, we will show that $R^{0}j_{G*}(\Lambda)$ defines an object in $\Loc(\Bun_{G})$. 

Let $b \in B(G)$, we write $\Bun_{G}^{\leq b}$ (resp. $\Bun_{G}^{< b}$) for the open substack parametrizing $G$-bundles whose associated element in $B(G)$ is $\leq b$ (resp. $< b$) after pulling back to a geometric point. We write $j^{b} \colon \Bun_{G}^{< b} \hookrightarrow \Bun_{G}^{\leq b}$ for the open inclusion and $i^{b}\colon \Bun_{G}^{b} \hookrightarrow \Bun_{G}^{\leq b}$ for the complementary closed inclusion.

By induction and excision on $D(\Bun_{G})$ with respect to the HN-strata and its partial ordering, we can reduce to showing that, for $A \in \Loc(\Bun_{G}^{\leq b})$ and all $b$ non-basic, the natural adjunction map
\[ A \ra R^{0}j^{b}_*j^{b*}(A) \]
is an isomorphism. 
It suffices to check this $v$-locally and, since $A$ is a local system with respect to the \'etale topology, we can reduce to checking this for the constant sheaf on $\Bun_{G}$. Since the cone is supported on the closed substack defined by $i^{b}$, it suffices to show that the map
\[ \Lambda \ra i^{b*}R^{0}j^{b}_*(\Lambda)  \]
is an isomorphism for all $b \in B(G)$.

To do this, we recall the $\ell$-cohomologically smooth chart $\pi_{b}\colon \mathcal{M}_{b} \ra \Bun_{G}^{\leq b}$ considered in \cite[Section~V.3]{FaScGeomLLC}. This has a natural map $\mf{q}_{b}^{+}\colon \mathcal{M}_{b} \ra [\ast/\ul{G_{b}(F)}]$, as in Proposition \ref{prop: moduluscharacterinnaturalsituations}. The map $\mf{q}_{b}^{+}$ admits a closed section given by the split reduction. We write $\mathcal{M}_{b}^{\circ}$ for the open subspace obtained by removing the image of this section and $\wt{\mf{q}}^{+}_{b}\colon \wt{\mathcal{M}}_{b} \ra \ast$ for the base change of $\mf{q}^{+}_{b}$ along $\ast \ra [\ast/\ul{G_{b}(F)}]$, 
with its corresponding open subspace $\wt{\mathcal{M}}_{b}^{\circ}$. We write $(\wt{i}^{b},\wt{j}^{b})$ for the base change of $(i^{b},j^{b})$ along $\pi_{b}$. 

By applying smooth base change with respect to $\pi_{b}$, the $G_{b}(F)$-representation $i^{b*}R^{0}j^{b}_*(\Lambda)$ can be identified with
\[ \wt{i}^{b*}R^{0}\wt{j}^{b}_*(\Lambda),   \]
together with its $G_{b}(F)$ action. 

Applying \cite[Proposition~V.4.2]{FaScGeomLLC}, we can identify this with the complex
\[ R\Gamma(\wt{\mathcal{M}}_{b},R^{0}\wt{j}^{b}_*(\Lambda)) \cong R^{0}\Gamma(\wt{\mathcal{M}}^{\circ}_{b},\Lambda), \]
where we note that $\mathcal{M}_{b}^{\circ}$ is precisely the preimage of the open substack $\Bun_{G}^{< b} \subset \Bun_{G}^{\leq b}$ under $\pi_{b}$. Therefore, we are reduced to the following.
\begin{lem}{\label{lem: Mbcohomology}}
For $b$ non-basic, we have that  
\[ R^{i}\Gamma(\wt{\mathcal{M}}^{\circ}_{b},\Lambda) \cong 0 \]
unless $i \in \{0,2\langle 2\rho_{G},\nu_{b} \rangle - 1\}$. If $i = 0$ then it is isomorphic to the trivial representation, and if $i = 2\langle 2\rho_{G},\nu_{b} \rangle - 1$ then it is isomorphic to $\delta_{b}^{-1}$, as in Definition \ref{defn: generaldefnofmoduluschar}
\end{lem}
\begin{proof}
By Proposition \ref{prop: moduluscharacterinnaturalsituations}, it follows that we have an isomorphism
\[ \wt{\mf{q}}_{b}^{+!}(\Lambda) \cong \delta_{b}[2\langle 2\rho_{G},\nu_{b} \rangle], \]
using that the dualizing complex on $[\ast/G_{b}(F)]$ is the constant sheaf by \cite[Example~4.4.3]{HKWKotloc} and the unimodularity of $G_{b}$.

Since $\wt{\mathcal{M}}_{b}$ is an iterated fibration of negative Banach--Colmez spaces, we have an isomorphism  $\wt{\mf{q}}^{+}_{b!}\wt{\mf{q}}_{b}^{+!}(\Lambda) \cong \Lambda$ by Lemma \ref{lem:BCadjisom}. 
It follows that the compactly supported cohomology of $\wt{\mathcal{M}}_{b}$ is isomorphic to $\delta_{b}^{-1}[-2\langle 2\rho_{G},\nu_{b} \rangle]$.

We can apply excision with respect to the pair $(\wt{i}^{b},\wt{j}^{b})$, which is a $G_{b}(F)$-equivariant stratification. This tells us that $R\Gamma_{c}(\wt{\mathcal{M}}^{\circ}_{b},\Lambda)$ has cohomology as a $G_{b}(F)$-representation in degree $1$ isomorphic to the trivial representation, and cohomology in degree $2\langle 2\rho_{G},\nu_{b} \rangle$ given by $\delta_{b}^{-1}$. Using the description of the dualizing complex on $\wt{\mathcal{M}}_{b}$ provided above and in turn on the open subspace $\wt{\mathcal{M}}^{\circ}_{b}$, we obtain that the non-compactly supported cohomology of $\wt{\mathcal{M}}_{b}^{\circ}$ has cohomology in degree $0$ isomorphic to the trivial representation, and cohomology in degree $\langle 2\rho_{G},\nu_{b} \rangle - 1$ with $G_{b}(F)$-action given by $\delta_{b}^{-1}$, 
as desired.
\end{proof}
This completes the proof of Proposition \ref{prop: Localsystemsextend}. 
\end{proof}
From this, we can deduce the following. 
\begin{prop}{{\label{prop: dualcomplexforreductive}}}
We have an isomorphism 
\[ K_{\Bun_{G}} \cong \Lambda \]
of objects in $D(\Bun_{G})$.
\end{prop}
\begin{proof}
We know that $\Bun_{G}$ is $\ell$-cohomologically smooth of pure $\ell$-dimension $0$, by \cite[Theorem~IV.1.19]{FaScGeomLLC}. It follows that $K_{\Bun_{G}}$ is \'etale locally isomorphic to the constant sheaf. Therefore, we have that $K_{\Bun_{G}} \in \Loc(\Bun_{G})$ and, by the previous proposition, we have an isomorphism: 
\[ K_{\Bun_{G}} \cong R^{0}j_{G*}j_{G}^{*}(K_{\Bun_{G}}) \cong R^{0}j_{G*}(K_{\Bun_{G}^{\mathrm{ss}}}). \]
Note that it is easy to describe the dualizing complex on $\Bun_{G}^{\mathrm{ss}} \cong \bigsqcup_{b \in B(G)_{\basic}} [\ast/G_{b}(F)]$, since it is a disjoint union of classifying stacks and the groups $G_{b}$ are unimodular. In particular, by \cite[Example~4.4.3]{HKWKotloc}, we have an identification 
\[ K_{\Bun_{G}^{\mathrm{ss}}} \cong \Lambda, \]
after fixing a choice of Haar measure on the $G_{b}(F)$ for $b \in B(G)_{\basic}$. The claim immediately follows. 
\end{proof}
\begin{rem}{\label{rem: WildProof}}
For completeness, we also sketch the alternative proof of Proposition \ref{prop: Localsystemsextend} and Proposition \ref{prop: dualcomplexforreductive}, explained by Wild for $G = \GL_{2}$ \cite{WildBunGcohomology}. In particular, the key point is to show Proposition \ref{prop: dualcomplexforreductive} first, by using the easy calculation of the dualizing complex over the semistable locus explained above and then using induction on the partial ordering of $B(G)$, just as was done above. Similarly, the inductive step proceeds by using the excision sequence of the form
\[ \xrightarrow{+1} i_{*}^{b}K_{\Bun_{G}^{b}} \ra K_{\Bun_{G}^{\leq b}} \ra j_{*}^{< b}K_{\Bun_{G}^{< b}} \xrightarrow{+1}. \]
One then invokes the very soft claim that $K_{\Bun_{G}^{b}}$ is concentrated in degree $2\langle 2\rho_{G},\nu_{b} \rangle \geq 2$ for $b$ non-basic \cite[Theorem~I.4.1 (iv)]{FaScGeomLLC} to see that the map $K_{\Bun_{G}^{\leq b}}  \ra R^{0}j_{*}^{< b}(K_{\Bun_{G}^{< b}}) \simeq R^{0}j_{*}^{< b}(\Lambda)$ is an isomorphism in this case, where we use that $K_{\Bun_{G}}$ is \'etale locally known to be concentrated in degree $0$ by \cite[Theorem~I.4.1 (vii)]{FaScGeomLLC}. This similarly implies that the natural adjunction map $\Lambda \ra R^{0}j_{*}^{< b}(\Lambda) \simeq K_{\Bun_{G}^{\leq b}}$ is an isomorphism, since \'etale locally it is an isomorphism, implying the global isomorphism $K_{\Bun_{G}} \simeq \Lambda$ as desired. From here, Proposition \ref{prop: Localsystemsextend} will also follow, since as already explained above one can reduce from a general $A$ to the constant sheaf. 
\end{rem}

We now have the following.
\begin{thm}{\label{thm: main}}
We have an isomorphism 
\[ K_{\Bun_{P}} \cong \mf{q}_{P}^{\dagger*}(\ol{\Delta}_{P})[2\dim(\Bun_{P})] 
\]
of objects in $D(\Bun_{P})$.
\end{thm}
\begin{proof}
By Proposition \ref{prop: dualcomplexforreductive} and its proof, we have isomorphisms: 
\[ K_{\Bun_{L}} \cong \Lambda \cong R^{0}(j_{L*})(\Lambda). \]
We apply $\mf{q}_{P}^{!}$ to both sides of this relationship, giving us an isomorphism:
\[ K_{\Bun_{P}} \cong \mf{q}_{P}^{!}R^{0}(j_{L*})(\Lambda). \]
It suffices to restrict to a connected component indexed by $\theta \in B(L)_{\basic}$. The map $\mf{q}_{P}^{(\theta)}$ is pure of $\ell$-dimension equal to $d_{P,\theta}$ by Proposition \ref{prop: qissmooth}. This implies that we have an identification
\[ K_{\Bun_{P}^{(\theta)}} \cong R^{-2d_{P,\theta}}(\mf{q}_{P}^{(\theta)!})R^{0}i^{\theta}_{*}(\Lambda)[2d_{P,\theta}] \cong R^{0}(i_{P*}^{\theta})R^{-2d_{P,\theta}}(\mf{q}_{P}^{\theta!})(\Lambda)[2d_{P,\theta}], 
\]
where we have used proper base change for the last identification, and $i_{P}^{\theta}\colon \Bun_{P}^{\theta} \hookrightarrow \Bun_{P}$ denotes the base change of $i^{\theta}\colon \Bun_{L}^{\theta} \hookrightarrow \Bun_{L}$ along $\mf{q}_{P}$. 

Using Corollary  \ref{cor: mainthmsemistable} and the fact that the dualizing complex on $[\ast/\ul{L_{\theta}(F)}]$ is the constant sheaf since it is unimodular, we can rewrite the RHS of the previous identification as 
\[ R^{0}(i_{P*}^{\theta})\mf{q}_{P}^{\theta*}(\delta_{P,\theta})[2d_{P,\theta}]. 
\]
However, now by applying smooth base change with respect to $\mf{q}_{P}^{(\theta)}$ via Proposition \ref{prop: qissmooth}, we can rewrite this as
\[ \mf{q}_{P}^{(\theta)*}(R^{0}(i^{\theta}_{*})(\delta_{P,\theta}))[2d_{P,\theta}], \]
but, using Proposition \ref{prop: Localsystemsextend}\footnote{Note that we can trivialize the sheaf $\delta_{P,\theta}$ on $\Bun_{M}^{\theta}$ on an \'etale cover of the connect component $\Bun_{L^{\mathrm{ab}}}^{(\theta)}$ by pulling back an \'etale  cover of $\Bun_{L^{\mathrm{ab}}}^{\ol{\theta}} = [\ast/\ul{L^{\mathrm{ab}}(F)}]$ which trivializes the character $\ol{\delta}_{P,\theta}$.}, this is easily identified with 
\[ (\mf{q}_{P}^{(\theta)*} \circ \ol{\mf{q}}_{L}^{(\theta)*})(\ol{\delta}_{P})[2d_{P,\theta}] \cong \mf{q}_{P}^{\dagger,(\theta)*}(\ol{\delta}_{P})[2d_{P,\theta}], 
\]
as desired, where $\ol{\theta} \in \Lcoinv$ is the element induced by $\theta \in B(L)_{\basic}$, and $\ol{\mf{q}}_{L}^{(\theta)} \colon \Bun_{L}^{(\theta)} \ra \Bun_{L^{\mathrm{ab}}}^{\ol{\theta}}$ (resp. $\mf{q}_{P}^{\dagger,(\theta)} \colon  \Bun_{P}^{(\theta)} \ra \Bun_{L^{\mathrm{ab}}}^{\ol{\theta}}$) is the natural map induced by $\ol{\mf{q}}_{L}$ (resp. $\mf{q}_{P}^{\dagger}$).
\end{proof}

\section{Applications}\label{sec:App}
\subsection{Connection with other results}\label{ssec:connection}
We now discuss the relation of this paper with various other results in the literature. 

For $b \in B(G)$, the natural map 
\[ p_{b}\colon \Bun_{G}^{b} \cong [\ast/\wt{G}_{b}] \ra [\ast/\ul{G_{b}(F)}] \]
considered in \cite[Proposition~V.2.2]{FaScGeomLLC} defines via pullback an equivalence 
\[ 
 D(\Bun_{G}^{b}) \cong D(G_{b}(F),\Lambda), 
\] 
where $D(G_{b}(F),\Lambda)$ is the unbounded derived category of smooth $\Lambda$-representations. This identifies with the map $\mf{q}_{b}^{-}$ considered in Proposition \ref{prop: moduluscharacterinnaturalsituations}. 

\begin{prop}{\label{prop: Bunbdualcpx}}
For all $b \in B(G)$, we have an isomorphism 
\[ p_{b}^{!}(\Lambda) \cong p_{b}^{*}(\delta_{b}^{-1})[-2d_{b}] \]
of objects in $D(\Bun_{G}^{b})$, where we recall that $d_{b} := \langle 2\rho_{G},\nu_{b} \rangle$ and $\delta_{b}$ is as in Proposition \ref{prop: moduluscharacterinnaturalsituations}. 
\end{prop}
\begin{proof}
This follows from Proposition \ref{prop: moduluscharacterinnaturalsituations}
together with the unimodularity of $G_{b}(F)$ and \cite[Example~4.4.3]{HKWKotloc}. 
\end{proof}

The action of $G_{b}(F)$ on the dualizing complex of $\Bun_{G}^{b}$ describes the inverse of the character appearing in \cite[Theorem~1.15]{HamLeeTorsVan} and \cite[Theorem~7.1, Lemma~7.4]{KosOngenlg}, by \cite[Lemma~9.1]{KosOngenlg}. 
Therefore, we have the following corollary. 

\begin{cor}{\label{cor: moduluscharappearinginShimVars}}
The sheaf $\ol{\mathbb{F}}_{\ell}(d_{b})$ with $G_{b}(F)$-equivariant structure appearing in \cite[Theorem~7.1,Lemma~7.4]{KosOngenlg} and \cite[Theorem~1.15]{HamLeeTorsVan} is isomorphic to $\delta_{b}$ as a sheaf with $G_{b}(F)$-action. 
\end{cor}

Now we will briefly change the assumption on $\Lambda$, and let $\Lambda$ be a $\bb{Z}_{\ell}$-algebra until Proposition \ref{prop: compactsuppcohomU}. 
We use the notations $D_{\blacksquare}(-,\Lambda)$ and $Rf_{\natural}$ in \cite[VII]{FaScGeomLLC}. 
We have the following analogue of Lemma \ref {lem:BCadjisom}. 

\begin{lem}\label{lem:BCadjisomnat}
Let $f \colon S' \to S$ be a morphism of small v-stacks. We assume one of the following: 
\begin{enumerate}
\item $S \in \Perf_k$ and $\ca{E}$ is a vector bundle on $X_S$ that is everywhere of positive (resp. negative) slopes.  
\item $\ca{E}$ is a vector bundle on $X$ that is of positive (resp. negative) slopes. 
\end{enumerate}
If $f$ is a torsor under $\ca{H}^0(\ca{E})$ (resp. $\ca{H}^1(\ca{E})$), then 
\[
 f^* \colon D_{\blacksquare}(S,\Lambda) \to D_{\blacksquare}(S',\Lambda) 
\] 
is fully faithful and 
the adjunction morphism 
$Rf_{\natural} Rf^* A \to A$ is an isomorphism for $A \in D_{\blacksquare}(S,\Lambda)$. 
\end{lem}
\begin{proof}
This is proven in the same way as \cite[Proposition V.2.1]{FaScGeomLLC}, using \cite[Proposition VII.3.1]{FaScGeomLLC} and \cite[Lemma 1.3]{ImaConv}. 
\end{proof}

For an Artin v-stack $X$, let 
$D_{\lis}(X,\Lambda)$ be the category of lisse-\'etale $\Lambda$-sheaves defined in 
\cite[Definition VII.6.1]{FaScGeomLLC}. 
When we are working with $D_{\lis}(-,\Lambda)$, we consider pushforward functors in $D_{\lis}(-,\Lambda)$. 

For an $\ell$-cohomologically smooth morphism $f \colon X \to Y$ of Artin v-stacks, we put 
\[
 f^! \Lambda =\left( \varprojlim_{n} f^!(\bb{Z}/\ell^n \bb{Z})\right) \otimes_{\bb{Z}_{\ell}} \Lambda 
\]
which is an invertible object by the smoothness, and 
\[
 f_! (A) =f_{\natural} \left( A \otimes (f^! \Lambda)^{-1} \right) 
\]
for $A \in D_{\lis}(X,\Lambda)$ (\cf \cite[Definition VII.6.1]{FaScGeomLLC}). 

\begin{prop}{\label{prop: compactsuppcohom}}
Consider $f_b \colon \wt{G}_{b}^{>0} \to *$ with its right $G_{b}(F)$-action given by right conjugation. Then $f_{b !}\Lambda \cong \Lambda(\delta_{b}^{-1})[-2d_b]$ as a complex of sheaves with left $G_{b}(F)$-action. 
\end{prop}
\begin{proof}
Let 
\begin{equation*}
 [*/\ul{G_b(F)}] \stackrel{s_{b}}{\lra} [*/\wt{G}_{b}] \stackrel{p_{b}}{\lra} [*/\ul{G_b(F)}]
\end{equation*}
be the natural morphisms. 
By Proposition \ref{prop: Bunbdualcpx}, 
we have an isomorphism: 
\begin{equation}\label{eq:sb!torsion}
s_{b}^! (\bb{Z}/\ell^n\bb{Z}) \cong (\bb{Z}/\ell^n\bb{Z})(\delta_{b})[2d_{b}]. 
\end{equation}
Hence we have $f_b^! (\bb{Z}/\ell^n\bb{Z}) \cong (\bb{Z}/\ell^n\bb{Z})[2d_{b}]$, 
since the fiber of $s_b$ at $*$ gives $f_b$. 
By this, Lemma \ref{lem:BCadjisomnat} and \cite[Proposition VII.3.1]{FaScGeomLLC}, we have 
$f_{b !}\Lambda \cong \Lambda [-2d_b]$ as a complex of sheaves. 
Therefore the claim is reduced to the case where $\Lambda=\bb{Z}/\ell^n\bb{Z}$ by \cite[Proposition VII.3.1]{FaScGeomLLC}. 
In this case, \eqref{eq:sb!torsion} and Lemma \ref{lem:BCadjisom} imply that $s_{b !} \Lambda \cong p_{b}^* (\delta_{b}^{-1})[-2d_{b}]$. 
Hence we obtain the claim. 
\end{proof}

Let $\theta \in L(\breve{F})$, and let $b \in G(\breve{F})$ be the image of $\theta$. 
Assume that the conjugate action of $\nu_{\theta}$ on $\Lie (U)$ is non-positive. Under this assumption, we have 
$\wt{G}_{b}(S) =\mathcal{Q}_{\theta,S} (X_S^{\alg})$ for $S \in \AffPerf_k$. 
We put 
\[
 \wt{G}_{b,U}^{>0}(S) =(R_{\mathrm{u}} \mathcal{Q}_{\theta,S}) (X_S^{\alg}) 
\]
for $S \in \AffPerf_k$. 
Then we have $\wt{G}_{b}=\wt{G}_{b,U}^{>0} \rtimes \wt{L}_{\theta}$. 
We can view $\ul{G_{b}(F)}$ as a subgroup diamond of $\wt{L}_{\theta}$ via $\ul{G_{b}(F)} \cong \ul{L_{\theta}(F)} \subset \wt{L}_{\theta}$. 

\begin{prop}{\label{prop: compactsuppcohomU}}
Consider $f_{b,U} \colon \wt{G}_{b,U}^{>0} \to *$ with its right $G_{b}(F)$-action given by right conjugation. Then $f_{b,U !} \Lambda \cong \Lambda (\delta_{P,\theta})[2d_{P,\theta}]$ as a sheaf with left $G_{b}(F)$-action. 
\end{prop}
\begin{proof}
This is proved in the same way as Proposition \ref{prop: compactsuppcohom} using 
the natural morphisms 
\begin{equation*}
 [*/\wt{L}_{\theta}] \lra [*/\wt{G}_{b}] \lra [*/\wt{L}_{\theta}], 
\end{equation*}
Corollary \ref{cor:qPidentify} and Proposition \ref{prop:KBunPtheta}. 
\end{proof}

Proposition \ref{prop: compactsuppcohom} and Proposition \ref{prop: compactsuppcohomU} are compatible with \cite[Proposition~11.3]{HamGeomES}, as well as compatible with \cite[Lemma~4.18]{GINsemi}. 

\begin{cor}\label{cor:cptcohch}
The following is true. 
\begin{enumerate}
\item The character $\kappa$ appearing in \cite[Lemma~11.1]{HamGeomES} is isomorphic to $\delta_{b}^{-1}$. 
\item 
The character $\kappa$ appearing in \cite[Lemma~4.18 (2), Theorem~4.26]{GINsemi} is equal to $\delta_{P,\theta}$. 
\end{enumerate}
\end{cor}
\begin{proof}
This follows from Proposition \ref{prop: compactsuppcohom} and Proposition \ref{prop: compactsuppcohomU} respectively. 
\end{proof}

Proposition \ref{prop: Bunbdualcpx} also appeared in \cite[Proposition~1.14]{HanBeijLec}, where an argument was mentioned that used properties of the Fargues--Scholze local Langlands correspondence. We now fill in the details and show that those properties are compatible with what we have proved. 

\begin{proof}[Another proof of Proposition \ref{prop: Bunbdualcpx}]
Let $E$ be a finite extension of $\mathbb{Q}_{\ell}$ with a fixed choice of square root of $q$. 
In general, there does not exist a $!$-pushforward in the category of lisse-\'etale $E$-sheaves. However, if $i^{b}\colon \Bun_{G}^{b} \hookrightarrow \Bun_{G}$ is an inclusion of HN-strata, this follows from the discussion after \cite[Lemma~3.1]{ImaConv}. 

As in the proof of \cite[Proposition~V.2.2]{FaScGeomLLC}, we recall that the map $p_{b}$ admits a section $s_{b}$, which is an iterated fibration in positive Banach--Colmez spaces of total dimension $d_{b}$. Therefore, 
we have that $s_{b!}s_{b}^{!}(\Lambda) = \Lambda$ by Lemma \ref{lem:BCadjisom}, where $\Lambda$ is a torsion ring. As usual, it follows that the dualizing complex $p_{b}^{!}(\Lambda)$ is an invertible sheaf concentrated in degrees $2d_{b}$, and that the pullback functor  $p_{b}^{*}$ induces an equivalence of categories, as in \cite[Proposition~V.2.2]{FaScGeomLLC}. It follows that we have an isomorphism 
\[ p_{b}^{!}(\Lambda) \cong p_{b}^{*}(\kappa_{\Lambda})[-2d_{b}] \]
for some character $\kappa_{\Lambda}\colon G_{b}(F) \ra \Lambda^{*}$. It follows that $p_{b}^{!}(E) \in D_{\mathrm{lis}}(\Bun_{G}^{b},E)$ is given by $p_{b}^{*}(\kappa_{E})[-2d_{b}]$ for a smooth character $\kappa_{E}\colon G_{b}(F) \ra E^{*}$ (By the definition of upper $!$ given above.).  

We will now compute $\kappa_{E}$. Let us consider some smooth irreducible representation $\rho$ of $G_{b}(F)$ on $E$-vector spaces with Fargues--Scholze parameter $\phi\colon W_{F} \ra \phantom{}^{L}G_{b}(E)$ valued in the $L$-group of $G_{b}$\footnote{For our argument, it is important that such a $\rho$ always exists, but of course we could always take the trivial representation.}. 
By the definition of $G_b$, we have an embedding $(G_b)_{\breve{F}} \hookrightarrow G_{\breve{F}}$, where $(G_b)_{\breve{F}}$ is the centralizer of $\nu_b$. This gives $(G_b)_{\widehat{\ol{F}}} \hookrightarrow G_{\widehat{\ol{F}}}$. We take a Borel pair $(B,T)$ of $G_{\widehat{\ol{F}}}$ such that $\nu_b$ factors through $T$ and $\nu_b$ is anti-dominant with respect to $B$. Then $(B\cap (G_b)_{\widehat{\ol{F}}}),T)$ is a Borel pair of $(G_b)_{\widehat{\ol{F}}}$. This induces an embedding $\eta_b \colon \widehat{G}_b \hookrightarrow \widehat{G}$, which is independent of the choice of $(B,T)$. 
This gives the embedding $\eta_{b}\colon \phantom{}^{L}G_{b} \ra \phantom{}^{L}G$.

Let $i^{b}\colon \Bun_{G}^{b} \hookrightarrow \Bun_{G}$ be the inclusion of the HN-strata. We have the following lemma.
\begin{lem}{\label{lemma: *pushforwardshavecorrectFSparameter}}
For $\Lambda$ a $\bb{Z}_{\ell}$-algebra and $\rho \in D(G_{b}(F),\Lambda)$ irreducible with Fargues--Scholze parameter $\phi \colon W_{F} \ra \phantom{}^{L}G_{b}$, 
the objects
\[ i_{*}^{b}p_{b}^{*}(\rho) \text{ , } i_{!}^{b}p_{b}^{*}(\rho) \]
are Schur-irreducible and their Fargues--Scholze parameters are equal to the composition
\[   
W_{F} \xrightarrow{\phi} \phantom{}^{L}G_{b}(\Lambda) \xrightarrow{\eta_{b}^{\mathrm{tw}}} \phantom{}^{L}G(\Lambda), \]
where $\eta_{b}^{\mathrm{tw}}$ is the twisted embedding. Explicitly, this is 
\[ w \mapsto \eta_{b} \circ ((2\widehat{\rho}_{G} - 2\widehat{\rho}_{G_{b}})(\sqrt{q})^{|w|} \phi(w) ), \]
where $|\cdot|\colon W_{F} \ra \mathbb{Z}$ is the unramified character sending geometric Frobenius to $1$. 
\end{lem}
\begin{proof}
The Schur irreducibility follows easily from the isomorphisms 
\[ R^{0}\Hom(i_{b*}p_{b}^{*}(\rho),i_{b*}p_{b}^{*}(\rho)) \simeq R^{0}\Hom(p_{b}^{*}\rho,p_{b}^{*}\rho) \simeq R^{0}\Hom(i_{b!}p_{b}^{*}(\rho),i_{b!}p_{b}^{*}(\rho)) \]
coming from adjunction, and the irreducibility of $\rho$ (Recalling that $p_{b}^{*}$ is an equivalence of categories). More precisely, the isomorphism between the left most term and right most term is induced from a non-zero map 
\[ i_{b!}p_{b}^{*}(\rho) \ra i_{b*}p_{b}^{*}(\rho), \]
in $D_{\lis}(\Bun_{G},\Lambda)$ coming from adjunction, which induces the ring isomorphism
\[ \Lambda \simeq \mathrm{End}(i_{!}^{b}p_{b}^{*}(\rho)) \xrightarrow{\simeq} \mathrm{End}(i_{*}^{b}p_{b}^{*}(\rho)) \simeq \Lambda \]
when passing to the endomorphism algebras. Now, since the action of the excursion algebra factors through a map to endomorphisms of the identity functor on $D_{\lis}(\Bun_{G},\Lambda)$ \cite[Theorem~IX.5.2]{FaScGeomLLC}, the excursion action commutes with this morphism. Therefore, the ideal corresponding to $\phi$ in the excursion algebra determined by $i_{!}^{b}p_{b}^{*}(\rho)$ agrees with the ideal in the excursion algebra corresponding to $i_{*}^{b}p_{b}^{*}(\rho)$, and in turn we deduce that the Fargues--Scholze parameters of these two agree. The desired description of the Fargues--Scholze parameter for $i_{!}^{b}p_{b}^{*}(\rho)$ and in turn $i_{*}^{b}p_{b}^{*}(\rho)$ now follows from \cite[Theorem~IX.7.2]{FaScGeomLLC}.
\end{proof}

We note that we have $\eta_{b}^{\mathrm{tw}}(-) = \eta_{b}(-) \otimes \ol{\delta}_{b}^{1/2}$, where $\ol{\delta}_{b}$ is the character of $W_{F}$ attached to $\delta_{b}$ under local class field theory. 

Now, it follows by the cohomological smoothness of $p_{b}$ that we have an isomorphism $\mathbb{D}_{\Bun_{G}}(i^{b}_{!}p_{b}^{*}(\rho)) \cong i^{b}_{*}p_{b}^{!}(\rho^{*}) \cong i^{b}_{*}p_{b}^{*}(\rho^{*} \otimes \kappa_{E})$, where $\mathbb{D}_{\Bun_{G}}$ denotes Verdier duality on $\Bun_{G}$ and we have used \cite[Lemma~3.6]{ImaConv} for the first isomorphism. By the compatibility of the Fargues--Scholze correspondence with smooth duality \cite[Proposition~IX.5.3]{FaScGeomLLC}, 
character twists, and local class field theory, as well as Lemma \ref{lemma: *pushforwardshavecorrectFSparameter}, it follows that the Fargues--Scholze parameter of $i^{b}_{*}p_{b}^{*}(\rho^{*} \otimes \kappa_{E})$ is given by $\eta_{b}^{\mathrm{tw}}
\circ (\phi^{\vee_{G_{b}}} \otimes \ol{\kappa}_{E})$, where $\ol{\kappa}_{E}$ is the image of $\kappa_{E}$ under the geometric normalization of local class field theory, and $(-)^{\vee_{G_{b}}}$ is given by the Chevalley involution on $\widehat{G}_{b}$. On the other hand, by combining Lemma \ref{lemma: *pushforwardshavecorrectFSparameter} with compatibility of the Fargues--Scholze correspondence with Verdier duality (Combine \cite[Theorem~IX.2.2]{FaScGeomLLC} with Proposition \ref{prop: dualcomplexforreductive} after extending $E$ as necessary so that we may use \cite[Proposition~VI.12.1]{FaScGeomLLC}), 
it must also be given by $(\eta_{b}^{\mathrm{tw}} \circ \phi)^{\vee_{G}}$, where $\vee_{G}$ is induced by the Chevalley involution on $\widehat{G}$. In summary, we have an equality 
\[ (\eta_{b}^{\mathrm{tw}} \circ \phi)^{\vee_{G}} = \eta_{b}^{\mathrm{tw}}
\circ (\phi^{\vee_{G_{b}}} \otimes \ol{\kappa}_{E}) \]
of conjugacy classes of parameters. In particular, substituting in the relationship $\eta_{b}^{\mathrm{tw}}(-) = \eta_{b}(-) \otimes \ol{\delta}_{b}^{1/2}$, we obtain that
\[ ((\eta_{b} \circ \phi) \otimes \ol{\delta}_{b}^{1/2}))^{\vee_{G}} = (\eta_{b}
\circ \phi^{\vee_{G_{b}}}) \otimes \ol{\kappa}_{E} \otimes \ol{\delta}_{b}^{1/2} \]
and the LHS identifies with $(\eta_{b} \circ \phi^{\vee_{G_{b}}}) \otimes \ol{\delta}_{b}^{-1/2}$. From here, we conclude that $\kappa_{E} = \delta_{b}^{-1}$, as desired. Now, to handle the case of a torsion coefficient system $\Lambda$, we can use that $\kappa_{E}$ is the base change of a $\hat{\mathbb{Z}}^{p}$-valued character $\kappa_{\hat{\mathbb{Z}}^{p}}$, and that $\kappa_{\Lambda} = \kappa_{\hat{\mathbb{Z}}^{p}} \otimes_{\hat{\mathbb{Z}}^{p}} \Lambda$, by definition of $p_{b}^{!}(E)$.   
\end{proof}

\subsection{Harris--Viehmann conjecture}\label{ssec:HVconj}

We show the Harris--Viehmann conjecture in the Hodge--Newton reducible case using our results and results in \cite{GINsemi}. 
In this subsection, we assume that $\Lambda$ is a $\bb{Z}_{\ell}$-algebra that admits a fixed choice of square root of $q$. 

For a cocharacter $\mu$ of $G$, let $E_{\mu}$ denote the reflex field of $\mu$. 
We take a Borel pair $(B,T)$ of the quasi-split inner form of $G$. Let $B^{-}$ be the opposite Borel. Let $\hat{\mu} \in \bb{X}^*(\widehat{T})^{+,W_{E_{\mu}}}$ be the character determined by $\mu$ (\cf \cite[(1.1.3) Lemma]{KotShtw}). 
We put $\rho_{\mu,G}=\Ind^{\wh{G}}_{\wh{B}^{-}} \hat{\mu}$. 
Let $r'_{\mu,G}$ be the image of the natural morphism $(\rho_{\mu,G}^*)^{\sigma} \to \rho_{\mu,G}$, where $\sigma$ is the Chevalley involution of $\wh{G}$. 
We assume that $r_{\mu,G}'$ is projective over $\Lambda$. 
This assumption is automatically satisfied if $\Lambda =\ol{\bb{Q}}_{\ell}, \ol{\bb{Z}}_{\ell}, \ol{\bb{F}}_{\ell}$ or $\mu$ is minuscule. 

For a finite separable extension $E$ of $F$ in $\ol{F}$, we put $E_{\mu}'=E\cdot E_{\mu}$ and extend 
$r'_{\mu,G}$ to a representation 
$r'_{\mu,G,E}$ of $\wh{G} \rtimes W_{E_{\mu}'}$ as in \cite[(2.1.2) Lemma]{KotShtw}. We put 
$r_{\mu,G,E}=\Ind^{\wh{G} \rtimes W_E}_{\wh{G} \rtimes W_{E_{\mu}'}} r'_{\mu,G,E}$. 

Let $\Div_E^1$ be the moduli of Cartier divisor of degree $1$ on the Fargues--Fontaine curve over $E$. We put 
\[
 \Hck_{G,E} = \Hck_G \times_{\Div_F^1} \Div_E^1. 
\]
We write $\ca{S}_{\mu,G,E}$ for the sheaf on $\Hck_{G,E}$ determined from $r_{\mu,G,E}$ via the geometric Satake equivalence. 
We put $\ca{S}'_{\mu,G,E}=\bb{D}(\ca{S}_{\mu,G,E})^{\vee}$ as in \cite[IX.2]{FaScGeomLLC}. 
We consider the diagram 
\begin{equation*}
\begin{tikzcd}
	\Bun_G \arrow[d,equal]  
 & \Hck_{G,E}  \arrow[d] 
 \arrow[l,"h_{1,G,E}"'] \arrow[r,"h_{2,G,E}"] & \Bun_G \times \Div_E^1 \arrow[d,"p_{E/F}"] 
 \\ 
 	\Bun_G   & \Hck_G  \arrow[l,"h_{1,G}"'] \arrow[r,"h_{2,G}"] & \Bun_G \times \Div_F^1 . 
\end{tikzcd}  
\end{equation*}
We put $T_{\mu,G,E}(A)=h_{2,G,E \natural}(  h_{1,G,E}^* (A) \otimes \ca{S}'_{\mu,G,E})$. 
We simply write $T_{\mu,G}$ for $T_{\mu,G,F}$, 
and omit $G$ from the notation if it is clear from the context. This applies also for other related notations. 

\begin{lem}\label{lem:HecTorus}
Let $T$ be a torus over $F$, $\mu$ a cocharacter of $T$, $E$ a finite separable extension of $F$, and $\chi$ a character of $T(F)$. 
We have an isomorphism: $T_{\mu,T,E}(\chi) \cong \chi \otimes (r_{\mu,T,E} \circ \varphi_{\chi}|_{W_{E}})$. 
\end{lem}
\begin{proof}
Under the natural morphism $\Hck_T \times_{\Div_F^1} \times \Div_E^1 \to \Hck_T \times_{\Div_F^1} \times \Div_{E \cap E_{\mu}}^1$, the pullback of $\ca{S}'_{\mu,T,E \cap E_{\mu}}$ is $\ca{S}'_{\mu,T,E}$. 
Hence the claim is reduced to the case of $E \cap E_{\mu}$. So we may assume that $E \subset E_{\mu}$. 

We consider the diagram
\begin{equation*}
\begin{tikzcd}
	\Bun_{T_E} \arrow[d]  & \Hck_{T_E}^{\mu}  \arrow[d,"p_{\mathrm{H},T_E/T}"] \arrow[l,"h_{1,T_E}"'] \arrow[r,"h_{2,T_E}"] & \Bun_{T_E} \times \Div_E^1 \arrow[d,"p_{T_E/T}"] \\ 
 \Bun_T \arrow[d,equal]  & \Hck_{T,E}^{\mu} \arrow[d,"p_{\mathrm{H},E/F}"] \arrow[l,"h_{1,T,E}"'] \arrow[r,"h_{2,T,E}"] & \Bun_T \times \Div_E^1 \arrow[d,"p_{E/F}"] \\ 
 	\Bun_T   & \Hck_T^{\mu}  \arrow[l,"h_{1,T}"'] \arrow[r,"h_{2,T}"] & \Bun_T \times \Div_F^1 , 
\end{tikzcd}  
\end{equation*}
where the right two squares are Cartesian. 
By $p_{\mathrm{H},T_E/T}^* (\ca{S}'_{\mu,T,E}) \cong \ca{S}'_{\mu,T_E}$, we have 
\begin{align*}
    p_{T_E/T}^* (T_{\mu,T,E}(\chi)) \cong T_{\mu,T_E} (\chi \circ N_{T(E)/T(F)}) \cong (\chi \circ N_{T(E)/T(F)}) \otimes (r_{\mu,T_E} \circ \varphi_{\chi}|_{W_{E}}). 
\end{align*}
Hence, we have $T_{\mu,T,E}(\chi) \cong \chi' \otimes (r_{\mu,T,E} \circ \varphi_{\chi}|_{W_{E}})$ for some character $\chi'$ of $T(F)$. 
Since $E \subset E_{\mu}$, we have $r_{\mu,T}=\Ind_{\wh{T} \rtimes W_E}^{\wh{T} \rtimes W_F} r_{\mu,T,E}$. This implies $\ca{S}'_{\mu,T} \cong p_{\mathrm{H},E/F \natural}(\ca{S}'_{\mu,T,E})$. 
Hence we have 
$p_{E/F \natural} (T_{\mu,T,E}(\chi)) \cong T_{\mu,T}(\chi) \cong \chi \otimes (r_{\mu,T} \circ \varphi_{\chi})$. 
Therefore, we obtain $\chi' \simeq \chi$. 
\end{proof}

\begin{lem}\label{lem:projflc}
Let $f\colon X \to Y$ be a morphism of Artin v-stacks, 
$A \in D_{\lis} (X,\Lambda)$ and $B \in D_{\lis} (Y,\Lambda)$. Assume that $B$ is $\ell$-cohomologically smooth locally constant. 
Then we have $(f_* A) \otimes B \cong f_* (A \otimes f^* B)$. 
\end{lem}
\begin{proof}
Since we have a natural morphism $(f_* A) \otimes B \to f_* (A \otimes f^* B)$, it suffices to show that this is an isomorphism $\ell$-cohomologically smooth locally. We can check this by the assumption on $B$ and \cite[Lemma 1.6]{ImaConv}. 
\end{proof}

\begin{lem}\label{lem:ibuppshrast}
Let $b \in B(G)$ and 
$A,B \in D_{\lis} (\Bun_G,\Lambda)$. Assume that $B$ is $\ell$-cohomologically smooth locally constant. Then we have $i^{b!}(A \otimes B) \cong i^{b!}A \otimes i^{b*}B$. 
\end{lem}
\begin{proof}
We note that $i^{b!} \cong i^{b*}(\mathrm{fib}(\id \to Rj^b_*j^{b*}))$ by \cite[Lemma 3.8]{ImaConv}. Hence the claim follows from isomorphisms 
\begin{align*}
i^{b*}Rj^b_*j^{b*}(A \otimes B) &\cong i^{b*}Rj^b_* (j^{b*}A \otimes j^{b*} B) \\ 
&\cong i^{b*} ((Rj^b_* j^{b*}A) \otimes B) \cong (i^{b*} Rj^b_* j^{b*}A) \otimes i^{b*}B, 
\end{align*}
where we use Lemma \ref{lem:projflc} at the second isomorphism. 
\end{proof}

Let $b,b' \in B(G)$ and $\mu \in X_*(G)$. 
We define the cohomology of moduli spaces of shtukas 
$R\Gamma_{\mathrm{c}}(\Sht_{G,b,b',K'}^{\mu})$ and 
$R\Gamma_{\mathrm{c}}(\Sht_{G,b,b'}^{\mu})$, as in \cite[\S 3]{ImaConv}. 

For a character $\chi$ of $G^{\ab}(F)$, we write $\chi_b$ for the character of $G_b(F)$ determined by $\chi$ and $G_b(F) \to G^{\ab}_{b^{\ab}}(F)=G^{\ab}(F)$. 

\begin{prop}\label{prop:cohtwist}
Let $\rho$ be a smooth representation of $G_b(F)$, and 
$\chi$ a character of $G^{\ab}(F)$. 
Let $\chi_{\mu}$ be the character of $W_{E_{\mu}}$ determined by the L-parameter of $\chi$ and $\mu$. 
We have 
\begin{align*}
\varinjlim_{K' \subset G_{b'}(F)} 
R\Hom_{G_b(F)} (R\Gamma_{\mathrm{c}}&(\Sht_{G,b,b',K'}^{\mu}),\rho \otimes \chi_b) \\ &\cong 
\varinjlim_{K' \subset G_{b'}(F)} R\Hom_{G_b(F)} (R\Gamma_{\mathrm{c}}(\Sht_{G,b,b',K'}^{\mu}),\rho ) \otimes \chi_{b'} \otimes \chi_{\mu} 
\end{align*}
in the derived category of $G_{b'}(F) \times W_{E_{\mu}}$-representations, where $G_{b'}(F)$ acts smoothly and $W_{E_{\mu}}$ acts continuously.
\end{prop}
\begin{proof}
We put $G_0=G \times G^{\ab}$. Then the natural homomorphism $G \to G_0$ and the projection $G_0 \to G$ gives the diagram 
\begin{equation}\label{eq:HckGG'}
\begin{tikzcd}
	\Bun_G \arrow[d,"\pi_{\mathrm{B}}"]  & \Hck_G^{\leq \mu} \arrow[d,"\pi_{\mathrm{H}}"] \arrow[l,"h_1"'] \arrow[r,"h_2"] & \Bun_G \times \Div_F^1 \arrow[d,"\pi"] \\ 
	\Bun_{G_0} & \Hck_{G_0}^{\leq \mu_0} \arrow[d,"\pi_{0,\mathrm{H}}"] \arrow[l,"h_{0,1}"'] \arrow[r,"h_{0,2}"] & \Bun_{G_0} \times \Div_F^1 \arrow[d,"\pi_0"] \\ 
 & \Hck_G^{\leq \mu}  \arrow[r,"h_2"] & \Bun_G \times \Div_F^1 . 
\end{tikzcd}  
\end{equation}
The lower right square in \eqref{eq:HckGG'} is Cartesian because the natural map 
\[
\Hck_{G^{\ab}}^{\leq \mu^{\ab}} \to \Bun_{G^{\ab}} \times \Div_F^1 
\]
is an isomorphism. Hence the upper right square is also Cartesian since $\pi_0 \circ \pi$ and $\pi_{0,\mathrm{H}} \circ \pi_{\mathrm{H}}$ are the identity. 

Let $\mu_0$ and $\mu^{\ab}$ be the cocharacters of $G_0$ and $G^{\ab}$ determined by $\mu$. 
Then we have 
\begin{align}\label{eq:Tmumu0}
T_{\mu}(\pi_{\mathrm{B}}^* A_0)\cong h_{2\natural} ((\pi_{\mathrm{H}}^* h_{0,1}^* A_0) \otimes \ca{S}'_{\mu})\cong h_{2\natural} (\pi_{\mathrm{H}}^* (h_{0,1}^* A_0 \otimes \ca{S}'_{\mu_0})) 
\cong \pi^* T_{\mu_0} (A_0) 
\end{align}
using \cite[Proposition VII.3.1]{FaScGeomLLC}. 
We have the diagram 
\begin{equation}\label{eq:G0dec}
\begin{tikzcd}
	\Bun_{G_0} \arrow[d,"\sim" sloped] & \Hck_{G_0,E_{\mu}} \arrow[d] \arrow[l,"h_{0,1,E_{\mu}}"'] \arrow[r,"h_{0,2,E_{\mu}}"] & \Bun_{G_0} \times \Div_{E_{\mu}}^1 \arrow[d,"\Delta_{\mathrm{D}}"] \\ 
	\Bun_G \times \Bun_{G^{\ab}} & \Hck_G \times \Hck_{G^{\ab},E_{\mu}} \arrow[l,"h_1 \times h^{\ab}_{1,E_{\mu}}"'] \arrow[r,"h_2 \times h^{\ab}_{2,E_{\mu}}"] & \Bun_G \times \Div_F^1 \times \Bun_{G^{\ab}} \times \Div_{E_{\mu}}^1,
\end{tikzcd} 
\end{equation}
where $\Delta_{\mathrm{D}}$ is induced by the diagonal morphism $\Div_{E_{\mu}}^1 \to \Div_{E_{\mu}}^1 \times \Div_{E_{\mu}}^1$ together with the projection $\Div_{E_{\mu}}^1 \to \Div_{F}^1$ and the other morphisms are natural ones. 
The right square in \eqref{eq:G0dec} is Cartesian. 
Let 
\begin{equation*}
\begin{tikzcd}
	\Hck_{G_0}  & \Hck_{G_0,E_{\mu}} \arrow[d,"p_{\mathrm{H},G,E_{\mu}/E}"] \arrow[l,"p_{\mathrm{H},E_{\mu}/E}"'] \arrow[r,"p_{\mathrm{H},G^{ab}}"] & \Hck_{G^{\ab},E_{\mu}} \\ 
	& \Hck_G  & 
\end{tikzcd} 
\end{equation*}
be the natural projections. Then we have 
\begin{equation*}
p_{\mathrm{H},E_{\mu}/E}^* (\ca{S}'_{\mu_0,G_0}) \cong p_{\mathrm{H},G,E_{\mu}/E}^* (\ca{S}'_{\mu,G}) \otimes p_{\mathrm{H},G^{ab}}^* (\ca{S}'_{\mu^{\ab},G^{\ab},E_{\mu}} ) 
\end{equation*}
since \begin{align*}
r_{\mu_0,G_0}|_{\wh{G}_0 \rtimes W_{E_{\mu}}} &\cong 
( \Ind_{\wh{G}_0 \rtimes W_{E_{\mu}}}^{\wh{G}_0 \rtimes W_{F}} r_{\mu_0,G_0,E_{\mu}} )|_{\wh{G}_0 \rtimes W_{E_{\mu}}}\\ 
&\cong ( \Ind_{\wh{G}_0 \rtimes W_{E_{\mu}}}^{\wh{G}_0 \rtimes W_{F}} (r_{\mu,G,E_{\mu}} \otimes r_{\mu^{\ab},G^{\ab},E_{\mu}}) )|_{\wh{G}_0 \rtimes W_{E_{\mu}}} \\ 
&\cong 
( \Ind_{\wh{G}_0 \rtimes W_{E_{\mu}}}^{\wh{G}_0 \rtimes W_{F}} r_{\mu,G,E_{\mu}})|_{\wh{G}_0 \rtimes W_{E_{\mu}}}  \otimes r_{\mu^{\ab},G^{\ab},E_{\mu}} \\ 
&\cong r_{\mu,G}|_{\wh{G}_0 \rtimes W_{E_{\mu}}}  \otimes r_{\mu^{\ab},G^{\ab},E_{\mu}} . 
\end{align*}
Therefore, we have 
\begin{equation}\label{eq:Tmumuab}
 p_{E_{\mu}/F}^* T_{\mu_0}(A \boxtimes A') \cong \Delta_{\mathrm{D}}^* (T_{\mu}(A) \boxtimes T_{\mu^{\ab},E_{\mu}}(A') ) 
\end{equation}
using \cite[Proposition VII.3.1]{FaScGeomLLC}. 
By Lemma \ref{lem:HecTorus}, we have 
\begin{equation}\label{eq:Tmuab}
i^{b'^{\ab}*}T_{\mu^{\ab},E_{\mu}}i^{b^{\ab}}_* [\chi_b]\cong \chi_{b'} \otimes \chi_{\mu}. 
\end{equation}
Let $b_0$ and $b_0'$ be the image of $b$ and $b'$ in $G_0$. 
We show that 
\begin{equation}\label{eq:ib*rhochi}
 i^{b}_* [\rho \otimes \chi_b] \cong 
 \pi_{\mathrm{B}}^* i^{b_0}_* [\rho \boxtimes \chi] . 
\end{equation}
We take a compact open subgroup $K$ of $G^{\ab}(F)$ such that $\chi$ is trivial on $K$. 
We put 
\[
 \Bun_{G}^{b^{ab}} := (\Bun_G \times \Bun_{G^{\ab}}^{b^{\ab}}) \times_{\Bun_{G_0}} \Bun_G 
\]
and consider the diagaram 
\begin{equation}\label{eq:BunGbbab}
\begin{tikzcd}
	\Bun_G^b \arrow[d,"\pi_{\mathrm{B}}^b"] \arrow[r,"i^{b,b^{\ab}}"] & \Bun_G^{b^{\ab}} \arrow[d,"\pi_{\mathrm{B}}^{b^{\ab}}"] \\ 
 	\Bun_G^b \times \Bun_{G^{\ab}}^{b^{\ab}} \arrow[d,"\pi_{K}^b"] \arrow[r,"i^{b_0,b^{\ab}}"] & \Bun_G \times \Bun_{G^{\ab}}^{b^{\ab}} \arrow[d,"\pi_{K}^{b^{\ab}}"] \\  
   	\Bun_G^b \times [*/(\ul{G^{\ab}(F)/K})] \arrow[r,"i^{b_0}_K"] & \Bun_G \times [*/(\ul{G^{\ab}(F)/K})]  
\end{tikzcd} 
\end{equation}
given by the natural morphisms. Both squares in \eqref{eq:BunGbbab} are Cartesian. In the Cartesian squares 
\begin{equation*}
\begin{tikzcd}
	* \times_{[*/(\ul{G^{\ab}(F)/K})]} \Bun_G^{b^{\ab}} \arrow[d,"\pi'"] \arrow[r] & \Bun_G^{b^{\ab}} \arrow[d,"\pi_{K}^{b^{\ab}} \circ \pi_{\mathrm{B}}^{b^{\ab}}"] \\ 
 	\Bun_G \arrow[d] \arrow[r] & \Bun_G \times [*/(\ul{G^{\ab}(F)/K})] \arrow[d] \\ 
   	* \arrow[r] & {[*/(\ul{G^{\ab}(F)/K})]}, 
\end{tikzcd} 
\end{equation*}
the morphism $\pi'$ is equal to 
\[
 * \times_{[*/(\ul{G^{\ab}(F)/K})]} \Bun_G^{b^{\ab}} \lra \Bun_G^{b^{\ab}} \hookrightarrow \Bun_G, 
\]
where the second morphism is the natural open and closed immersion. 
Hence $\pi_{K}^{b^{\ab}} \circ \pi_{\mathrm{B}}^{b^{\ab}}$ is $\ell$-cohomologically smooth since $\ast \ra [\ast/(\ul{G^{\ab}(F)/K})]$ is $\ell$-cohomologically smooth. 
Since $\Bun_G \times \Bun_{G^{\ab}}^{b^{\ab}} \hookrightarrow \Bun_{G_0}$ is an open and closed immersion, to show \eqref{eq:ib*rhochi}, it suffices to show 
\begin{equation*}
 i^{b,b^{\ab}}_* \pi_{\mathrm{B}}^{b *} [\rho \boxtimes \chi] \cong 
 \pi_{\mathrm{B}}^{b^{\ab} *} i^{b_0,b^{\ab}}_*  [\rho \boxtimes \chi].
\end{equation*} 
Let $[\chi]_K$ be the sheaf on $[*/(\ul{G^{\ab}(F)/K})]$ determined by $\chi$. 
Then we have 
\begin{align*}
  i^{b,b^{\ab}}_* \pi_{\mathrm{B}}^{b *} [\rho \boxtimes \chi] &\cong 
  i^{b,b^{\ab}}_* (\pi_K^b \circ \pi_{\mathrm{B}}^{b})^* ([\rho] \boxtimes [\chi]_K) \\
  &\cong (\pi_{K}^{b^{\ab}} \circ \pi_{\mathrm{B}}^{b^{\ab}})^* i_{K *}^{b_0}([\rho] \boxtimes [\chi]_K) \\ 
  &\cong (\pi_{K}^{b^{\ab}} \circ \pi_{\mathrm{B}}^{b^{\ab}})^* (i_{K *}^{b_0}([\rho] \boxtimes [1]_K) \otimes ([1] \boxtimes [\chi]_K))\\ 
  &\cong (\pi_{K}^{b^{\ab}} \circ \pi_{\mathrm{B}}^{b^{\ab}})^* (i^b_*[\rho] \boxtimes [\chi]_K) \\ 
  &\cong \pi_{\mathrm{B}}^{b^{\ab} *} (i^b_*[\rho] \boxtimes [\chi])  
  \cong 
 \pi_{\mathrm{B}}^{b^{\ab} *} i^{b_0,b^{\ab}}_*  [\rho \boxtimes \chi], 
\end{align*}
where we use \cite[Lemma 1.6]{ImaConv} at the second isomorphism, Lemma \ref{lem:projflc} at the third isomorphism, the $\ell$-cohomologically smoothness of $[\ast/(\ul{G^{\ab}(F)/K})] \ra \ast$ and \cite[Lemma 1.6]{ImaConv} at the fourth isomorphism. 
Thus we have proved \eqref{eq:ib*rhochi}. 
Let 
\[
 \pi_{E_{\mu}} \colon \Bun_G \times \Div_{E_{\mu}}^1 \to \Bun_{G_0} \times \Div_{E_{\mu}}^1 
\] 
be the base change of $\pi$ and 
\[
 \pi_{\ab,E_{\mu}} \colon \Bun_G \times \Div_{E_{\mu}}^1 \to \Bun_{G^{\ab}} \times \Div_{E_{\mu}}^1 
\] 
the natural morphism. 
Then we have 
\begin{align*}
\varinjlim_{K' \subset G_{b'}(F)} 
R\Hom_{G_b(F)} &(R\Gamma_{\mathrm{c}}(\Sht_{G,b,b',K'}^{\mu}),\rho \otimes \chi_b) \cong 
i^{b'!} p_{E_{\mu}/F}^* T_{\mu} i^{b}_* [\rho \otimes \chi_b] (\delta_{b'})[2d_{b'}] \\ 
&\cong i^{b' !} p_{E_{\mu}/F}^* T_{\mu} \pi_{\mathrm{B}}^* i^{b_0}_* [\rho \boxtimes \chi] (\delta_{b'})[2d_{b'}]\\ &\cong i^{b' !} p_{E_{\mu}/F}^* \pi^* T_{\mu_0}  i^{b_0}_* [\rho \boxtimes \chi] (\delta_{b'})[2d_{b'}]\\  &\cong i^{b' !} \pi_{E_{\mu}}^* \Delta_{\mathrm{D}}^* ( T_{\mu} i^{b}_* [\rho] \boxtimes T_{\mu^{\ab},E_{\mu}}i^{b^{\ab}}_* [\chi])  (\delta_{b'})[2d_{b'}]\\ 
&\cong (i^{b' !} p_{E_{\mu}/F}^* T_{\mu} i^{b}_* [\rho]) (\delta_{b'})[2d_{b'}]\otimes i^{b' *} \pi_{\ab,E_{\mu}}^* T_{\mu^{\ab},E_{\mu}}i^{b^{\ab}}_* [\chi] \\ 
&\cong (i^{b' !} p_{E_{\mu}/F}^* T_{\mu} i^{b}_* [\rho]) (\delta_{b'})[2d_{b'}]\otimes \chi_{b'} \otimes \chi_{\mu} \\ 
&\cong 
\varinjlim_{K' \subset G_{b'}(F)} R\Hom_{G_b(F)} (R\Gamma_{\mathrm{c}}(\Sht_{G,b,b',K'}^{\mu}),\rho ) \otimes \chi_{b'} \otimes \chi_{\mu}, 
\end{align*}
where we use Proposition \ref{prop: compactsuppcohom} at the first and last isomorphism, and 
\eqref{eq:ib*rhochi}, \eqref{eq:Tmumu0}, \eqref{eq:Tmumuab}, Lemma \ref{lem:ibuppshrast} and \eqref{eq:Tmuab} at the second, third, fourth, fifth and sixth isomorphisms respectively. 
\end{proof}

Assume that $F$ is $p$-adic, $G$ is quasi-split, 
$([b],[b'],\mu)$ is Hodge--Newton reducible (\cite[Definition 4.5]{GINsemi}) for $L$ with reductions $[\theta],[\theta'] \in B(L)$ of $b$ and $b'$ to $B(L)$, 
and $b'$ is basic. 
We assume that $L \supset L^b$ and $P=LP^{b,-}$. 
We sometimes view $\delta_{P,\theta} \colon L_{\theta}(F) \to \Lambda^{\times}$ as a character of $G_b(F)$ via $G_b(F) \cong L_{\theta}(F)$. 
Let $P'$ be the parabolic subgroup of $G_{b'}$ corresponding to $P \subset G$. We note that the Levi subgroup of $P'$ corresponding to $L \subset P$ is $L_{\theta'}$. 
Let $\mathring{\Sht}_{G,b,b'}^{\mu}$ be the open subspace of $\Sht_{G,b,b'}^{\mu}$ defined by the condition that the meromorphy of the modification is equal to $\mu$. 

\begin{prop}\label{prop:RGammaparab}
We have an isomorphism 
\[
R\Gamma_{\mathrm{c}}(\mathring{\Sht}_{G,b,b'}^{\mu}) \cong 
\Ind_{P'(F)}^{G_{b'}(F)} R\Gamma_{\mathrm{c}}(\mathring{\Sht}_{L,\theta,\theta'}^{\mu}) \otimes \delta_{P,\theta} \otimes || \cdot ||^{-\frac{d_{P,\theta}}{2}}[d_{P,\theta}]
\]
which is compatible with the actions of $G_b(F) \cong L_{\theta}(F)$ and $W_{E_{\mu}}$. 
\end{prop}
\begin{proof}
This follows from \cite[Corollary 4.14, Proposition 4.24]{GINsemi} and Corollary \ref{cor:cptcohch}, as in the proof of \cite[Theorem 4.26]{GINsemi}. In particular, we note that we have isomorphisms
\[ \varinjlim_{K' \subset G_{b'}(F)} R\Gamma_{\mathrm{c}}(\mathring{\Sht}_{G,b,b',K'}^{\mu},\Lambda[d_{\mu}](\frac{d_{\mu}}{2})) \simeq R\Gamma_{c}(\mathring{\Sht}_{G,b,b'}^{\mu}), \]
since the restriction of the sheaf $\mathcal{S}_{\mu,G,E}$ to the locally closed substack of $\Hck_{G,\mu,E}$ parametrizing modifications of meromorphy exactly equal to $\mu$ identifies with $\Lambda[d_{\mu}](\frac{d_{\mu}}{2})$, by definition (cf. \cite[Proposition~VI.7.5]{FaScGeomLLC}). Here $d_{\mu} := \langle 2\rho_{G}, \mu \rangle$, where $\rho_{G}$ is the half-sum of all positive roots. The LHS is in turn isomorphic to 
\[ \varinjlim_{K' \subset G_{b'}(F)} R\Gamma_{\mathrm{c}}(\mathring{\Sht}_{G,b,b',K'}^{\mu},\Lambda)[d_{\mu}] \otimes || \cdot ||^{\frac{d_{\mu}}{2}}.  \]
Similarly, we have an identification
\[ \varinjlim_{K' \subset L_{\theta'}(F)} R\Gamma_{\mathrm{c}}(\mathring{\Sht}_{L,\theta,\theta',K'}^{\mu},\Lambda)[d_{\mu_L}] \otimes || \cdot ||^{\frac{d_{\mu_{L}}}{2}}  \simeq R\Gamma_{\mathrm{c}}(\mathring{\Sht}_{L,\theta,\theta'}^{\mu}). \]
Then, using the identity $d_{P,\theta} = d_{\mu_{L}} - d_{\mu}$, we deduce the claimed relationship from the analysis in \cite{GINsemi}.
\end{proof}

\begin{thm}{\label{thm: HarrisViehmannConj}}
Assume that $\mu$ is minuscule. 
Let $\rho$ be a smooth representation of $G_b(F)$. 
We view $\rho$ also as a representation of $L_{\theta}(F)$ via the canonical isomorphism $G_b \cong L_{\theta}$. 
We have 
\begin{align*}
\varinjlim_{K' \subset G_{b'}(F)} &
R\Hom_{G_b(F)}(R\Gamma_{\mathrm{c}}(\Sht_{G,b,b',K'}^{\mu}),\rho)\\ & \cong 
\Ind_{P'(F)}^{G_{b'}(F)} \varinjlim_{K' \subset L_{\theta'}(F)} R\Hom_{L_{\theta}(F)} (R\Gamma_{\mathrm{c}}(\Sht_{L,\theta,\theta',K'}^{\mu}),\rho) \otimes || \cdot ||^{-\frac{d_{P,\theta}}{2}}[-d_{P,\theta}] 
\end{align*}
as objects in the derived category of $G_{b}(F) \times W_{E_{\mu}}$-representations, where $G_{b}(F)$ acts smoothly and $W_{E_{\mu}}$ acts continuously.
In particular, we have 
\begin{align*}
\sum_{i,j}&(-1)^{i+j} [\varinjlim_{K' \subset G_{b'}(F)} \Ext^j_{G_b(F)} (H_{\mathrm{c}}^{i}(\Sht_{G,b,b',K'}^{\mu},\Lambda),\rho)] \\ 
&=   
\Ind_{P'(F)}^{G_{b'}(F)} \sum_{i,j}(-1)^{i+j} [\varinjlim_{K' \subset L_{\theta'}(F)} \Ext^j_{L_{\theta}(F)} (H_{\mathrm{c}}^i(\Sht_{L,\theta,\theta',K'}^{\mu},\Lambda),\rho) \otimes || \cdot ||^{-d_{P,\theta}}] 
\end{align*}
in the Grothendieck group of finite length admissible representations. 
\end{thm}
\begin{proof}
We have 
\begin{align*}
\varinjlim_{K' \subset G_{b'}(F)} &
R\Hom_{G_b(F)}(R\Gamma_{\mathrm{c}}(\Sht_{G,b,b',K'}^{\mu}),\rho) \\ 
&\cong  
R\Hom_{G_b(F)}(\varinjlim_{K' \subset G_{b'}(F)} R\Gamma_{\mathrm{c}}(\Sht_{G,b,b',K'}^{\mu}),\rho)^{\mathrm{sm}} \\ 
&\cong  
R\Hom_{G_b(F)}(R\Gamma_{\mathrm{c}}(\Sht_{G,b,b'}^{\mu}),\rho)^{\mathrm{sm}} \\ 
&\cong  
R\Hom_{L_{\theta}(F)}(\Ind_{P'(F)}^{G_{b'}(F)} R\Gamma_{\mathrm{c}}(\Sht_{L,\theta,\theta'}^{\mu}) \otimes \delta_{P,\theta} \otimes || \cdot ||^{-\frac{d_{P,\theta}}{2}}[d_{P,\theta}],\rho)^{\mathrm{sm}} \\ 
&\cong R\Hom_{L_{\theta}(F)} (\Ind_{P'(F)}^{G_{b'}(F)}  R\Gamma_{\mathrm{c}}(\Sht_{L,\theta,\theta'}^{\mu}) ,\rho \otimes \delta_{P,\theta}^{-1})^{\mathrm{sm}} \otimes || \cdot ||^{\frac{d_{P,\theta}}{2}}[-d_{P,\theta}]\\ 
&\cong  
\Ind_{P'(F)}^{G_{b'}(F)} R\Hom_{L_{\theta}(F)}(R\Gamma_{\mathrm{c}}(\Sht_{L,\theta,\theta'}^{\mu}) ,\rho \otimes \delta_{P,\theta}^{-1})^{\mathrm{sm}} \otimes \delta_{P'} \otimes || \cdot ||^{\frac{d_{P,\theta}}{2}}[-d_{P,\theta}],
\end{align*}
where we use Proposition \ref{prop:RGammaparab} at the third isomorphism. 
Further we have 
\begin{align*}
&R\Hom_{L_{\theta}(F)}(R\Gamma_{\mathrm{c}}(\Sht_{L,\theta,\theta'}^{\mu}) ,\rho \otimes \delta_{P,\theta}^{-1})^{\mathrm{sm}} \otimes \delta_{P'} \otimes || \cdot ||^{\frac{d_{P,\theta}}{2}}[-d_{P,\theta}]\\ 
&\cong  
R\Hom_{L_{\theta}(F)}(\varinjlim_{K' \subset L_{\theta'}(F)} R\Gamma_{\mathrm{c}}(\Sht_{L,\theta,\theta',K'}^{\mu}) ,\rho \otimes \delta_{P,\theta}^{-1})^{\mathrm{sm}} \otimes \delta_{P'} \otimes || \cdot ||^{\frac{d_{P,\theta}}{2}}[-d_{P,\theta}] \\ 
&\cong  
\varinjlim_{K' \subset L_{\theta'}(F)} R\Hom_{L_{\theta}(F)}( R\Gamma_{\mathrm{c}}(\Sht_{L,\theta,\theta',K'}^{\mu}) ,\rho \otimes \delta_{P,\theta}^{-1}) \otimes \delta_{P'} \otimes || \cdot ||^{\frac{d_{P,\theta}}{2}}[-d_{P,\theta}] \\ 
&\cong  
\varinjlim_{K' \subset L_{\theta'}(F)} R\Hom_{L_{\theta}(F)}( R\Gamma_{\mathrm{c}}(\Sht_{L,\theta,\theta',K'}^{\mu}) ,\rho) 
\otimes ( \delta_{P,\theta'}^{-1} \otimes \delta_{P'}) 
\otimes || \cdot ||^{\frac{d_{P,\theta}}{2}-\langle \xi_P, \mu^{\ab} \rangle } [-d_{P,\theta}], 
\end{align*}
where we use Proposition \ref{prop:cohtwist} at the last isomorphism. 
We have $\delta_{P,\theta'}= \delta_{P'}$ as representations of $L_{\theta'}(F)$ because 
$(L_{\theta'})^{\ab}(F) \cong L^{\ab}(F) \xrightarrow{\ol{\delta}_P} \Lambda^{\times}$ is equal to $\ol{\delta}_{P'}$. 
We have also $d_{P,\theta}=\langle \xi_P, \bar{\theta} \rangle=\langle \xi_P, \mu^{\ab} \rangle$. 
Therefore we obtain the claim. 
\end{proof}
\begin{rem}
In the above claim, we used the notation $R\Hom_{G_{b}(F)}(-,-)^{\mathrm{sm}}$ to indicate that this is coming from an external Hom in the category of smooth representations of $G_{b}(F)$. For the more homotopically inclined, a precise definition of this can be given as follows. We consider the $\infty$-categorical enhancement of the derived category of smooth $G_{b}(F) \times G_{b'}(F)$-representations $\mathcal{D}(G_{b}(F) \times G_{b'}(F),\Lambda)$, which is a presentable stable $\Lambda$-linear $\infty$-category. Since the categories $\mathcal{D}(G_{b}(F),\Lambda)$ and $\mathcal{D}(G_{b'}(F),\Lambda)$ of smooth representations are compactly generated, we have an isomorphism 
\[ \mathcal{D}(G_{b}(F) \times G_{b'}(F),\Lambda) \simeq \mathcal{D}(G_{b}(F),\Lambda) \otimes_{\mathcal{D}(\Lambda)} \mathcal{D}(G_{b'}(F),\Lambda) \]
(see \cite[Remark~V.7.3]{FaScGeomLLC} and the proof of \cite[Proposition~V.7.2]{FaScGeomLLC} for details). Here $\mathcal{D}(\Lambda)$ is the $\infty$-categorical enhancement of the derived category of $\Lambda$-modules and $\otimes$ denotes the Lurie tensor product. We consider the functor
\[ \mathcal{D}(G_{b}(F) \times G_{b'}(F),\Lambda)^{\mathrm{op}} \times \mathcal{D}(G_{b}(F) \times G_{b'}(F),\Lambda) \ra \mathcal{D}(G_{b}(F) \times G_{b'}(F),\Lambda)  \]
given by internal Hom in the category of smooth $G_{b}(F) \times G_{b'}(F)$-representations. We compose this with the functor 
\[ \mathcal{D}(G_{b}(F),\Lambda) \otimes \mathcal{D}(G_{b'}(F),\Lambda) \ra \mathcal{D}(G_{b'}(F),\Lambda)    \]
given by tensoring the global section functor $\mathcal{D}(G_{b}(F),\Lambda) \ra \mathcal{D}(\Lambda)$ by $\mathcal{D}(G_{b'}(F),\Lambda)$. For $A,B \in \mathcal{D}(G_{b}(F) \times G_{b'}(F),\Lambda)$, the value of this functor on $(A,B)$ is $R\Hom_{G_{b}(F)}(A,B)^{\mathrm{sm}}$, where above we regarded $\rho$ as a smooth $G_{b}(F) \times G_{b'}(F)$-representation by inflating along the natural projection map $G_{b}(F) \times G_{b'}(F) \ra G_{b}(F)$. To be even more precise, by viewing the above categories as condensed $\infty$-categories $\mathcal{C}$, we can also form the set of objects with a continuous $W_{E_{\mu}}$-action denoted $\mathcal{C}^{BW_{E_{\mu}}}$, as  in \cite[Section~IX.1]{FaScGeomLLC}. Then one can similarly see that one also gets an induced functor
\[ \mathcal{D}(G_{b}(F) \times G_{b'}(F),\Lambda)^{\mathrm{op},BW_{E_{\mu}}} \times \mathcal{D}(G_{b}(F) \times G_{b'}(F),\Lambda)^{BW_{E_{\mu}}} \ra \mathcal{D}(G_{b'}(F),\Lambda)^{BW_{E_{\mu}}}, \]
which shows that the above operation also respects the continuous $W_{E_{\mu}}$-action.
\end{rem}

\subsection{Geometric Eisenstein series}{\label{sec: Eisensteinseries}}
We again consider the diagram 
\[ \begin{tikzcd}
\Bun_{P} \arrow[r,"\mf{p}_{P}"] \arrow[d,"\mf{q}_{P}"]  & \Bun_{G} \\
\Bun_{L}, & 
\end{tikzcd} \] 
and attach to it the unnormalized Eisenstein functor 
\[ \EisP^{G}(-)  := \mf{p}_{P!}\mf{q}_{P}^{*}(-)[\dim(\Bun_{P})]\colon D(\Bun_{L}) \ra D(\Bun_{G}), \]
where $\dim(\Bun_{P})\colon |\Bun_{P}| \ra \mathbb{Z}$ records the $\ell$-cohomological dimension of the cohomologically smooth $v$-stack $\Bun_{P}$ over the base. We saw in Proposition \ref{prop: qissmooth} that this is constant on each connected component indexed by $\theta \in B(L)_{\basic}$ with value equal to $d_{P,\theta} = \langle \xi_{P},\ol{\theta} \rangle$, where we recall $\ol{\theta} \in B(L^{\mathrm{ab}})_{\basic} \cong \Lcoinv$ is the corresponding element. 

We assume now that $\Lambda$ admits a square root of $q$, which we now fix. This allows us to make the following definition. 
\begin{defn}
We define the sheaf $\IC_{\Bun_{P}} := \mf{q}_{P}^{*}(\Delta_{P}^{1/2})[\dim(\Bun_{P})] \in D(\Bun_{P})$, where $\Delta_{P}^{1/2} \in D(\Bun_{L})$ is the pullback of the sheaf on $\Bun_{L^{\mathrm{ab}}}$ whose value on each connected component is given by $\ol{\delta}_{P}^{1/2}$. 
\end{defn}
We note the following corollary of Theorem \ref{thm: main}. 
\begin{cor}
The sheaf $\IC_{\Bun_{P}}$ on $\Bun_{P}$ is Verdier self-dual.
\end{cor}
In particular, this motivates the morally correct definition of the Eisenstein functor. 
\begin{defn}{\label{defn: normalizedeisfunct}}
We define the normalized Eisenstein functor to be 
\[ \nmEisP^{G}\colon D(\Bun_{L}) \ra D(\Bun_{G}) \]
\[ A \mapsto \mf{p}_{P!}(\mf{q}_{P}^{*}(A) \otimes \IC_{\Bun_{P}}) . \]
\end{defn}
\begin{rem}{\label{rem: unnormalizedvsnormalizedeisenstein}}
We note that we have an isomorphism $\nmEisP(-) \cong \EisP(- \otimes \Delta_{P}^{1/2})$, which is analogous to the relationship between normalized and unnormalized parabolic induction.
\end{rem} 
We write $\EisP^{G}(-) = \bigoplus_{\theta \in B(L)_{\basic}} \EisP^{G,\theta}(-)$  and $\nmEisP^{G}(-) = \bigoplus_{\theta \in B(L)_{\basic}} \nmEisP^{G,\theta}(-)$ for the direct sum decompositions induced by Corollary \ref{cor: conncompsBunP}. These functors are compatible with inclusions of parabolics in the obvious sense.
\begin{lem}{\label{lem: compositionofeisensteinseries}}
Let $P_{1} \subset P_{2} \subset G$ be an inclusion of parabolic subgroups with Levi factors $L_{1}$ and $L_{2}$. We write $P_{1} \cap L_{2}$ for the image of $P_{1}$ under the composition $P_{1} \subset P_{2} \twoheadrightarrow L_{2}$. We have a natural equivalence
\[ \nmEis_{P_{2}}^{G} \circ \nmEis_{P_{1} \cap L_{2}}^{L_{2}}(-) \cong \nmEis_{P_{1}}^{G}(-)  \]
of objects in $D(\Bun_{G})$. 

More precisely, given $\theta_{1} \in B(L_{1})_{\basic} \cong \pi_{1}(L_{1})_{\Gamma_{F}}$ and $\theta_{2} \in B(L_{2})_{\basic} \cong \pi_{2}(L_{2})_{\Gamma_{F}}$, we have that 
\[ \nmEis_{P_{2}}^{G,\theta_{2}} \circ \nmEis_{P_{1} \cap L_{2}}^{L_{2},\theta_{1}}(-) \cong 0 \]
unless $\theta_{2}$ is the image of $\theta_{1}$ under the natural map $\pi_{1}(L_{1})_{\Gamma_{F}} \ra \pi_{1}(L_{2})_{\Gamma_{F}}$. In this case, we have that 
\begin{equation*}
 \nmEis_{P_{2}}^{G,\theta_{2}} \circ \nmEis_{P_{1} \cap L_{2}}^{L_{2},\theta_{1}}(-) \cong \nmEis_{P_{1}}^{G,\theta_{1}}(-). 
\end{equation*}
The analogous claim also holds for the unnormalized functors.
\end{lem}
\begin{proof}
We treat the case of the normalized functors, with the case of the unnormalized functors being strictly easier.

It is easy to see that 
\[ \nmEis_{P_{2}}^{G,\theta_{2}} \circ \nmEis_{P_{1} \cap L_{2}}^{L_{2},\theta_{1}}(-) \cong 0 \]
unless $\theta_{2}$ is the image of $\theta_{1}$. In particular, the image of $\Bun_{P_{1} \cap L_{2}}^{(\theta_{1})}$ under $\mf{p}_{P_{1} \cap L_{2}}$ will be supported on the connected component $\Bun_{L_{2}}^{(\theta_{2})}$, where $\theta_{2}$ is the image of $\theta_{1}$. We assume that $\theta_{1}$ and $\theta_{2}$ are of this form for the rest of the proof.

The key point can be found in the proof of  \cite[Lemma~6.4]{CFSpadicperiod}. We recall this now.

For $S \in \AffPerf_k$ and a $G$-bundle $\mathcal{E}^{\mathrm{alg}}$ on $X_{S}^{\alg}$ which we identify with its geometric realization, consider the natural map 
\[ \mathcal{E}^{\mathrm{alg}}/P_{1} \ra \mathcal{E}^{\mathrm{alg}}/P_{2} \]
of schemes over $X_{S}^{\alg}$. This map is a locally trivially fibration with fiber $P_{2}/P_{1} = L_{2}/(P_{1} \cap L_{2})$ in the \'etale topology on $X_{S}^{\alg}$\footnote{In other words, there exists an \'etale covering  $T \ra X_{S}^{\alg}$ such that the pullback along this map identifies with the projection $(\mathcal{E}^{\alg}/P_2) \times_{X_{S}^{\alg}} T \times_{\Spec{(F)}} P_{2}/P_{1} \ra (\mathcal{E}^{\alg}/P_2) \times_{X_{S}^{\alg}} T$.}. 

If we are given a $P_{2}$-structure $\mathcal{E}^{\mathrm{alg}}_{P_{2}}$ on $\mathcal{E}^{\mathrm{alg}}$ 
then this corresponds to a section $X_{S}^{\alg} \ra \mathcal{E}^{\mathrm{alg}}/P_{2}$, and if we pullback along the previous fibration then the pullback is the same as 
\[ \mathcal{E}^{\mathrm{alg}}_{P_{2}} \times^{P_{2}} L_{2}/(P_{1} \cap L_{2}) \ra X_{S}^{\alg}.  \]
It follows that, for a fixed $G$-bundle $\mathcal{E}^{\mathrm{alg}}$ on $X_{S}^{\alg}$, we have an equivalence between 
\begin{enumerate}
\item a $P_{1}$-structure on $\mathcal{E}^{\mathrm{alg}}$, 
\item a $P_{2}$-structure on $\mathcal{E}^{\mathrm{alg}}$ together with a $P_{1} \cap L_{2}$-structure on $\mathcal{E}^{\mathrm{alg}}_{P_{2}} \times^{P_{2}} L_{2}$.
\end{enumerate}
In particular, we obtain a map $\mf{q}\colon \Bun_{P_{1}} \ra \Bun_{P_{1} \cap L_{2}}$. This fits into a commutative diagram 
\[ \begin{tikzcd}
  \Bun_{P_{1}} \arrow[r,"\mf{p}"] \arrow[d,"\mf{q}"] \arrow[dd,bend right=80,"\mf{q}_{P_{1}}"] \arrow[rr,bend left=30,"\mf{p}_{P_{1}}"]  & \Bun_{P_{2}} \arrow[r,"\mf{p}_{P_{2}}"] \arrow[d,"\mf{q}_{P_{2}}"] & \Bun_{G} \\
  \Bun_{P_{1} \cap L_{2}} \arrow[d,"\mf{q}_{P_{1} \cap L_{2}}"] \arrow[r,"\mf{p}_{P_{1} \cap L_{2}}"] & \Bun_{L_{2}} & \\
 \Bun_{L_{1}},& & 
    \end{tikzcd}
\]
where top left square is Cartesian by the previous equivalence. Restricting to connected components this gives
\[ \begin{tikzcd}
  \Bun_{P_{1}}^{(\theta_{1})} \arrow[r,"\mf{p}"] \arrow[d,"\mf{q}"] \arrow[dd,bend right=80,"\mf{q}^{(\theta_{1})}_{P_{1}}"] \arrow[rr,bend left=40,"\mf{p}^{(\theta_{1})}_{P_{1}}"]  & \Bun_{P_{2}}^{(\theta_{2})} \arrow[r,"\mf{p}^{(\theta_{2})}_{P_{2}}"] \arrow[d,"\mf{q}^{(\theta_{2})}_{P_{2}}"] & \Bun_{G} \\
  \Bun_{P_{1} \cap L_{2}}^{(\theta_{1})} \arrow[d,"\mf{q}^{(\theta_{1})}_{P_{1} \cap L_{2}}"] \arrow[r,"\mf{p}^{(\theta_{1})}_{P_{1} \cap L_{2}}"] & \Bun_{L_{2}}^{(\theta_{2})} & \\
 \Bun^{(\theta_{1})}_{L_{1}}.& & 
    \end{tikzcd}
\]
It follows that, for all $A \in D(\Bun_{L_{1}})$, we have that 
 \begin{align*}
    \nmEis_{P_{1}}^{G,\theta}(A) & \cong \mf{p}^{(\theta_{1})}_{P_{1}!}\mf{q}_{P_{1}}^{(\theta_{1})*}(A \otimes \ol{\mf{q}}_{L_{1}}^{(\theta_{1})*}(\ol{\delta}_{P_{1}}^{1/2}))[\langle \xi_{P_{1}}, \ol{\theta}_{1} \rangle] \\
    &  \cong \mf{p}^{(\theta_{2})}_{P_{2}!}\mf{p}_{!}\mf{q}^{*}\mf{q}^{(\theta_{1})*}_{P_{1} \cap L_{2}}(A \otimes \ol{\mf{q}}_{L_{1}}^{(\theta_{1})*}(\ol{\delta}_{P_{1}}^{1/2}))[\langle \xi_{P_{1}}, \ol{\theta}_{1} \rangle] \\ 
    & \cong \mf{p}^{(\theta_{2})}_{P_{2}!}\mf{q}^{(\theta_{2})*}_{P_{2}}\mf{p}^{(\theta_{1})}_{P_{1} \cap L_{2}!}\mf{q}^{(\theta_{1})*}_{P_{1} \cap L_{2}}(A \otimes \ol{\mf{q}}_{L_{1}}^{(\theta_{1})*}(\ol{\delta}_{P_{1}}^{1/2}))[\langle \xi_{P_{1}}, \ol{\theta}_{1} \rangle], 
\end{align*}
where the last isomorphism follows from proper base change and 
\[
\ol{\mf{q}}_{L_{1}}^{(\theta_{1})}\colon \Bun_{L_{1}}^{(\theta_{1})} \ra \Bun_{L_{1}^{\mathrm{ab}}}^{\ol{\theta}} \cong [\ast/\ul{L_{1}^{\mathrm{ab}}(F)}] \] 
is the natural map. We now note that we have an equality 
\[ \langle \xi_{P_{1}}, \ol{\theta}_{1} \rangle = \langle \xi_{P_{2}}, \ol{\theta}_{2} \rangle + \langle \xi_{P_{1} \cap L_{2}},\ol{\theta}_{1} \rangle. \]
Indeed, this follows from breaking up $\xi_{P_{1}}$ in terms of the sum of roots that come from $P_{2}$ and the sum of roots that come from $P_{1} \cap L_{2}$. Similarly, we have an equality of characters of $L_{1}^{\mathrm{ab}}(F)$:
\[ \ol{\delta}_{P_{1}} = \ol{\delta}_{P_{1} \cap L_{2}} \otimes \ol{\delta}_{P_{2}}|_{L_{1}^{\ab}(F)}. \]

To deal with the modulus character twists, consider the following commutative diagram
\[\begin{tikzcd} \Bun_{P_{1} \cap L_{2}}^{(\theta_{1})} \arrow[d,"\mf{q}^{(\theta_{1})}_{P_{1} \cap L_{2}}"] \arrow[r,"\mf{p}^{(\theta_{1})}_{P_{1} \cap L_{2}}"] & \Bun^{(\theta_{2})}_{L_{2}} \arrow[dd,"\ol{\mf{q}}_{L_{2}}^{(\theta_{2})}"] \\
 \Bun^{(\theta_{1})}_{L_{1}} \arrow[d,"\ol{\mf{q}}_{L_{1}}^{(\theta_{1})}"] &  \\
\left[\ast/\ul{L_{1}^{\mathrm{ab}}(F)}\right] \arrow[r,"r"] & \left[\ast/\ul{L_{2}^{\mathrm{ab}}(F)}\right]. 
 \end{tikzcd} \]
This tells us that $\nmEis_{P_{1}}^{G,\theta}(A)$ is isomorphic to 
 \begin{align*}
    &\nmEis_{P_{1}}^{G,\theta}(A) \\ & \cong \mf{p}^{(\theta_{2})}_{P_{2}!}\mf{q}^{(\theta_{2})*}_{P_{2}}\mf{p}^{(\theta_{1})}_{P_{1} \cap L_{2}!}\mf{q}^{(\theta_{1})*}_{P_{1} \cap L_{2}}(A \otimes \ol{\mf{q}}_{L_{1}}^{(\theta_{1})*}(\ol{\delta}_{P_{1}}^{1/2}))[\langle \xi_{P_{1}}, \ol{\theta}_{1} \rangle], \\ 
    & \cong \mf{p}^{(\theta_{2})}_{P_{2}!}\mf{q}^{(\theta_{2})*}_{P_{2}}\mf{p}^{(\theta_{1})}_{P_{1} \cap L_{2}!}\mf{q}^{(\theta_{1})*}_{P_{1} \cap L_{2}}(A \otimes \ol{\mf{q}}_{L_{1}}^{(\theta_{1})*}(\ol{\delta}_{P_{1} \cap L_{2}} \otimes r^{*}(\ol{\delta}_{P_{2}})))[\langle \xi_{P_{1}}, \ol{\theta}_{1} \rangle] \\
     & \cong \mf{p}^{(\theta_{2})}_{P_{2}!}\mf{q}^{(\theta_{2})*}_{P_{2}}\mf{p}^{(\theta_{1})}_{P_{1} \cap L_{2}!}(\mf{q}^{(\theta_{1})*}_{P_{1} \cap L_{2}}(A \otimes \ol{\mf{q}}_{L_{1}}^{(\theta_{1})*}(\ol{\delta}_{P_{1} \cap L_{2}})) \otimes \mf{q}^{(\theta_{1})*}_{P_{1} \cap L_{2}}\ol{\mf{q}}_{L_{1}}^{(\theta_{1})*}r^{*}(\ol{\delta}_{P_{2}}))[\langle \xi_{P_{1}}, \ol{\theta}_{1} \rangle] 
     \\ & \cong \mf{p}^{(\theta_{2})}_{P_{2}!}\mf{q}^{(\theta_{2})*}_{P_{2}}\mf{p}^{(\theta_{1})}_{P_{1} \cap L_{2}!}(\mf{q}^{(\theta_{1})*}_{P_{1} \cap L_{2}}(A \otimes \ol{\mf{q}}_{L_{1}}^{(\theta_{1})*}(\ol{\delta}_{P_{1} \cap L_{2}})) \otimes \mf{p}_{P_{1} \cap L_{2}}^{(\theta_{1})*}\ol{\mf{q}}_{L_{2}}^{(\theta_{2})*}(\ol{\delta}_{P_{2}}))[\langle \xi_{P_{1}}, \ol{\theta}_{1} \rangle] 
     \\ & \cong \mf{p}^{(\theta_{2})}_{P_{2}!}\mf{q}^{(\theta_{2})*}_{P_{2}}(\mf{p}^{(\theta_{1})}_{P_{1} \cap L_{2}!}(\mf{q}^{(\theta_{1})*}_{P_{1} \cap L_{2}}(A \otimes \ol{\mf{q}}_{L_{1}}^{(\theta_{1})*}(\ol{\delta}_{P_{1} \cap L_{2}})[\langle \xi_{P_{1} \cap L_{2}},\ol{\theta}_{1} \rangle])) \otimes \ol{\mf{q}}_{L_{2}}^{(\theta_{2})*}(\ol{\delta}_{P_{2}}))[\langle \xi_{P_{2}}, \ol{\theta}_{2} \rangle]      \\ & \cong \nmEis_{P_{2}}^{G,\theta_{2}} \circ \nmEis_{P_{1} \cap L_{2}}^{L_{2},\theta_{1}}(A),
\end{align*}
as desired. Here we have used a projection formula in the fifth isomorphism. 
\end{proof}

We assume from now on that the group $G$ is quasi-split with a choice of Borel $B \subset G$ unless otherwise stated. We will need the following notation.
\begin{enumerate}
\item Let $\mathcal{J}$ denote the set of vertices in the (relative) Dynkin diagram. For each $i \in \mathcal{J}$, we write $\alpha_{i}$ for the corresponding (reduced) simple root. For a Levi subgroup $L$, we let $\mathcal{J}_{L} \subset \mathcal{J}$ denote the corresponding subset. 
\item For a Levi subgroup $L \subset G$, we write $\mathcal{J}_{G,L} := \mathcal{J} \setminus \mathcal{J}_{L}$. 
\item Let $W_{G}$ denote the relative Weyl group of $G$.
\end{enumerate}
We have the following lemma. 
\begin{lem}{\label{lem: Weylgrouptranslate}}
Let $w(L)$ and $w(P)$ denote the conjugates of the Levi and parabolic under some element $w \in W_{G}$. We write $w\colon \Bun_{L} \xrightarrow{\cong} \Bun_{w(L)}$ for the map induced by the isomorphism $L \cong w(L)$. Then, for all $A \in D(\Bun_{L})$, there is a natural isomorphism 
\[ \nmEis_{P}^{G}(A) \cong \nmEis_{w(P)}^{G}(w_{*}(A)) \]
of objects in $D(\Bun_{G})$, and similarly for the unnormalized functors.
\end{lem}
\begin{proof}
We note that $w_{*}(\Delta_{P}) \cong \Delta_{w(P)}$, 
and, for all $\theta \in B(L)_{\basic}$, we have that $\langle \xi_{P}, \ol{\theta} \rangle = \langle \xi_{w(P)}, w(\ol{\theta}) \rangle$. Therefore, we reduce to showing that 
\[ \mf{p}_{P!}\mf{q}_{P}^{*}(A) \cong \mf{p}_{w(P)!}\mf{q}_{w(P)}^{*}(w_{*}(A)). \]
We write $w_{P}\colon \Bun_{P} \xrightarrow{\cong} \Bun_{w(P)}$ for the map induced by the isomorphism $P \cong w(P)$. We have a commutative diagram 
\[ \begin{tikzcd}
\Bun_{P} \arrow[r,"w_{P}"]  \arrow[rr,bend left=30,"\mf{p}_{P}"] 
 \arrow[d,"\mf{q}_{P}"] & \Bun_{w(P)} \arrow[r,"\mf{p}_{w(P)}"] \arrow[d,"\mf{q}_{w(P)}"] & \Bun_{G} \\
\Bun_{L} \arrow[r,"w"]  & \Bun_{w(L)}. &  
\end{tikzcd}\]
Using the isomorphisms $w^{*}w_{*}(A) \cong A$ and $w_{P}^{*}w_{P*}(A) \cong A$ coming from adjunction, we write
\begin{align*}
\mf{p}_{P!}\mf{q}_{P}^{*}(A)  & \cong \mf{p}_{P!}\mf{q}_{P}^{*}w^{*}w_{*}(A) \\ 
& \cong \mf{p}_{w(P)!}w_{P*}w_{P}^{*}\mf{q}_{w(P)}^{*}w_{*}(A) \\
& \cong \mf{p}_{w(P)!}\mf{q}_{w(P)}^{*}w_{*}(A),
\end{align*}
as desired. 
\end{proof}
In general the values of these geometric Eisenstein functors are very complicated and difficult to compute (See for example \cite[Example~3.3.2]{HanBeijLec}), but there are certain simple contributions coming from the split reductions. This is the locus of $\Bun_{P}$ coming from the image of the natural section  
\[ \mf{s}_{P}\colon \Bun_{L} \ra \Bun_{P} \]
of $\mf{q}_{P}$ induced by the inclusion $L \subset P$. To understand this, we need to understand the topological image of the composition $\mf{p}_{P} \circ \mf{s}_{P}\colon \Bun_{L} \ra \Bun_{G}$. For our calculation, we will restrict to the locus given by the semistable bundles.
\begin{defn}
We define $B(G)_{L}$, the set of $L$-reducible elements, to be the image $\mathrm{Im}(i_{L}\colon B(L)_{\basic} \ra B(G)) =: B(G)_{L}$ of the map induced by the inclusion $L \subset G$.  
\end{defn}
We will now study the set $B(G)_{L}$ and its properties. We let $\Delta_{G,L}$ be the set of relative simple roots $\alpha_{i}$, for $i \in \mathcal{J}_{G,L} := \mathcal{J} \setminus \mathcal{J}_{L}$. 

Let $\Lcoinvdom \subset \Lcoinv$ be the subset of elements  which pair non-negatively 
with the image of $\Delta_{G,L}$ under the identification $\mathbb{X}^{*}(A_{L})_{\mathbb{Q}} \simeq \mathbb{X}^{*}(L^{\mathrm{ab}}_{\ol{F}})^{\Gamma_{F}}_{\mathbb{Q}}$. 
We say an element $\theta \in B(L)_{\basic}$ such that $\ol{\theta}$ lies in $\Lcoinvdom$ is dominant, 
and if it lies in the space which pairs non-positively with the image of $\Delta_{G,L}$ then we say that it is HN-dominant.

For $b \in B(G)_{L}$, we consider the set $W_{L,b} := W[L,L^{b}]$ as defined in \cite[Section~5.3]{BMOBGPar}. This will be identified with the set of elements in $W_{G}$ such that 
\[ 
w(L) \subset L^{b}, \quad 
w(L \cap B) \subset B, \quad w^{-1}(L^{b} \cap B) \subset B,\]
where $L^{b}$ denotes the centralizer of the slope homomorphism of $b \in B(G)$. We now want to relate this Weyl group to the fibers of the map $i_{L}\colon B(L)_{\basic} \ra B(G)$. We begin with some preparations.

\begin{lem}\label{lem:innredL}
Assume that $b \in B(G)$ is basic. 
\begin{enumerate} 
\item\label{en:BLBP} 
The inclusion $L \subset P$ induces $B(L) \cong B(P)$. 
\item\label{en:Pinnerred}
The parabolic subgroup $P$ of $G$ transfer to a parabolic subgroup of $G_b$ under the inner twisting if and only if $b$ admits a reduction to $L$. 
\end{enumerate}
\end{lem}
\begin{proof}
The claim \ref{en:BLBP} follows from \cite[1.4, 3.6]{KotIsoII}. The claim \ref{en:Pinnerred} follows from the claim \ref{en:BLBP} and \cite[p.~259]{CFSpadicperiod}. 
\end{proof}

\begin{lem}\label{lem:WeylLevicoset}
Let $G$ be a connected reductive group over a field $F$. We fix a maximal split torus $S$ and a minimal parabolic subgroup $P_0$ containing $S$. 
Let $W_G$ be the relative Weyl group of $G$ with respect to $S$. 
Let $P$ and $Q$ be standard parabolic subgroups of $G$ with  
standard Levi subgroups $L$ and $M$, respectively, with respect to $P_0$ and $S$. 
We put \[
S(L,M) :=\{ g \in G(F) \mid \textrm{$gLg^{-1} \cap M$ contains a maximal split torus of $G$} \}. 
\]
Then we have a natural bijection 
\[
 W_M \backslash W_G /W_L \cong M(F) \backslash S(L,M) / L(F) . 
\]
\end{lem}
\begin{proof}
The map is given by taking a representative of an element of $W_G$ in $N_G(S)(F)$. The map is surjective since any maximal split tori are rationally conjugate. 
By \cite[5.20 Corollaire]{BoTiGred}, the composition 
\[
 W_M \backslash W_G /W_L \to M(F) \backslash S(L,M) / L(F) \to 
 Q(F) \backslash G(F) / P(F) 
\]
is a bijection. Therefore, the first surjective map is also bijective. 
\end{proof}

\begin{lem}\label{lem:LeviH1inj}
Let $G$ be a connected reductive group over a field $F$. 
Let $P$ be a parabolic subgroup of $G$ with a Levi subgroup $L$. 
Then the inclusions $L \subset P \subset G$ induce 
\[
 H^1(F,L) \cong H^1(F,P) \hookrightarrow H^1(F,G). 
\]
\end{lem}
\begin{proof}
The injectivity of the second map follows from \cite[4.13 Th\'eor\`eme]{BoTiGred} (\cf \cite[p.~136, Exercice 1]{SerreGaloisCohom}). The map $H^1(F,P) \to H^1(F,L)$ induced by the natural projection $P \to L$ is injective by $H^1(F,R_{\mathrm{u}}(P))=1$ and a torsion argument. This map is surjective too since $P \to L$ has a splitting $L \subset P$. Hence the first map in the claim is bijective. 
\end{proof}
We now work again with our fixed quasi-split $G$. We recall again that $b$ admits a reduction to a basic element $b_{L^{b}} \in B(L^{b})_{\mathrm{basic}}$ such that $\ol{b}_{L^{b}} \in \mathbb{X}_{*}(L_{\ol{F}}^b)^{+}_{\Gamma_{F}}$. We now have the following lemma. 
\begin{lem}{\label{lemma: Lreducibleelements}}
For every element $b \in B(G)_{L}$, there exists an injective map of sets
\begin{align*}
i_{L}^{-1}(b) &\ra W_{L,b} \\  
\theta &\mapsto w_{\theta} 
\end{align*}
that sends $\theta$ to the unique $w_{\theta}$ such that $\ol{w_{\theta}(\theta)} \in \mathbb{X}_{*}(w_{\theta}(L)^{\mathrm{ab}})_{\Gamma_F}$ is dominant. 
It follows that $w_{\theta}(\theta) \in B(w_{\theta}(L))_{\basic}$ is a reduction of $b_{L^{b}} \in B(L^{b})_{\basic}$ to $B(w_{\theta}(L))$. Moreover, the image of this map is given by the set of $w \in W_{L,b}$ such that $w(L) \subset L^{b}$ transfers to a Levi subgroup of $(L^{b})_{b_{L^{b}}}$ under the inner twisting. 
\end{lem}
\begin{proof}
We first start by constructing the map and showing it has the desired properties. By \cite[Lemma~2.11]{BeZeRepGL}, we have a bijection 
\[ W_{L,b} \cong W_{L^{b}} \backslash W(L,L^{b}) /W_{L}, \]
where $W(L,L^{b}) := \{ w \in W_{G} \mid w(L) \subset L^{b} \}$. By Lemma \ref{lem:WeylLevicoset}, 
we deduce that this is isomorphic to 
\[ L^{b}(F)\backslash I(L,L^{b})/L(F), \]
where 
\[ I(L,L^{b}) := \{g \in G(F) \mid gLg^{-1} \subset L^{b} \}. \]
Now, given $\theta \in i_{L}^{-1}(b)$, we let $L^{\theta}$ denote the $G$-centralizer of the slope homomorphism, and let $P^{\theta}$ be the standard parabolic with Levi factor $L^{\theta}$. Since $\theta$ maps to $b$, 
there exists $g_{\theta} \in G(F)$ such that $g_{\theta}(P^{\theta},L^{\theta})g_{\theta}^{-1} = (P^{b},L^{b})$. We note that $L \subset L^{\theta}$ by definition, and therefore it follows that $g_{\theta} \in I(L,L^{b})$. We now simply let $w_{\theta}$ denote the element corresponding to $g_{\theta}$ under the equivalence 
\[  W_{L,b} \cong L^{b}(F)\backslash I(L,L^{b})/L(F), \]
described above\footnote{Another way of describing $w_{\theta}$ is as the generic relative position of the two reductions $\mathcal{E}_{L} \times^{L} P$ and $\mathcal{E}_{L^{b}} \times^{L^{b}} P^{b}$ of $\mathcal{E}$, as in \cite[Lemma~4.1]{HNStrataSchi}.}. We note, by construction, we have an equality of slope polygons $\nu_{b_{L^{b}}} = \nu_{b} = \nu_{w_{\theta}(\theta)}$ as elements in $\mathbb{X}_{*}(T_{\ol{F}})^{\Gamma}_{\mathbb{Q}}$. 

Now let $P^{w_{\theta}}$ 
denote the standard parabolic with Levi factor $w_{\theta}(L)$. We consider the parabolic structures $\mathcal{E}_{b_{L^{b}}} \times^{L^{b}} P^{b}$ and $\mathcal{E}_{w_{\theta}(\theta)} \times^{w_{\theta}(L)} P^{w_{\theta}}$ 
on the $G$-bundle $\mathcal{E}_{b}$. It follows by \cite[Theorem~4.1]{HNStrataSchi} (cf. \cite[Theorem~1.8]{CFSpadicperiod} and \cite[Appendix~A,Theorem~A.5]{NguyenViehmannBdRGrassmannian}, for a discussion of this in the Fargues--Fontaine setting)
and the aforementioned equality of slope polygons that $\mathcal{E}_{b_{L^{b}}} \times^{L^{b}} P^{b}$ is obtained by extension of structure group along the inclusion $P^{w_{\theta}} \subset P^{b}$. Therefore, $w_{\theta}(\theta)$ is a reduction of $b_{L^{b}}$ to $B(w_{\theta}(L))$, 
as desired.

We now show the injectivity of the map. Suppose we have elements $\theta$ and $\theta'$ such that $w_{\theta} = w_{\theta'} = w$. Then $w_{\theta}(\theta)$ and $w_{\theta'}(\theta')$ are both reductions of $b_{L^{b}}$ to $B(w(L))_{\basic}$. We claim that $w_{\theta}(\theta) = w_{\theta'}(\theta')$, which implies that $\theta = \theta'$. 

Since $(w(L))_{w_{\theta}(\theta)}$ is isomorphic to a Levi $M$ of $(L^b)_{b_{L^{b}}}$, we have a commutative diagram 
\[\begin{tikzcd}
B(w(L)) \arrow[r, "\sim"] \arrow[d]& B(M) \arrow[d] \\
B(L^{b}) \arrow[r, "\sim"] & B((L^b)_{b_{L^{b}}}) . 
\end{tikzcd}\]
Hence we can reduce the problem to showing that the map $B(M) \ra B((L^b)_{b_{L^{b}}})$ is injective on the slope $0$ elements. In other words, the map $\pi_{1}(M)_{\Gamma_{F},\mathrm{tors}} \ra \pi_{1}((L^b)_{b_{L^{b}}})_{\Gamma_{F},\mathrm{tors}}$ is injective (where we have used the commutative diagram (2) in \cite[Theorem~1.15]{RaRiFiso}). However, by \cite[Theorem~1.15]{RaRiFiso}, this identifies with the map on Galois Cohomology $H^{1}(\Gamma_{F},M) \ra H^{1}(\Gamma_{F},(L^b)_{b_{L^{b}}})$, and this is injective by Lemma \ref{lem:LeviH1inj}. 

We saw in the previous paragraph that the map $i_{L}^{-1}(b) \hookrightarrow W_{L,b}$ takes values in the set of elements $w \in W_{L,b}$ such that $w(L) \subset L^{b}$ transfers to $(L^{b})_{b_{L^{b}}}$. Conversely, if we are given such an element $w \in W_{L,b}$ then, by Lemma \ref{lem:innredL}, the element $b_{L^{b}}$ admits a reduction to $\theta' \in B(w(L))_{\basic}$, where we note that this element is necessarily basic since $b_{L^{b}}$ is. Then $w^{-1}(\theta') \in B(L)_{\basic}$ is an element in $i^{-1}(b)$ whose attached Weyl group element is $w$, as desired. 
\end{proof}
We will also need the following result. 
\begin{lem}{\label{lem: parabolicstransfer}}
An element $b \in B(G)$ lies in $B(G)_{L}$ if and only if there exists $w \in W_{L,b}$ such that the parabolic $w(P) \cap L^{b}$ of $L^{b}$ transfers to a parabolic subgroup $Q_{b,w} \subset G_{b}$ under the inner twisting. More precisely, if $\theta$ maps to $b \in B(G)$ with corresponding Weyl group element $w_{\theta}$ as in the previous lemma then $w_{\theta}(P) \cap L^{b}$ transfers to a parabolic subgroup of $G_{b}$. Moreover, for every element $\theta \in B(L)_{\mathrm{basic}}$ mapping to $b$, the Levi factor of $Q_{b,w_{\theta}}$ is equal to $w_{\theta}(L)_{w_{\theta}(\theta)}$.
\end{lem}
\begin{proof}
This follows from combining Lemma \ref{lemma: Lreducibleelements} with Lemma \ref{lem:innredL} applied to the basic reduction $b_{L^{b}}$ of $b$ by using the isomorphism $(L^{b})_{b_{L^{b}}} \cong G_{b}$, as seen in the previous argument.
\end{proof}

We now turn to computing our Eisenstein functors. Let us start with the simplest family of cases.
\begin{lem}{\label{lem: sloperzeroeisensteincalc}}
Suppose that $\theta \in B(L)_{\basic}$ with image $b \in B(G)_{L}$ satisfies $\langle \xi_{P},\ol{\theta} \rangle = 0$. It follows that $b$ is basic, so in particular $W_{L,b}$ is trivial, and we have that 
\[ \Eis_{P}^{G,\theta}(i^{\theta}_{!}(A)) \cong i^{b}_{!}(\Ind_{Q_{b}(F)}^{G_{b}(F)}(A)) \]
\[ \nmEis_{P}^{G,\theta}(i^{\theta}_{!}(A)) \cong i^{b}_{!}(i_{Q_{b}(F)}^{G_{b}(F)}(A)), \]
where $Q_{b} := Q_{b,1}$ is the transfer of $P \subset G$ to $G_{b}$, as in Lemma \ref{lem: parabolicstransfer}. 
\end{lem}
\begin{proof}
The second isomorphism follows from the first using the relationship 
\[ \nmEisP^{G,\theta}(i^{\theta}_{!}(A))|_{\Bun_{G}^{b}} \cong \EisP^{G,\theta}(i^{\theta}_{!}(A \otimes \delta_{P,\theta}^{1/2})), \]
as in Remark \ref{rem: unnormalizedvsnormalizedeisenstein}.  To prove the first isomorphism, by an application of proper base change we are interested in computing $\mf{p}_{!}\mf{q}^{*}(A)$, as in the diagram
\[\begin{tikzcd}
\Bun_{P}^{\theta} \arrow[r,"\mf{p}"] \arrow[d,"\mf{q}"]& \Bun_{G}  \\
\Bun_{L}^{\theta} \cong [\ast/\ul{L_{\theta}(F)}].& 
\end{tikzcd}\]
Note that, since $\theta$ maps to $b$, we have an equality: $\langle \xi_{P}, \theta \rangle = \langle 2\rho_{G}, \nu_{b} \rangle = 0$, so in particular $b$ must be basic. By a standard argument with HN-slopes (See for example the proof of \cite[Lemma~9.3]{HamGeomES}), it follows that the topological image of the map $\mf{p}$ must be just $\Bun_{G}^{b}$. Moreover, if we look at the space $\Bun_{P}^{\theta}$ and use Lemma \ref{lem:geomdescpofconncomp}, we see that all the vector bundles $\mcLie(U)_{\alpha_i,\theta}$ are semistable of slope $0$. Using this, we deduce an isomorphism $\Bun_{P}^{\theta} \cong [\ast/\ul{Q_{b}(F)}]$. 
Therefore, the previous diagram just becomes 
\[\begin{tikzcd}
\left[\ast/\ul{Q_{b}(F)}\right] \arrow[r,"p"] \arrow[d,"q"] & \left[\ast/\ul{G_{b}(F)}\right] \cong \Bun_{G}^{b} \arrow[r,"i^{b}"] & \Bun_{G} \\
\left[\ast/\ul{L_{\theta}(F)}\right],& & 
\end{tikzcd}\]
where $p$ and $q$ are induced by the natural maps of groups. If $\rho$ is a smooth representation of $L_{\theta}(F)$, we note that $p_{!}q^{*}(\rho)$ will be identified with the set of compactly supported functions on $G_{b}(F)$ whose supports are compact in $G_{b}(F)/Q_{b}(F)$, and which transform under the right translation action of $Q_{b}(F)$ by the action of $\rho$ inflated to $Q_{b}(F)$, 
which is $\Ind_{Q_{b}(F)}^{G_{b}(F)}(\rho)$, by definition. Using the exactness of $\Ind_{Q_{b}(F)}^{G_{b}(F)}(-)$, we can further deduce that $p_{!}q^{*}(A) \cong \Ind_{Q_{b}(F)}^{G_{b}(F)}(A)$ for all $A \in D(L_{\theta}(F),\Lambda)$, as desired.
\end{proof}
With this warmup out of the way, we come to the following more general claim.
\begin{thm}{\label{thm: eisensteinvalues}}
For all $\theta \in B(L)_{\mathrm{basic}}$ mapping to $b \in B(G)$ with attached Weyl group element $w_{\theta} \in W_{L,b}$, as in Lemma \ref{lemma: Lreducibleelements}, we consider $A \in D(L_{\theta}(F),\Lambda)$ and write $w_{\theta*}(A) \in D(w_{\theta}(L)_{w_{\theta}(\theta)}(F),\Lambda)$ 
for the image under the equivalence of categories induced by the isomorphism $L_{\theta} \cong w_{\theta}(L)_{w(\theta)}$. We have an isomorphism
\[ \nmEisP^{G,\theta}(i^{\theta}_{!}(A))|_{\Bun_{G}^{b}} \cong i_{Q_{b,w_{\theta}}(F)}^{G_{b}(F)}(w_{\theta*}(A)) \otimes \delta_{b}^{-1/2}[-\langle 2\rho_{G}, \nu_{b} \rangle] \]
in $D(\Bun_{G}^{b}) \cong D(G_{b}(F),\Lambda)$, where $Q_{b,w_{\theta}}$ is as in Lemma \ref{lem: parabolicstransfer}. 
Moreover, if $\theta$ is HN-dominant with respect to the Borel for which $P$ is standard then this defines an isomorphism
\[ \nmEisP^{G,\theta}(i^{\theta}_{!}(A)) \cong i^{b}_{!}(i_{Q_{b,w_{\theta}}(F)}^{G_{b}(F)}(w_{\theta*}(A)) \otimes \delta_{b}^{-1/2})[-\langle 2\rho_{G}, \nu_{b} \rangle]  \]
of objects in $D(\Bun_{G})$. 
\end{thm}
\begin{proof}
The second claim follows immediately from the first by noting that, in the case that $\theta$ is HN-dominant with respect to the Borel $B$ for which $P$ is standard then the space $\Bun_{P}^{\theta}$ is a classifying stack given by the point defined by the split reduction quotiented out by the torsor. In particular, its topological image under $\mf{p}_{P}$ is just $b$. 

For the first claim, we set $w := w_{\theta}$. By applying Lemma \ref{lem: Weylgrouptranslate}, we reduce to computing 
\[ \nmEis_{w(P)}^{G,w(\theta)}(i^{w(\theta)}_{!}(w_{*}(A)). \]
We let $P^{b}$ be the standard Levi subgroup with Levi factor $L^{b}$ given by the centralizer of the slope homomorphism of $b$. We let $b_{L^{b}} \in B(L^{b})_{\basic}$ be the canonical dominant reduction.

By definition of $w_{\theta} = w$, we have that $w(L) \subset L^{b}$. We let $P^{b,w}$ be the unique parabolic subgroup whose Levi factor is equal to $L^{b}$ and contains $w(P)$. In other words, the standard parabolic with Levi factor $L^{b}$ and defined with respect to the Borel $w(B)$. The element $w(\theta)$ 
maps under the natural composition $B(w(L))_{\basic} \cong \pi_{1}(w(L))_{\Gamma_F} \ra \pi_{1}(L^{b})_{\Gamma_F} \cong B(L^{b})_{\basic}$ to the dominant reduction $b_{L^{b}}$ of $b \in B(G)_{L}$, using the second part of Lemma \ref{lemma: Lreducibleelements}. 
Therefore it follows, by applying Lemma \ref{lem: compositionofeisensteinseries}, that we can rewrite this as 
\[ \nmEis_{P^{b,w}}^{G,b_{L^{b}}} \circ \nmEis_{w(P) \cap L^{b}}^{L^{b},w(\theta)}(i^{w(\theta)}_{!}(w_{*}(A))). \] 
By Lemma \ref{lem: sloperzeroeisensteincalc}, this is isomorphic to
\[ \nmEis_{P^{b,w}}^{G,b_{L^{b}}}(i^{b_{L^{b}}}_{!}(i_{Q_{b,w}(F)}^{G_{b}(F)}(w_{*}(A))). \]
Therefore, we have reduced to showing the following.
\begin{lem}
For all $A \in D(G_{b}(F),\Lambda) \cong D(\Bun_{L^{b}}^{b_{L^{b}}})$, we have an isomorphism 
\[ \nmEis_{P^{b,w}}^{G,b_{L^{b}}}(i^{b_{L^{b}}}_{!}(A))|_{\Bun_{G}^{b}} \cong A \otimes \delta_{b}^{-1/2}[-\langle 2\rho_{G},\nu_{b} \rangle] \]
of objects in $D(G_{b}(F),\Lambda) \cong D(\Bun_{G}^{b})$.
\end{lem}
\begin{proof}
Let $U^{b}$, $U^{b,-}$, $U^{b,w}$ denote the unipotent radical of $P^{b}$, $P^{b,-}$, $P^{b,w}$, respectively.

We write $\Bun_{P^{b,w}}^{\theta,b}$ for the preimage of $\Bun_{G}^{b} \cong [\ast/\wt{G}_{b}]$ along the map $\Bun_{P^{b,w}}^{\theta} \ra \Bun_{G}$ induced by $\mf{p}_{P^{b,w}}$. Two applications of proper base change tell us that we are tasked with computing
\[ \mf{p}_{!}\mf{q}^{*}(A \otimes \delta_{P^{b,w},b_{L^{b}}}^{1/2})[\langle \xi_{P^{b,w}},\ol{b}_{L^{b}} \rangle], \]
with maps as in the diagram
\[ \begin{tikzcd}
\Bun_{P^{b,w}}^{b_{L^{b}},b} \arrow[r,"\mf{p}"] \arrow[d,"\mf{q}"] & \left[\ast/\wt{G}_{b}\right] \cong \Bun_{G}^{b} \\
\Bun_{L^{b}}^{b_{L^{b}}} \cong \left[\ast/\ul{G_{b}(F)}\right].& 
\end{tikzcd} \]
We look at $\Bun_{P^{b,w},\ast}^{b_{L^{b}},b}$, the fiber of $\mf{q}$ over $\ast \ra [\ast/\ul{G_{b}(F)}]$. This defines a closed subspace of $\Bun_{P^{b,w},\ast}^{b_{L^{b}}}$, where, by Lemma \ref{lem:geomdescpofconncomp}, the latter is an iterated fibration of the Picard $v$-groupoids $\mathcal{P}(\mcLie(U^{b,w})_{\alpha,b_{L^{b}}}[1]) \ra \ast$ for all roots $\alpha$ occurring in $\Lie(U^{b,w})$. In particular, 
\[ \Bun_{P^{b,w},*}^{b_{L^{b}}} \cong X_{n}, \]
where $X_{n} \ra \cdots \ra X_{1} \ra X_0=\ast$ is a sequence of fibrations $f_{i} \colon X_{i} \ra X_{i-1}$ in the spaces $\mathcal{P}(\mcLie(U^{b,w})_{\alpha_{i-1},b_{L^{b}}}[1])$, where $\alpha_0, \ldots, \alpha_{n-1}$ are the roots occurring in $\Lie(U^{b,w})$. 
Moreover, there is a filtration 
\[ U^{b,w}=U_0 \supset U_1 \supset \cdots  \supset U_n =\{1\} 
\]
of the unipotent radical of $P^{b,w}$ 
stable under the conjugation action of $P^{b,w}$ 
such that $X_i(S)$ parametrizes the torsors under $\ca{E}_{b_{L^b},S}^{\alg} \times^{L^b} (U_0/U_i)$ for $S \in \AffPerf_k$ by the proof of Lemma \ref{lem:geomdescpofconncomp}. 
Let $Y_i$ be the closed substack of $X_i$ corresponding to the trivial torsor. 
Then we redefine $Y_{i}=f_{i}^{-1}(Y_{i - 1})$ if $\langle \alpha_{i},\ol{b}_{L^{b}} \rangle \leq 0$ and $Y_{i}\cong Y_{i - 1}$ if $\langle \alpha_{i},\ol{b}_{L^{b}} \rangle > 0$. The closed substack $Y_{n} \hookrightarrow X_{n} \cong \Bun_{P^{b,w},*}^{b_{L^{b}}}$ then identifies with $\Bun_{P^{b,w},\ast}^{b_{L^{b}},b}$.

It follows that $\Bun_{P^{b,w}}^{b_{L^{b}},b} \cong [\ast/\wt{G}_{w}]$ is a classifying stack attached to a group diamond $\wt{G}_{w}$. The group diamond $\wt{G}_{w}$ is an iterated fibration of $\mathcal{H}^{0}(\mcLie(U^{b,w})_{\alpha,b_{L^{b}}})$ over $\ul{G_{b}(F)}$, for $\alpha$ a root occurring in $U^{b,w}$ such that $\langle \alpha, \ol{b}_{L^{b}} \rangle \leq 0$. 
We get an induced map of classifying stacks $[\ast/\wt{G}_{w}] \ra [\ast/\wt{G}_{b}]$ that will be induced by a natural inclusion of group diamonds $\wt{G}_{w} \hookrightarrow \wt{G}_{b}$. We write $\wt{G}^{w}$ for the quotient of $\wt{G}_{b}$ by $\wt{G}_{w}$. 
In light of the isomorphism $\Bun_{P^{b,-}}^{b_{L^{b}}} \cong [\ast/\wt{G}_{b}]$, as in the proof of Proposition \ref{prop: moduluscharacterinnaturalsituations}, we can write $\wt{G}_{b}$ as an iterated fibration of $\mathcal{H}^{0}(\mcLie(U^{b,-})_{\alpha,b_{L^{b}}})$ for all roots $\alpha$ occurring in $\Lie(U^{b,-})$ over $\ul{G_{b}(F)}$. 
It follows that $\wt{G}^{w}$ is an iterated fibration of $\mathcal{H}^{0}(\mcLie(U^{b,-})_{\alpha,b_{L^{b}}})$ for all negative roots $\alpha < 0$ such that $w^{-1}(\alpha) < 0$. 
This sits in a Cartesian diagram 
\[ \begin{tikzcd}
\wt{G}^{w} \arrow[r,"\wt{\mf{p}}"] \arrow[d] & \ast \arrow[d] \\ 
\left[\ast/\wt{G}_{w}\right] \arrow[r,"\mf{p}"] & \left[\ast/\wt{G}_{b}\right]. 
\end{tikzcd}
\]
We claim we can now reduce to the following. 
\begin{lem}{\label{lem: keycalc}}
The compactly supported cohomology of $\wt{G}^{w}$ is isomorphic to $\delta_{P^{b,w},b_{L^{b}}}^{-1/2} \otimes \delta_{b}^{-1/2}[-2(\langle \xi_{P^{b}}, \nu_{b} \rangle - \langle \xi_{P^{b}}^{w},\nu_{b} \rangle)]$ 
as a $G_{b}(F)$-representation, 
where $\xi_{P^{b}}^{w}$ is the sum of positive roots $\alpha > 0$ occurring in $\Lie(U_{b})$ such that $w^{-1}(\alpha) < 0$. 
\end{lem}
We note that we have a chain of equalities: $d_{P^{b,w},b_{L^{b}}} = \langle \xi_{P^{b,w}}, \ol{b}_{L^{b}} \rangle = \langle \xi_{P^{b,w}}, \nu_{b_{L^{b}}} \rangle = \langle \xi_{P^{b,w}}, \nu_{b} \rangle$, 
by definition of $b_{L^{b}}$. Since $\nu_{b}$ is dominant and $P^{b}$ is a standard parabolic, we see that this is equal to
\[ \langle \xi_{P^{b}}, \nu_{b} \rangle - 2\langle \xi_{P^{b}}^{w}, \nu_{b} \rangle. \]
We deduce the following chain of isomorphisms: 
\begin{align*}
 \mf{p}_{!}\mf{q}^{*}(A \otimes \delta_{P^{b,w},b_{L^{b}}}^{1/2})[d_{P^{b,w},b_{L^{b}}}] &  \cong \mf{p}_{!}\mf{q}^{*}(A \otimes \delta_{P^{b,w},b_{L^{b}}}^{1/2})[\langle \xi_{P^{b}}, \nu_{b} \rangle - 2\langle \xi_{P^{b}}^{w},\nu_{b} \rangle]
 \\ & \cong  \mf{p}_{!}\mf{p}^{*}(A \otimes \delta_{P^{b,w},b_{L^{b}}}^{1/2})[\langle \xi_{P^{b}}, \nu_{b} \rangle - 2\langle \xi_{P^{b}}^{w},\nu_{b} \rangle] \\
 & \cong A \otimes \delta_{P^{b,w},b_{L^{b}}}^{1/2} \otimes \mf{p}_{!}(\Lambda)[\langle \xi_{P^{b}}, \nu_{b} \rangle - 2\langle \xi_{P^{b}}^{w},\nu_{b} \rangle] \\
 & \cong A \otimes \delta_{P^{b,w},b_{L^{b}}}^{1/2} \otimes \delta_{b}^{-1/2} \otimes \delta_{P^{b,w},b_{L^{b}}}^{-1/2}[-\langle \xi_{P^{b}}, \nu_{b} \rangle] \\
 & \cong A \otimes \delta_{b}^{-1/2}[-\langle \xi_{P^{b}}, \nu_{b} \rangle].
\end{align*}
For the second isomorphism, we have used that $\mf{q}$ can be identified with the composition $p_{b} \circ \mf{p}$, where $p_{b}\colon [\ast/\wt{G}_{b}] \ra [\ast/\ul{G_{b}(F)}]$ 
is the natural map. For the third isomorphism, we have used a projection formula. We have used Lemma \ref{lem: keycalc} in the fourth isomorphism. We now note that $-\langle \xi_{P^{b}}, \nu_{b} \rangle = - \langle 2\rho_{G}, \nu_{b} \rangle$, by virtue of the fact that $P^{b}$ is a standard parabolic and that $\nu_{b}$ pairs trivially with all roots of $G$ lying in its centralizer $L^{b}$. This shows the desired claim.
\end{proof}
This completes the proof of Theorem \ref{thm: eisensteinvalues}. 
\end{proof}
\begin{proof}[Proof of Lemma \ref{lem: keycalc}]
Let $\mf{g} := \mathrm{Lie}(G)$. Note that $\delta_{b}$ identifies with the modulus character of $P^{b}$ transferred to $(L^{b})_{b_{L^{b}}}$ along the inner twisting, 
since the group is quasi-split. Therefore, we have that $\delta_{P^{b,w},b_{L^{b}}}^{-1/2} \otimes \delta_{b}^{-1/2}$ is the unique rational cocharacter of $G_{b}(F)$ 
such that, for $t \in A_{G_{b}}(F)$, 
we obtain that
\[ \delta_{P^{b,w},b_{L^{b}}}^{-1/2} \otimes \delta_{b}^{-1/2}(t) = \prod_{\substack{\alpha > 0 \\ w^{-1}(\alpha) > 0}} |\det(t|\mf{g}_{\alpha})|^{-1},  \]
as in the equality $\xi_{P^{b}} + \xi_{P^{b,w}} = 2(\xi_{P} - \xi_{P_{b}}^{w})$ 
or alternatively $\prod_{\substack{\alpha < 0 \\ w^{-1}(\alpha) < 0}} |\det(t|\mf{g}_{\alpha})|$.
We also note that 
\[ \langle \xi_{P^{b}}, \nu_{b} \rangle - \langle \xi_{P^{b}}^{w}, \nu_{b} \rangle = 
 \sum_{\substack{\alpha > 0 \\ w^{-1}(\alpha) > 0}} \langle \alpha, \nu_{b} \rangle. \] 
The claim now follows from the previous description of $\wt{G}^{w}$ as an iterated fibration of the Banach--Colmez spaces $\mathcal{H}^{0}(\mcLie(U^{b,-})_{\alpha,b_{L^{b}}})$ for negative roots $\alpha < 0$ such that $w^{-1}(\alpha) < 0$, by arguing as in the proof of Proposition \ref{prop:KBunPtheta}. 
\end{proof}

\noindent
Linus Hamann\\
Department of Mathematics, Harvard University \\
Science Center, Cambridge, MA 02138, USA\\
hamann@math.harvard.edu\\[0.5cm]
Naoki Imai\\
Graduate School of Mathematical Sciences, The University of Tokyo, 3-8-1 Komaba, Meguro-ku, Tokyo, 153-8914, Japan \\
naoki@ms.u-tokyo.ac.jp

\end{document}